\documentclass{arxiv}
\pdfoutput=1
\usepackage[utf8]{inputenc}
\usepackage[T1]{fontenc}
\usepackage[english]{babel}
\usepackage{enumitem}
\usepackage[hidelinks=true]{hyperref}

\usepackage{xfrac}
\usepackage{cancel}
\usepackage{bm}
\usepackage{lmodern}% http://ctan.org/pkg/lm

\usepackage{subcaption}
\usepackage{xcolor}
\usepackage{physics}

\OneAndAHalfSpacedXI
\TheoremsNumberedThrough     % Preferred (Theorem 1, Lemma 1, Theorem 2)
\ECRepeatTheorems

\EquationsNumberedThrough    % Default: (1), (2), ...
\MANUSCRIPTNO{}

\allowdisplaybreaks % To allow page breaks inside equation environment

\begin{document}
%%%%%%%%%%%%%%%%

\TITLE{Heavy Traffic Queue Length Behaviour in a Switch under Markovian Arrivals}

\ARTICLEAUTHORS{%
\AUTHOR{Shancong Mou}
\EMAIL{smou7@gatech.edu}
\AUTHOR{Siva Theja Maguluri}
\EMAIL{siva.theja@gatech.edu}
\AFF{Georgia Institute of Technology, Atlanta, Georgia 30332}} 

\ABSTRACT{%
This paper studies the input queued switch operating under the MaxWeight algorithm when the arrivals are according to a Markovian process. We exactly  characterize the heavy-traffic scaled mean sum queue length in the heavy-traffic limit, and shows that it is within a factor of less than $2$ from a universal lower bound. Moreover, we obtain lower and upper bounds, that are applicable in all traffic regimes, and they become tight in the heavy-traffic regime. 

The paper obtains these results by generalizing the drift method recently developed for the case of i.i.d. arrivals, to the case of Markovian arrivals. The paper illustrates this generalization by first obtaining the heavy-traffic mean queue length and its distribution in a single server queue under Markovian arrivals and then applying it to the case of input queued switch.  
The key idea is to exploit the geometric mixing of finite-state Markov chains, and to work with a time horizon that is picked so that the error due to  mixing depends on the heavy-traffic parameter.
}%


\maketitle

\section{Introduction}
Big data and machine learning revolution are powered by large scale data centers. 
With the growing size of data centers, design and operation of efficient networks that facilitate exchange of data has become important \cite{singh2015jupiter}. 
A goal in the design of a data center network is to create a network that has full bisection bandwidth \cite{perry2014fastpass,alizadeh2013pfabric} or a network that is logically equivalent to an input queued switch. 
An operational challenge in such a data center is to schedule packets in order to maximize throughput and minimize delay. 

In this paper, we study the problem of scheduling packets in an input queued switch. In addition to serving as a model for data center networks, input queued switches are important because they form building blocks of any data network.
An input queued switch can be modeled as a matrix of queues operating in discrete time. Packets arrive into each of the queues according to a stochastic process. All packets need exactly one time slot of service. The key constraint is that at each time, exactly one packet can be served from each row or column. Thus the set of allowed schedules at each time forms a permutation matrix. 

A popular algorithm for scheduling in input queued switch, is the 
 MaxWeight algorithm \cite{tassiulas1990stability}, which was first proposed in the context of wireless networks. 
 In this algorithm, at each time, the permutation matrix with the maximum weight is chosen, using queue lengths as weights. While it is known to maximize throughput in an input queued switch \cite{mckeown1999achieving}, understanding the delay or queue-length performance is much more challenging, and so one uses asymptotic analysis. 
 In this paper, we consider the heavy-traffic regime, where the total arrival rate in each row and column approaches the maximum possible value of one. The mean queue length in heavy-traffic under such the MaxWeight algorithm was characterized in \cite{maguluri2018optimal,maguluri2016heavy},
 and it was shown that MaxWeight has the optimal scaling.  

The key limitation of the work in \cite{maguluri2018optimal,maguluri2016heavy} is that it assumes that the arrivals are according to an i.i.d. process. However, it is known that real data centers experience short bursts of high traffic  \cite{datacenter_traffic}. 
The focus of this paper, therefore, is to consider  arrivals that are modulated by a Markov chain, which can model a  rich class of arrival patterns. For instance, in a continuous-time setting, it is known that
a Markovian arrival process approximates any marked point process to an arbitrary degree of accuracy \cite{asmussen1993marked}. 

\subsection{Main contributions}
The main contribution of this paper is the exact characterization of the mean sum of queue lengths in heavy-traffic in an input queued switch operating under MaxWeight algorithm. Due to Little'{{}s} law, such a result immediately implies a result on the mean delay. The key difference relative to the result in \cite{maguluri2016heavy} under i.i.d. traffic is that the mean queue lengths depends not only on the instantaneous variance in steady-state, but the entire auto covariance function of the arrival process. In addition to the heavy-traffic results, we obtain upper and lower bounds on the mean queue lengths under any traffic, and show that these match in the heavy-traffic regime. We also obtain a universal lower bound under any algorithm, and show that it is within a factor of at most two in the heavy-traffic limit. 

An input queued switch does not satisfy the so called complete resource pooling (CRP) condition, and to the best of our knowledge, this is the first result on heavy-traffic characterization of a nonCRP system. Methods based on fluid and diffusion limits \cite{gamarnik2006validity,harrison1988brownian,har_state_space,harlop_state_space,stolyar2004maxweight,Williams_state_space} 
allow for natural generalization beyond i.i.d. arrivals, but except in special cases, there are no known results on nonCRP systems using these methods. Drift method was introduced in \cite{maguluri2016heavy,maguluri2018optimal} to analyze nonCRP systems such as the switch, but was crucially limited to i.i.d. arrivals. A key methodological contribution of this paper is to extend drift method to Markovian arrivals, and thus obtain the first result on a nonCRP system under the non i.i.d. arrival process. 

The main ingredient in the drift method is to consider the one-step drift, i.e., the expected change in the value, of a test function. Under i.i.d. arrivals, the future arrivals are independent of the current queue length, and this property plays a crucial role in bounding the one-step drift. 
The key challenge in the study of Markovian arrivals is that such independence property does not hold because both the current queue length and future arrivals are correlated to the past arrivals. 
In order to overcome this challenge, we  recursively open  the current queue length and express it in terms of the queue length $m$ slots ago, as well as the arrivals and service over the last $m$ time slots. 
If $m$ is large enough, the old queue length is approximately independent of the current arrivals, due to fast (geometric) mixing of the  Markov chain that modulates the arrivals. However, this recursive opening also introduces certain error terms, which are small only if $m$ is small. We carefully choose $m$ as a function of the heavy-traffic parameter $\epsilon$, to optimally trade off these two competing phenomena, i.e., our choice of $m$  ensures approximate independence between the old queue length and the current arrival while also ensuring that the resultant error terms go to zero in the heavy-traffic regime.

In order to  illustrate our generalization of the drift method, in Section \ref{sec: singleserver}, we present a study of a single server queue operating in discrete time under Markovian arrivals. We  first present the Markovian generalization of Kingman's bound \cite{kingman1961single} on the mean queue length. In the case of single server queue, in addition to the mean queue length, it is possible to obtain the complete distribution of the queue length in heavy-traffic. We show that similar to the i.i.d. case, it is an exponential distribution, albeit with a modified mean that depends on the auto covariance function. We show this result using the transform method
\cite{hurtado2020transform}, a generalization of the drift method based on exponential test functions. 

\subsection{Related literature}

Heavy-traffic analysis of queueing systems is studied in the literature using fluid and diffusion limits 
\cite{gamarnik2006validity,harrison1988brownian,har_state_space,harlop_state_space,stolyar2004maxweight,Williams_state_space}. 
Systems with Markovian arrivals can naturally be studied using such an approach, and the results on single server queue that we present in Section \ref{sec: singleserver} are known. 
However, most of these results are applicable only when the systems satisfy a condition called the Complete Resource Pooling (CRP). Under the CRP condition, the system behaves as a single server queue. This is usually proved formally by a state space collapse result, which shows that in the heavy-traffic limit, the {{} multi-dimensional queue size vector stays close to a one-dimensional subspace}.  
Except in some special cases, there is no literature based on the diffusion limit approach to study systems that do not satisfy the CRP condition. The switch system considered in this paper does not satisfy the CRP condition, and exhibits a multi-dimensional state space collapse, i.e., the queue size vector stays close to a multi-dimensional subspace.

An alternate method for heavy-traffic analysis based on drift arguments was introduced in \cite{eryilmaz2012asymptotically} to study CRP systems. 
This drift method was generalized to study the switch system when the CRP condition is not met, in \cite{maguluri2018optimal,maguluri2016heavy}. The literature on drift method so far is limited to an i.i.d. arrival process. 

The switch under non i.i.d. traffic was studied in \cite{neely2008delay} and \cite{sharifnassab2020fluctuation}. A loose upper bound on queue lengths was obtained in  \cite{neely2008delay} and a state space collapse result was established in \cite{sharifnassab2020fluctuation}. To the best of our knowledge, this is the first work that exactly characterizes the heavy-traffic queue lengths for a system  without CRP, under a non i.i.d. arrival process.  

Input queued switch is one of the simplest systems that does not satisfy CRP, and so has served as a guidepost to study general queueing systems \cite{shah2011optimal}. The drift method that was developed in \cite{maguluri2016heavy} was used to study a variety of stochastic processing networks in a flurry of follow-up works, all of them under i.i.d. traffic. 
Input queued switch with prioritized customers was studied in \cite{IBM_switch,IBM_switch_arxiv}, switch under novel low complexity scheduling algorithms was studied in \cite{jhunjhunwala2020low}, and an optical switch that incurs queueing delay was studied in \cite{javidi_optical_ht}. The so called generalized switch model was studied in \cite{langegenswitch}, and it subsumes a rich class of stochastic networks including  wireless networks, cloud computing, data center networks, production systems, mobile base stations, etc. These methods have also been used to study load balancing in  \cite{hurtado2020logarithmic,zhou2018flexible} and the bandwidth sharing network in  \cite{wang2018heavy}.

\subsection{Notation}

Throughout the paper, we denote a random variable by an upper case letter, for example, $X$; the realization of a random variable $X$ by a lowercase letter $x$. If they are vectors, we use their corresponding boldface letter, for example, $\bm{X}$ and $\bm{x}$. 
We define $X \sim \pi$ as $X$ follows the distribution $\pi$ and $X \stackrel{d}{=} Y$ as $X$ follows the same distribution as $Y$. {{} We use $\mathcal{I}(\cdot)$ to denote the indicator function.} 
We denote $E_{X \sim \pi}[\cdot]$ and $Var_{X \sim \pi}[\cdot]$ as the expectation and variance calculated with respect to the distribution $X \sim \pi$. $\mathbb{N}_+$ denotes the set of positive integers, $\mathbb{N}$  denotes the set of non-negative integers. $\mathbb{R}_+$ denotes the set of positive real number.

The rest of the paper is organized as follows. In Section \ref{sec:preliminaries}, we present a few preliminary results that are essential for the proofs. 
In Section  \ref{sec: singleserver}, we introduce the single server queue model, and present its heavy-traffic results in order to illustrate the drift method for Markovian arrivals. 
In Section \ref{sec:switch}, we introduce the switch model, and present the main result of the paper on heavy-traffic behavior of switch under Markovian arrivals. These results are obtained by using the ideas developed in Section  \ref{sec: singleserver} in conjunction with the broad outline from \cite{maguluri2018optimal}. In particular, we  first present a state space collapse result, which is then used to characterize the heavy-traffic mean sum queue length. Finally, Section \ref{sec:conclusion} concludes the paper.

\section{Preliminaries}\label{sec:preliminaries}
In this section, we present a few preliminary results on Markov chains that will be used throughout the paper. 

\subsection{Moment bounds from Lyapunov-type drift conditions}
The results in this paper are based on studying the drift of functions of  Markov chains. A Lemma from \cite{hajek1982hitting} is usually used to get moment bounds based on conditions on one-step drift. However, in this paper, we will work with  the $m$-step drift instead, 
and so we need the following generalization of the Lemma from  \cite{hajek1982hitting}, which is proved in 
Appendix \ref{proof of Lemma: $m$-step bound for q_perp_norm_2}.
\begin{lemma}\label{Lemma: $m$-step bound for q_perp_norm_2}
        Let $\left\{\bm{Q}^t  ,\bm{X}^t \right\}_{t \geq 0}$ be an irreducible, aperiodic and positive recurrent Markov chain  over a countable state space $  (\mathcal{Q},\mathcal{X}  )$. Suppose $Z: \mathcal{X}\to \mathbb{R}_+$ is a nonnegative Lyapunov function. We define the $m$-step drift of $Z$ at $  (\bm{q},\bm{x}  )$ as
        \begin{equation}\label{eq: multistep drift}
            \Delta^m Z  (\bm{q}, \bm{x}  ) \triangleq   \left[Z  \left(\bm{Q}^{t+m},\bm{X}^{t+m}  \right)-Z  \left(\bm{Q}^t,\bm{X}^t \right)  \right]\mathcal{I}  \left(\left(\bm{Q}^{t},\bm{X}^{t}  \right) =   (\bm{q},\bm{x}  )  \right),
        \end{equation}
        where $\mathcal{I}  (\cdot  )$ is the indicator function. Thus, $\Delta^m Z  (\bm{q},\bm{x}  )$ is a random variable that measures the amount of change in the value of Z in $m$ steps, starting from state $  \left(\bm{q},\bm{x}\right)$. This drift is assumed to satisfy the following conditions, {for any $m$}:
        \begin{enumerate}  
        \item There exists a $\eta(m) > 0$, and a $\kappa(m) >0$ such that for any $t = 1, 2, . . .$ and for all $  (\bm{q},\bm{x}) \in   (\mathcal{Q},\mathcal{X}  )$ with $Z  (\bm{q},\bm{x}  ) \geq \kappa(m)$,
        \begin{equation*}
            E  \left[\Delta^m Z  (\bm{q}, \bm{x}  )\mid   \left(\bm{Q}^t,\bm{X}^t  \right) =   (\bm{q},\bm{x}  )  \right] \leq -\eta(m).
        \end{equation*}
        \item There exists a $D(m) < \infty$ such that for all $  \left(\bm{q},\bm{x}  \right) \in   (\mathcal{Q},\mathcal{X}  )$,
        \begin{equation*}
            P  (\abs{ \Delta^m Z  (\bm{q}, \bm{x}  )} \leq D(m)  ) =1.
        \end{equation*}
        \item {For any $t_0 \in [0,m]$, there exists a $\hat{D}(t_0) < \infty$ such that for all $  \left(\bm{q},\bm{x}  \right) \in   (\mathcal{Q},\mathcal{X}  )$,}
        \begin{equation*}
          { P  (\abs{ Z  (\bm{q}, \bm{x}  )} \leq \hat{D}(t_0))=1.}
        \end{equation*}
        
        \end{enumerate}
        Then, there exists a $\theta^*(m)>0$ and a $C^*(m) < \infty$ such that
        \begin{equation*}
            \limsup_{t \to \infty} E  \left[e^{\theta^*(m)Z  \left(\bm{Q}^t,\bm{X}^t  \right)}  \right] \leq C^{\star}(m).
        \end{equation*}
        If the Markov chain is positive recurrent, then $Z  \left(\bm{Q}^t,\bm{X}^t\right)$ converges in distribution to a random variable $\overline{Z}$ for which
        \begin{equation*}
            E  \left[e^{\theta^*\overline{Z}  }  \right] \leq C^{\star}(m),
        \end{equation*}
        which implies that all moments of $\overline{Z}$ exist and are finite.
\end{lemma}
\textbf{Remark}: if $\abs{ \Delta Z  (\bm{q}, \bm{x}  )} \leq D'$ is satisfied, Conditions 2 and 3 are satisfied with $D(m) = mD'$ and $\hat{D} = D'$.
\subsection{Geometric mixing of finite-state Markov chains}
The following two lemmas on geometric mixing of finite-state Markov chains will be exploited to obtain the results in the paper. 
\begin{lemma}\label{lem: geo}
    Let $\left\{X^t\right\}_{\ t\geq 0}$ be an irreducible, positive recurrent and aperiodic Markov chain on a finite-state space $\Omega$. Let $\pi$ denote its stationary distribution. Let $f(\cdot)$ be a real-valued function, i.e., $f:\Omega \to \mathbb{R}_+$.  Let $\lambda$ be the stationary mean of $f\left(X^t\right)$, i.e.,
    \begin{equation*}
    \lambda = E_{X\sim\pi}\left[f(X)\right].
    \end{equation*}
    Then,  there exist constants  $\alpha \in   (0,1  )$ and $C > 0$ such that, for any $m \in \mathbb{N}_+$, for any initial distribution $X^0 \sim \pi^0$, we have
        \begin{equation*}
          \vert E_{ X^0 \sim \pi^0 }    [  (f(X^m)-\lambda  )   ]  \vert \leq  2LC  \alpha^m,
        \end{equation*}
where $  L = \max_{x \in \Omega} f(x) $. 
\end{lemma}
The lemma is proved in Appendix \ref{Proof: lem: geo}.

\begin{lemma}\label{lem: finite}
Let $\left\{X^t\right\}_{\ t\geq 0}$ be an irreducible, positive recurrent and aperiodic  Markov chain with finite-state space $\Omega$ and stationary distribution $\pi$. {{}Let $f(\cdot)$ be a real-valued function upper bounded by $A_{max}$, i.e., $f:\Omega \to \mathbb{R}_+, f(x)\leq A_{max},\ \forall x\in \Omega$.}  Let $\lambda$ be the stationary mean of $f(X^t)$, i.e.,
    \begin{equation*}
    \lambda = E_{\pi}\left[f(X)\right].
    \end{equation*}
    Then,
\begin{align*}
        &\lim_{m \to \infty} \frac{Var_{ X^0 \sim \pi^0 }  \left(\sum_{t=1}^m f(X^t)  \right)}{m} \\
        &= \gamma(0)+2\lim_{m \to
        \infty}\sum_{t=1}^{m}\frac{m-t}{m}\gamma(t) = \gamma(0)+2\lim_{m \to \infty}\sum_{t=1}^{m}\gamma(t),
\end{align*}
    where 
    \[
        \gamma(t) = E_{X^0 \sim \pi}\left[  \left(f\left(X^t\right)-\lambda  \right)  \left(f\left(X^0\right)-\lambda  \right)\right] 
    \]
    is the auto covariance function.
\end{lemma}
The lemma is proved in Appendix \ref{Proof of Lemma: finite}. 

{\textbf{Remark}: There are several known equivalent formulations of the asymptotic variance. For instance, it can be shown that (see Theorem 21.2.5 in \cite{douc2018markov}) under appropriate assumptions,
\begin{align}
\lim_{m \to \infty} \frac{Var_{ X^0 \sim \pi^0 }  \left(\sum_{t=1}^m f(X^t)  \right)}{m} = 
2 E_{X \sim \pi}\left[    \left(f\left(X\right)-\lambda  \right)\hat{f}(X)\right]
-E_{X \sim \pi}\left[    \left(f\left(X\right)-\lambda  \right)^2\right] \label{eq:Poisson},
\end{align} where $\hat{f}(\cdot)$ is a solution of the Poisson equation of the Markov Chain, $\hat{f}(x) = f(x) + E\left[\hat{f}(X^1)|X^0=x\right] - \lambda$.
In this paper, we will state all the results in terms of the asymptotic variance in Lemma \ref{lem: finite}, and our results can be easily reformulated in terms of equivalent expressions such as \eqref{eq:Poisson}.

{\textbf{Remark}: The finite state space assumption is used primarily for establishing the positive recurrent result for the single server queue and the switch system in heavy traffic regime. All other results still hold for general state space. One way of stating the heavy-traffic results in the paper is for general state space  \textit{assuming} positive recurrence of the system (and to prove positive recurrence in the case of finite state space). However, we feel that such a presentation would be confusing, and so we will just present the finite state space case in the following sections.}

\section{Single Server Queue} \label{sec: singleserver}
In this section, we consider a single server queue operating in discrete-time. Under i.i.d. arrivals, queue length in such a system is equivalent to the waiting time in a $G/G/1$ queue. Drift arguments were used to study the heavy-traffic mean queue length in \cite{eryilmaz2012asymptotically,kingman1961single,srikant2013communication}, and transform method was used to study the heavy-traffic stationary distribution in \cite{hurtado2020transform}. In this section, we consider the discrete-time single server queue under Markovian arrivals, and extend both the drift method and the transform method. The key ingredients are to consider the dynamic $m$-step drift and exploit the geometric mixing of irreducible, positive recurrent, aperiodic finite-state-space Markov chains. 

\subsection{Mathematical Model}\label{subsec Mathematical Model.}
Consider a single server queue operating in discrete-time. 
Let $Q^t$ be the number of customers in the system at the beginning of time slot $t$. Arrivals occur according to an underlying  Markov chain $\left\{X^t\right\}_{\ t\geq 0}$ where the number of arrivals in the time slot $t$ is given by $A^t = f  (X^t  )$ for some non-negative integer-valued function $f(\cdot)$. Potential service $S^t$ is assumed to be i.i.d. with mean $\mu$ and variance $\sigma_s^2$. 

Assume that $\left\{X^t\right\}_{\ t\geq 0}$ is irreducible, positive recurrent and aperiodic on a finite-
state space $\Omega$. Thus $\left\{X^t\right\}_{\ t\geq 0}$ converges to its stationary distribution $\pi$ with geometric rate. Further assume that $A^t$ and $S^t$ are bounded above by $A_{max}$ and $S_{max}$ respectively.

In each time slot, we assume that the service occurs after arrivals, and the system evolves as follows: for each $t = 1,\ 2,...$
\begin{align}
        Q^{t+1}=&\max\left\{Q^t + A^t - S^t,0\right\}=Q^t + A^t - S^t+ U^t, \label{eq:1}
\end{align}
where $U^t$ denotes the unused service and is defined by $U^t = Q^{t+1} -\left(Q^t + A^t - S^t \right) $. From the definition of the unused service, we have 
\begin{align}\label{eq: 2}
    Q^{t+1}U^t=0. 
\end{align}

In order to study the heavy-traffic behavior of the single server queue, we will consider a sequence of arrival processes such that the  arrival rate approaches the service rate $\mu$. To this end, let $\left\{X^t\right\}^{  (\epsilon  )}_{t\geq 0}$ be a set of irreducible, positive recurrent and aperiodic underlying Markov chains indexed by the heavy-traffic parameter $\epsilon \in   (0,\mu  )$. Assume that the arrival process is such that the two dimensional Markov chain     $
      \left\{ \left(\left(Q^t\right)^{  (\epsilon  )},\ \left(X^t\right)^{  (\epsilon  )} \right)\right\}_{t\geq 0 }
    $ is irreducible and aperiodic.  
For any fixed $\epsilon$, let $\overline{X}^{  (\epsilon  )}$ be the steady state variable to which $\left(X^t\right)^{  (\epsilon  )}$ converges in distribution with $E\left[\overline{A}^{(\epsilon)} \right] = E\left[f\left(\overline{X}^{  (\epsilon  )}\right)\right] = \lambda^{(\epsilon)} = \mu -\epsilon$. 

Let $\gamma^{(\epsilon)}(t)$ to be the auto covariance function of the arrival process starting from steady state $X^0 \stackrel{d}{=}\overline{X}^{  (\epsilon  )}$, $\overline{A}^{(\epsilon)}= f\left(\overline{X}^{  (\epsilon  )}\right)$ and $\left(A^t\right)^{(\epsilon)} = f\left(\left(X^t\right)^{(\epsilon)}\right)$ indexed by $\epsilon$, i.e., 
    \[
        \gamma^{(\epsilon)}(t) = E \left[  \left(\left(A^t\right)^{(\epsilon)}-\lambda  \right)  \left(\left(A^0\right)^{(\epsilon)}-\lambda  \right) \right].
    \]
    Let
    \[
        \left(\sigma_a^{(\epsilon)}\right)^2 = \lim_{m \to \infty} \frac{Var \left(\sum_{t=1}^m \left(A^t\right)^{(\epsilon)}  \right)}{m}=\gamma^{(\epsilon)}(0)+2\lim_{m \to \infty}\sum_{t=1}^{m}\gamma^{(\epsilon)}(t),
    \]
where the equality follows from Lemma \ref{lem: finite}. {{} We will call $\left(\sigma_a^{(\epsilon)}\right)^2$ the effective variance. }
Assume that $\forall t \in \mathbb{N}$, 
\begin{align}\label{eq: gamma to gamma}
    \lim_{\epsilon \to 0}\gamma^{(\epsilon)}(t) = \gamma(t).
\end{align}
    Define 
    \begin{align*}
        \sigma_a^2 = \gamma(0)+2\lim_{m \to \infty}\sum_{t=1}^{m}\gamma(t).
    \end{align*}
Then, we have the following claim, which establishes and interchange of limit, and so we call $\sigma_a^2 $ the limiting effective variance. 
\begin{claim}\label{claim: interchange of limit}
$\lim_{\epsilon \to 0} \left(\sigma_a^{(\epsilon)}\right)^2 = \sigma_a^2.$
\end{claim}
The claim is proved in Appendix \ref{proof_Claim_interchange_of_limit}.

\subsection{Heavy-traffic Limit of the Mean Queue Length}
We now present the generalization of the Kingman's heavy-traffic bound \cite{kingman1961single} in a single server queue under Markovian arrivals. In other words, we characterize the steady state mean queue length in a single server queue in the heavy-traffic regime. 
\begin{theorem}
\label{thm: pos rec single server queue}
 Let $\left(A^t\right)^{  (\epsilon  )}$ be a set of arrival processes determined by the corresponding Markov chains, $\left\{\left(X^t\right)^{(\epsilon)}\right\}_{\ t\geq 0}$ as described before,
with steady-state arrival rates $\lambda^{  (\epsilon  )} = \mu -\epsilon$. Let $\alpha^{  (\epsilon  )}$ and $C^{  (\epsilon  )} $ be the corresponding geometric mixing parameters as mentioned in Lemma \ref{lem: geo}. 
 Assume $\sup_\epsilon \alpha^{  (\epsilon  )} \leq \alpha <1$ and $\sup_\epsilon C^{  (\epsilon  )} \leq C < \infty$. Then we have 
\begin{enumerate}
    \item For each $\epsilon \in   (0,\mu  )$, the two dimensional Markov chain 
    \begin{displaymath}
      \left\{ \left(\left(Q^t\right)^{  (\epsilon  )},\ \left(X^t\right)^{  (\epsilon  )} \right)\right\}_{t\geq 0 }
    \end{displaymath}
     is positive recurrent. 
    \item Let $\overline{Q}^{  (\epsilon  )}$ be a steady-state random variable to which the queue length processes $\left\{Q^t\right\}^{  (\epsilon  )}_{t \geq 1}$ converges in distribution. We have
    \begin{equation*}
            \lim_{\epsilon \to 0} E    \left[\epsilon \overline{Q}^{  (\epsilon  )}   \right]=\frac{\sigma_a^2+\sigma_s^2}{2}.
    \end{equation*}
    \end{enumerate}
\end{theorem}
\begin{proof}{of Theorem \ref{thm: pos rec single server queue}.}
 To prove the first part of the theorem, we use $m$-step Foster-Lyapunov theorem  (Proposition 2.2.4 in \cite{fayolle1995topics}).  We consider the quadratic Lyapunov function, $\left(Q^{(\epsilon)}\right)^2$, and show that its $m$-step drift is negative except in a finite set. 
 Consider a fixed $\epsilon \in (0,\mu)$.  For ease of exposition, we suppress the superscript $  (\cdot)^{(\epsilon )}$.
 
\begin{claim}\label{claim 3.2.1}
 For any $\epsilon \in (0,\mu)$ and $m \in \mathbb{N}_+$
 \begin{align*}
    \begin{aligned}
       & E  \left[\left(Q^{t+m}\right)^2-\left(Q^t\right)^2\mid   \left(Q^t,X^t  \right)=  (q,x  )  \right]\\
       & \leq E  \left [2Q^t\left(\sum_{i=1}^{m}  (A^{t+m-i}-\lambda  )\right)-2m\epsilon Q^t+m K_0  (m  )\mid   \left(Q^t,X^t  \right)=  (q,x  )   \right],\\
    \end{aligned}
\end{align*}
 where
 \begin{align*}
    K_0  (m ) = 2m  (A_{max}+S_{max}  )  (A_{max}+S_{max}  )+  (A_{max}+S_{max}  )^2.
\end{align*}
 \end{claim}
The proof of Claim \ref{claim 3.2.1} is in Appendix \ref{proof_Claim_3.2.1}.

We now consider the first term on the RHS. The main challenge in bounding this term is the correlation between queue length and arrival at the same time slot, due to the Markovian nature of the arrival process. We use the geometric mixing of the underlying Markov chain as stated in Lemma \ref{lem: geo}, to get 
\begin{align}\label{eq: geom mixing ssq}
        &E  \left[2Q^t  \left(A^{t+m}-\lambda \right )\mid   \left(Q^t,X^t  \right)=  (q,x  )  \right]\nonumber\\
        &\leq \abs{E  \left[2Q^t  \left(A^{t+m}-\lambda \right )\mid  \left(Q^t,X^t  \right)=  (q,x  )   \right]} \leq   4A_{max}C\alpha^m q .   
\end{align}
Thus,
\begin{align*}
        & E  \left [2Q^t\left(\sum_{i=1}^{m}  (A^{t+m-i}-\lambda  )\right)-2m\epsilon Q^t\mid   \left(Q^t,X^t  \right)=  (q,x  )   \right]\\
        &\leq 2q    \left[  \left(2A_{max}C\frac{1-\alpha^m}{1-\alpha}  \right)-m\epsilon   \right].
\end{align*}
Define $m  (\epsilon  )= \min\left\{m \in \mathbb{N}_+: 2A_{max}C\frac{1-\alpha^m}{1-\alpha}<\frac{1}{2}m\epsilon \right\}$. Such an $m  (\epsilon  )$ exists because $2A_{max}C\frac{1-\alpha^m}{1-\alpha}$ is finite and $\frac{1}{2}m\epsilon \to \infty$ as $m \to \infty$. Let $K_1(\epsilon) = \frac{\epsilon m(\epsilon)}{2}$. Then, we have,
 \begin{align*}
       & E  \left[\left(Q^{t+m(\epsilon)}\right)^2-\left(Q^t\right)^2\mid   \left(Q^t,X^t  \right)=  (q,x  )  \right] \leq -2K_1(\epsilon)q +m(\epsilon) K_0  (m(\epsilon)  ),
\end{align*}
Let $\mathcal{B} = \left\{q: q \leq \frac{m(\epsilon) K_0  (m(\epsilon)  )}{K_1({\epsilon})}\right\}$ denote a finite set. Then, we have
 \begin{align*}
       & E  \left[\left(Q^{t+m(\epsilon)}\right)^2-\left(Q^t\right)^2\mid   \left(Q^t,X^t  \right)=  (q,x  )  \right] \\
       &\leq -m(\epsilon) K_0  (m(\epsilon)  )\mathcal{I}  (q \in \mathcal{B}^c ) + m(\epsilon) K_0  (m(\epsilon)  )\mathcal{I}(q \in \mathcal{B}). 
\end{align*}
       
Positive recurrence of the two dimensional Markov chain $\{  (Q^t,\ X^t  ),\ t\geq 0\}$ then follows from the $m$-step Foster-Lyapunov theorem. 

Therefore, using irreducibility and aperiodicty of the two dimensional Markov chain $\{  (Q^t,\ X^t  ),\ t\geq 0\}$, we know that a unique stationary distribution exists. To prove the second part of the theorem, we will set the $m$-step drift of the quadratic test function to zero under the stationary distribution. In order to do this, we should first ensure that the stationary expectation of the quadratic function is finite, which is given by the following claim: 
\begin{claim}\label{claim: q^2 finite}
    For any $\epsilon >0$ in steady-state,
\begin{align*}
    E  \left[  \left(\overline{Q}^{  (\epsilon  )}  \right)^2  \right] < \infty.
\end{align*} 
\end{claim}
The proof of Claim \ref{claim: q^2 finite} is in \ref{proof q^2 finite}.

In the rest of the proof, we consider the system in its steady-state, and so for every time $t$, 
 $\left(Q^t\right)^{(\epsilon)} \stackrel{d}= \overline{Q}^{(\epsilon)}$ {{}(equivalently, it can be thought of as initializing the system in its steady-state)}. For ease of exposition, we again drop the superscript $(\cdot)^{(\epsilon)}$
 and just use $Q^t$. 
 Then, $A^t$  denotes the arrival in steady state, and the queue length at time $t+m$, is  $Q^{t+m}=Q^t+\sum_{l=0}^{m-1} A^{t+l}-\sum_{l=0}^{m-1} S^{t+l}+\sum_{l=0}^{m-1} U^{t+l}$, 
 which has the same distribution as $Q^t$ 
 for all $m \in \mathbb{N}_+$.  We can set the one-step drift equal to zero:
    \begin{align}\label{eq: one-step}
            0=E  \left[\left(Q^{t+1} \right)^2-\left(Q^{t} \right)^2 \right]=E  \left[\left(Q^t+A^t-S^t+U^t \right)^2-\left(Q^{t} \right)^2 \right],
    \end{align}
    where the last equation follows from (\ref{eq:1}). Expand (\ref{eq: one-step}) and apply (\ref{eq: 2}), we have
    \begin{align}\label{eq: qe}
       2\epsilon E  \left[Q^t   \right]=E  \left[2Q^t  (A^t-\lambda  )  \right]-E  \left[\left(U^t\right)^2  \right] + E  \left[  (A^t-\lambda  )^2  \right]+\sigma_s^2+  \epsilon^2.
    \end{align}

    Eq \eqref{eq: qe} can be obtained using standard arguments. The heavy-traffic limit of scaled queue length $\epsilon E  \left[Q^t \right]$ is typically obtained by  bounding the RHS and then letting $\epsilon \to 0$. The main challenge here is bounding the $E  \left[2Q^t  (A^t-\lambda  )  \right]$ term since the queue length and the arrival are correlated due to the Markovian arrival process.
    We bound this term by recursively expanding $Q^t$ using (\ref{eq:1}), and using Markov chain mixing result from Lemma \ref{lem: geo}. We then obtain the following claim, the proof of which can be found in Appendix \ref{proof Claim 3.2.2}.
    
    \begin{claim}\label{claim: 3.2.2}
    For any $\epsilon \in (0,\mu)$ and $m \in \mathbb{N}_+$, we have
    \begin{align*}
             2\left(1-2A_{max}C\frac{\alpha^m}{\epsilon}\right) E\left[\epsilon Q^t\right] &\leq \gamma(0) + 2\sum_{i=1}^{m}\gamma(i)  + \sigma_s^2 + 2m  (A_{max}+\lambda  ) \epsilon +   \epsilon^2\\
             2\left(1+2A_{max}C\frac{\alpha^m}{\epsilon}\right) E\left[\epsilon Q^t\right]  \geq& \gamma(0) + 2\sum_{i=1}^{m}\gamma(i)   + \sigma_s^2- 2m  (A_{max}+\lambda  )\epsilon\\& -  S_{max}\epsilon +   \epsilon^2,
    \end{align*}  
    \end{claim}

    Note that the claim is true for all $\epsilon$ and $m$. Put ${  (\cdot   )}^{  (\epsilon  )}$ back.  For a given $\epsilon$, we now pick $m = \left\lfloor \frac{1}{\sqrt{\epsilon}} \right\rfloor$, and take the limit as $\epsilon \to 0$ to get
    \begin{align}\label{eq: claim 3.2.3}
        \lim_{\epsilon \to 0} E  [\epsilon\overline{  Q}^{(\epsilon)}]=\frac{\lim_{\epsilon \to 0} \gamma^{(\epsilon)}(0) + \lim_{m(\epsilon) \to \infty} 2\sum_{t=1}^{m(\epsilon)}\gamma^{(\epsilon)}(t)+\sigma_s^2}{2}.
    \end{align}
\textbf{Discussion.} The key challenge in the proof is in handling the $E  \left[Q^t  (A^t-\lambda  )  \right]$ term. When the arrivals are i.i.d., the expectation can simply be written as a product of expectations. Under Markovian arrivals, this cannot be done due to correlation between the queue length and the arrivals at the same time slot. A multistep drift approach is usually used \cite{rajagopalan2009network} to overcome such a difficulty while establishing positive recurrence. We adopt the same approach in the first part of the proof to establish positive recurrence. However, it is unclear if such a multistep drift argument is refined enough to provide an exact expression for the mean heavy-traffic queue length. So, in the second part of the proof, we simply use the one-step drift, and use a different approach to bound the term $E  \left[Q^t  (A^t-\lambda  )  \right]$. Essentially, we recursively use \eqref{eq:1} $m$ times, to open up $Q^t$ in terms of $Q^{t-m}$ and eventually get a term of the form,  $E  \left[2Q^{t-m}  (A^t-\lambda  )  \right]$. Now, for large enough $m$, $Q^{t-m}$ and $A^t$ are approximately independent, due to fast mixing of the Markov chain underlying the arrival process, and so we can approximate the expectation by the product of expectations. This recursive opening gives several other terms of the form $E  \left[ (A^{t-i}-\lambda  )   (A^t-\lambda  )  \right]$, which give us the auto covariance of the arrival process in Claim \ref{claim: 3.2.2}, and eventually in Theorem \ref{thm: pos rec single server queue}.

Note that for the argument to work, $m$ should be picked carefully. In particular, it should be large enough to ensure independence between the queue length and the arrival, and on the other hand, it should be small enough to ensure that the error terms we introduce due to recursive opening are negligible in the heavy-traffic regime. It turns out that any $m$ that is between {$\Omega\left( \ln \frac{1}{\epsilon} \right)$} and $O(\frac{1}{\epsilon})$ works, and we choose  $m = \left\lfloor \frac{1}{\sqrt{\epsilon}} \right\rfloor$  arbitrarily within the range. 

%\textsuperscript{\dag} \footnotetext{{\textsuperscript{\dag}Let $f(x)$ and $g(x)$ be defined on some unbounded subset of the positive real numbers, and $g(x)$ be strictly positive for all large enough values of $x$. Then $f(x) = \Omega(g(x))$, if there exists a positive real number $M$ and a real number $x_0$ such that
% $f(x)\geq Mg(x), {\text{ for all }}x\geq x_{0}$.}}

The following claim, which is proved in Appendix \ref{proof: claim 3.2.3} is useful to evaluate
    \begin{align*}
        \lim_{m(\epsilon) \to \infty}\sum_{t=1}^{m(\epsilon)}\gamma^{(\epsilon)}(t).
    \end{align*}
    
    \begin{claim}\label{claim: 3.2.3}
    For any $\epsilon \in (0,\mu)$ and $m(\epsilon) = \left\lfloor \frac{1}{\sqrt{\epsilon}} \right\rfloor$, we have
        \begin{equation}
        \begin{aligned}
        \lim_{m(\epsilon) \to \infty}\sum_{t=1}^{m(\epsilon)}\gamma^{(\epsilon)}(t) 
        = \lim_{M \to \infty} \sum_{t=1}^{M} \gamma(t).    
       \end{aligned}
        \end{equation}
    \end{claim}
    Finally, combining (\ref{eq: claim 3.2.3}) and Claim \ref{claim: 3.2.3}, we have
    \begin{equation*}
        \begin{aligned}
            \lim_{\epsilon \to 0} E  \left[\epsilon \overline{Q}^{(\epsilon)}\right]
            &=\frac{ 2\lim_{M \to \infty} \sum_{i=1}^{M} \gamma(t)+\lim_{\epsilon \to 0} \gamma^{(\epsilon)}(0) +\sigma_s^2}{2} =  \frac{\sigma_a^2+\sigma_s^2}{2}. 
        \end{aligned}
    \end{equation*}
\end{proof}
Note that the key idea in the proof is to consider $m$-step drift where $m$ is picked to be a function of the heavy-traffic parameter $\epsilon$. In addition, mixing time bounds on the underlying Markov chain are exploited. 

Moreover, we also derive the heavy-traffic limiting distribution of the scaled queue length, which is in Appendix \ref{Appd, queue lenth distribution}.  

\section{Input Queued Switch} \label{sec:switch}
In this section, we will study an input queued switch operating under Markovian arrivals, and present an exact characterization of mean sum of the queue lengths in heavy-traffic. 
In this section, we first introduce the mathematical model of a switch under Markovian arrival and the MaxWeight algorithm. Then, we present throughput optimality, state space collapse and asymptotically tight upper and lower bounds under the MaxWeight scheduling algorithm, which are proved using the $m$-step Lyapunov-type drift argument developed in the previous section. 
For completeness, the universal lower bound under Markovian arrival under any feasible scheduling algorithm is also presented here.

{\bf Note on Notation.} We adopt the notation and definitions in \cite{maguluri2016heavy}. We restrict our discussion in Euclidean space $\mathbb{R}^{N^2}$. For ease of exposition and understanding, we express our elements $\bm{x}$ in $\mathbb{R}^{N^2}$ in two equivalent ways. First, as a $N^2$-dimensional vector which is the standard representation in Euclidean space. Second, as a $N\times N$ matrix with the $(i, j)$ element denoted by $\bm{x}_{ij}$. 
For any two vectors $\bm{x}$ and $\bm{y}$, the inner product is defined by,
\begin{equation*}
        \left\langle \bm{x}, \bm{y} \right \rangle \triangleq \sum_{i=1}^N\sum_{j=1}^N x_{ij}y_{ij}.
\end{equation*}
We say that $\bm{X}=\bm{Y}$ if $X_{ij}=Y_{ij}, \forall i,j \in\{1,2,..., N\}$.
The vector consisting of all ones is denoted by $\bm{1}$. Define 
$\bm{e}^{(i)}$ as the matrix with ones in the $i$-th row and zeros else where, and similarly, 
$\tilde{\bm{e}}^{(j)}$ is the matrix with ones in the $j$-th column and zeros everywhere else, i.e.,  
\begin{equation}\label{e_ij}
    \begin{aligned}
        &\bm{e}^{(i)}_{i,j}=1\ \forall j, &\bm{e}^{(i)}_{i',j}=0\ \forall i'\neq i,\ \forall j.\\
        &\tilde{\bm{e}}^{(j)}_{i,j}=1\ \forall i, &\tilde{\bm{e}}^{(j)}_{i,j'}=0\ \forall j'\neq j,\ \forall i.\\      
    \end{aligned}
\end{equation}

\subsection{Mathematical Model of a Switch}
An $ N \times N $ crossbar switch, also known as an input queued switch, has 
$N$ input ports and $N$ output ports with a separate queue for each input-output pair. Such a system can be modeled as an 
$N\times N$ matrix of queues, where each $(i,j)^{\text{th}}$ queue corresponds to packets entering input port $i$ and intended for output port $j$. 
The four key elements in the mathematical model of a switch under Markovian arrivals are as follows:

     $\bm{A}^t \in \mathbb{R}^{N^2}$: the vector of all arrivals at time slot $t$, of which the $(i,j)^{\text{th}}$ element corresponds to the number of packets arriving at the $i$-th input port and to be delivered to the $j$-th output port. 
     
     $\bm{S}^t \in \mathbb{R}^{N^2}$: the vector of all services at time slot $t$. In each time slot, in each column and row, at most one queue can be served and the service is at most $1$. 
     
     $\bm{U}^t \in \mathbb{R}^{N^2}$: the vector of unused services at time slot $t$. 
     
     $\bm{Q}^t \in \mathbb{R}^{N^2}$: the vector of all queue lengths at time slot $t$.

The arrival process $A_{ij}^t$ is determined by an underlying  Markov chain $\left\{\bm{X}^t\right\}_{t\geq 0 }$ by $A_{ij}^t= {{}f_{ij}}  \left(X_{ij}^t  \right)$. 
The multidimensional Markov chain $\left\{\bm{X}^t\right\}_{t\geq 0 }$ is assumed to be irreducible, positive recurrent and aperiodic on a finite-
state space $\Omega$. Thus $\left\{\bm{X}^t\right\}_{t\geq 0}$ converges to its stationary distribution $\bm{\pi}$ with geometric rate. Let $\bm{f}=({{}f_{ij}})$ and $\bm{\lambda} = E  \left[\overline{\bm{A}}  \right] = E  \left[\bm{f}  \left(\overline{\bm{X}}  \right)  \right]$, where $\overline{\bm{A}}$ and $\overline{\bm{X}}$ are the steady state variables to which $\bm{A}^t$ and $\bm{X}^t$ converges in distribution. 

The set of feasible schedules $\mathcal{S}$ can be written as,
\begin{equation*}
    \mathcal{S} = \left\{\bm{s} \in\{0,1\}^{N^2}: \sum_{i=1}^N s_{ij}\leq 1, \sum_{j=1}^N s_{ij}\leq 1\ \forall i,j\in\{1,2,..., N\}    \right\}.
\end{equation*}

The system operates as follows. 

    At the beginning of every time slot, the service $\bm{S}^t$ is determined based on the queue length $\bm{Q}^t$, according to the scheduling algorithm. 
    Then, the arrivals $\bm{A}^t$ occur according to the underlying Markov chain.  
    Finally the packets are served and this may result in an unused service if there are no packets in a scheduled queue.
The queue length evolves as follows:
\begin{equation*}
    \begin{aligned}
        Q_{ij}^{t+1} =& \left[Q_{ij}^t+A_{ij}^t-S_{ij}^t\right]^+=Q_{ij}^t+A_{ij}^t-S_{ij}^t +U_{ij}^t,
    \end{aligned}
\end{equation*}
or
\begin{equation*}
    \bm{Q}^{t+1} = \bm{Q}^{t}+ \bm{A}^{t}-\bm{S}^{t}+\bm{U}^t.
\end{equation*}
where $[x]^+ = \max(0,x)$. Assume that the Markov chain $\left\{\left(\bm{Q}^t,\bm{X}^t\right)\right\}_{t\geq 0 }$ is irreducible.

For the unused service $\bm{U}^t$, we have the natural constraints on $U_{ij}^t$:
\begin{gather*} 
        U_{ij}^t \leq S_{ij}^t,\\
        \sum_{i}U_{ij}^t \in \{0,1\}\ \forall j \in \{1,2,...,N\},\\ 
        \sum_{j}U_{ij}^t \in \{0,1\}\ \forall i \in \{1,2,...,N\}.
\end{gather*}

Moreover, we have $\forall i,j \in\{1,2,...,N\}$
\begin{equation*}
        U_{ij}^t A_{ij}^t =0,\ \ \ U_{ij}^t {{}Q_{ij}^t} =0,\ \ \ U_{ij}^t Q_{ij}^{t+1} =0.
\end{equation*}
{
The first two equations can be derived as follows: If $U^t_{ij}=0$, then $U^t_{ij}A^t_{ij}=0$ and $U^t_{ij}Q^t_{ij}=0$. Conversely, if $U^t_{ij}\neq0$, we know that the service is at most 1, so $U^t_{ij}=1$. Considering the system operation, the only possible case is when the queue length at the current time slot is zero ($Q_{ij}^t = 0$) and there are no arrivals during the current time slot ($A_{ij}^t = 0$). Consequently, we have $U^t_{ij}A^t_{ij}=0$ and $U^t_{ij}Q^t_{ij}=0$.}

\subsection{MaxWeight algorithm}\label{Mathematical Model of MaxWeight algorithm} The MaxWeight algorithm is a well-studied scheduling algorithm for switches \cite{srikant2013communication}.
In each time slot, MaxWeight algorithm picks a
service vector $S^t$ such that the weighted summation of service is maximized, where the weight vector is the current queue length vector, i.e.,  
\begin{equation*}
    \bm{S}^t = \arg \max_{\bm{s}\in \mathcal{S}} \sum_{ij}Q_{ij}(t)s_{ij} = \arg \max_{\bm{s}\in \mathcal{S}}\langle \bm{Q}^t, \bm{s}\rangle.
\end{equation*} 
In MaxWeight algorithm, ties are broken uniformly random.
The set of maximal feasible schedules or perfect matchings, $\mathcal{S}^*$ is defined as follows:
\begin{equation*}
    \mathcal{S}^* = \left\{\bm{s} \in\{0,1\}^{N^2}: \sum_{i=1}^N s_{ij}= 1, \sum_{j=1}^N s_{ij} = 1\ \forall i,j\in\{1,2,..., N\}    \right\}.
\end{equation*}
Without loss of generality,  
we assume that the MaxWeight algorithm always picks a perfect matching, 
i.e.,
\begin{equation*}
    \bm{S}^t \in \mathcal{S}^*,\ \forall t>0.
\end{equation*}
Note that this is without loss of generality because under MaxWeight scheduling algorithm, we can always choose the maximal schedule. Only when the queue length at some queue $(i,j)$ are zero, then $s_{ij} = 0$ and $s^*_{ij}=1$. In this case, we can pretend that $s_{ij} = 1$ and $u_{ij} = 1$.

Once a Markovian scheduling algorithm is fixed, the switch can be modeled by the Markov Chain, $\left\{\left(\bm{Q}^t,\bm{X}^t\right)\right\}_{t\geq 0 }$. We assume that the arrival process is such that under the MaxWeight algorithm, this Markov chain is irreducible and aperiodic.
A switch is said to be stable under a scheduling algorithm if the  Markov chain $\left\{\left(\bm{Q}^t,\bm{X}^t\right)\right\}_{t\geq 0 }$ is positive recurrent. The capacity region is defined as the set of arrival rates $\bm{\lambda}$ under which the switch is stabilizable by some algorithm. 

For a switch under i.i.d. arrivals, it is known \cite{srikant2013communication} that the capacity region $\mathcal{C}$ is the convex hull of the feasible schedule set, i.e., $\mathcal{C} = \text{Conv}(\mathcal{S})$. It can also be written as,
\begin{equation*}
    \begin{aligned}
        \mathcal{C} =&\left\{\bm{\lambda} \in \mathbb{R}^{N^2}_+: \sum_{i=1}^N \lambda_{ij}\leq 1,\sum_{j=1}^N \lambda_{ij} \leq 1\ \forall i,j \in \{1,2,...,N\}  \right\}\\  
         =&\left\{\bm{\lambda} \in \mathbb{R}^{N^2}_+: \left\langle \bm{\lambda}, \bm{e}^{(i)}  \right\rangle\leq1, \left\langle \bm{\lambda}, \bm{\tilde{e}}^{(j)}  \right\rangle \leq1\ \forall i,j \in \{1,2,...,N\}  \right\}.        
    \end{aligned}
\end{equation*}
If a scheduling algorithm stabilizes the switch under any arrival rate in the capacity region, then it is said to be throughput optimal. Moreover, it is known that the MaxWeight algorithm is throughput optimal \cite{srikant2013communication}. We will establish similar results under Markovian arrivals in Section \ref{Throughput optimality}. Before that, we first present some geometric observations about the capacity region from \cite{maguluri2016heavy}. 

\subsubsection{Some geometric observations}
The capacity region $\mathcal{C}$ is a convex polytope with dimension $N^2$. Let $\mathcal{F}$ denote the face of  $\mathcal{C}$, where all input and output ports are all saturated, defined by
\begin{equation*}
    \begin{aligned}
        \mathcal{F} =&\left\{\bm{\lambda} \in \mathbb{R}^{N^2}_+: \sum_{i=1}^N \lambda_{ij}= 1,\sum_{j=1}^N \lambda_{ij}= 1 \forall i,j \in \{1,2,...,N\}  \right\}\\
         =&\left\{\bm{\lambda} \in \mathbb{R}^{N^2}_+: \left\langle \bm{\lambda}, \bm{e}^{(i)}  \right\rangle = 1, \left\langle \bm{\lambda}, \bm{\tilde{e}}^{(j)}  \right\rangle = 1\ \forall i,j \in \{1,2,...,N\}  \right\}.        
    \end{aligned}
\end{equation*}
Note that $\mathcal{F} = \text{Conv}(\mathcal{S^*})$. Let $\mathcal{K}$ denote the normal cone of the face $\mathcal{F}$, which can explicitly be written as 
    \begin{equation*}
        \mathcal{K} = \left\{ \bm{x} \in \mathbb{R}^{N^2}: \bm{x} = \sum_{i=1} ^N w_i\bm{e}^{(i)}+\sum_{j=1} ^N \tilde{w}_j\tilde{\bm{e}}^{(j)},\ w_i, \tilde{w}_j \in \mathbb{R}_+ \forall i,j\right\}.
    \end{equation*}
It can be verified that this is indeed the normal cone, and so we  have $\forall \bm{x},\bm{y} \in \mathcal{F}$, $\forall \bm{z} \in \mathcal{K}$
    \begin{equation}\label{eq: orthognal}
        \left\langle \bm{x}-\bm{y},\bm{z}\right\rangle =0.
    \end{equation}
The polar cone $\mathcal{K}^o$ of cone $\mathcal{K}$ is defined as
\begin{equation*}
         \mathcal{K}^o = \left\{ \bm{x} \in \mathbb{R}^{N^2}: \left\langle\bm{x},\bm{y}\right\rangle\leq 0 \ \forall \bm{y}\in \mathcal{K} \right\}.   
\end{equation*}
Let $\mathcal{L}$ denote the subspace spanned by the cone $\mathcal{K}$, and so we have,
    \begin{equation*}
        \mathcal{L} = \left\{ \bm{x} \in \mathbb{R}^{N^2}: \bm{x} = \sum_{i=1} ^N w_i\bm{e}^{(i)}+\sum_{j=1} ^N  \tilde{w}_j\tilde{\bm{e}}^{(j)},\ \text{where} \ w_i,\ \tilde{w}_j \in \mathbb{R},\ \forall i,j\right\},
    \end{equation*} 
    For any $ \bm{x} \in \mathbb{R}^{N^2}$, and any convex set $\mathcal{W}$, the projection of $\bm{x}$ on to the set $\mathcal{W}$ is denoted as $\bm{x}_{{\parallel} \mathcal{W}}$ and defined as,
\begin{equation*}
    \bm{x}_{{\parallel} \mathcal{W}} = \arg\min_{\bm{y} \in \mathcal{W}}\norm{\bm{y} - \bm{x}} 
\end{equation*}
and consequently,
\begin{equation*}
    \bm{x}_{{\perp} \mathcal{W}} = \bm{x}- \bm{x}_{{\parallel} \mathcal{W}}.
\end{equation*}
The projection of $ \bm{x}$ on to the subspace $\mathcal{L}$ is given by
\begin{equation*}
    \left(\bm{x}_{{{\parallel}\mathcal{L}}}\right)_{ij} = \frac{\sum_{i}{x}_{ij}}{N} + \frac{\sum_{j}{x}_{ij}}{N} - \frac{\sum_{ij} {x}_{ij}}{N^2},
\end{equation*}
and its $\ell_2$-norm is 
{
\begin{equation*}
    \norm{\bm{x}_{{\parallel}\mathcal{L}} }^2= \frac{1}{N}\left( \sum_j \left( \sum_{i}{x}_{ij}\right)^2 + \sum_i \left(\sum_{j}{x}_{ij}\right)^2 - \frac{1}{N} \left(\sum_{ij}{x}_{ij}\right)^2 \right).
\end{equation*}
}
For any $\bm{x}_{{\parallel}\mathcal{L}}$ and $\bm{y}_{{\parallel}\mathcal{L}}$, we have
\begin{align}\label{eq: <x||,y||>}
        \langle \bm{x}_{{\parallel}\mathcal{L}}, \bm{y}_{{\parallel}\mathcal{L}}\rangle = \langle \bm{x}, \bm{y}_{{\parallel}\mathcal{L}}\rangle = \sum_{ij}x_{ij}\left[\frac{1}{N}\sum_{j'}y_{ij'}+\frac{1}{N}\sum_{i'}y_{i'j}-\frac{1}{N^2}\sum_{i'j'}y_{i'j'}\right]
\end{align}

\subsection{Throughput optimality}\label{Throughput optimality}
In this section, we show that under Markovian arrivals, $\mathcal{C}$ is indeed the capacity region and that MaxWeight is throughput optimal. 
\begin{proposition}\label{thm: Tp_0}
    If $\lambda \notin  \mathcal{C}$, no scheduling algorithm can support arrival rate matrix $\bm{\lambda}$.
\end{proposition}
The proof follows from Theorem 4.2.1 in \cite{srikant2013communication} after using the Markov chain ergodic theorem instead of the strong law of large numbers. 
\begin{proposition}\label{thm: Tp_main}
    For an input queued switch operating under Markovian arrivals, the MaxWeight scheduling algorithm can support any arrival rate matrix $\bm{\lambda}$ in the interior of  $\mathcal{C}$, i.e., $\bm{\lambda}\in int(\mathcal{C})$. Thus, MaxWeight is throughput optimal.
\end{proposition}
\begin{proof}{}
Note $\left\{\left(\bm{Q}^t,\bm{X}^t\right)\right\}_{t\geq 0 }$ is an irreducible Markov chain under MaxWeight scheduling. We prove this theorem by demonstrating that $\left\{\left(\bm{Q}^t,\bm{X}^t\right)\right\}_{t\geq 0 }$ is positive recurrent, which can be done using Foster-Lyapunov theorem   (Proposition 2.2.4 in \cite{fayolle1995topics}). 
Consider the Lyapunov function,
\begin{equation*}
    V  \left(\bm{Q}^t,\bm{X}^t  \right)=\norm{\bm{Q}^t}^2.
\end{equation*}
We will consider the $m$-step drift of this Lyapunov function, and bound it as follows. 
\begin{claim}\label{claim: 4.2.1}
   Let $C_{max} = \max_{ij}C_{ij}$ and $\alpha_{max} = \max_{ij}\alpha_{ij}$ where $C_{ij}$ and $\alpha_{ij}$ are the geometric mixing parameters given by Lemma \ref{lem: geo} for the $(i,j)$th underlying Markov Chain $\left\{ X^t_{ij} \right\}_{t\geq 0}$. There exist $\epsilon  >0$ such that $\bm{\lambda}+\epsilon \bm{1} \in \mathcal{C}$, let
   \begin{align*}
    m(\epsilon ) =  \min \left\{m \in \mathbb{N}_+\mid 2NA_{max}C_{max}\frac{1-\alpha_{max}^m}{1-\alpha_{max}} < \frac{m\epsilon }{2} \right\},
   \end{align*}
   then
    \begin{equation*}
    \begin{aligned}
    & E  \left[V  \left(\bm{Q}^{t+m(\epsilon)} , \bm{X}^{t+m(\epsilon)} \right)-V  \left(\bm{Q}^t,\bm{X}^t  \right)\mid   \left(\bm{Q}^t  ,\bm{X}^t  \right)=(\bm{q},\bm{x})  \right]\\
    & \leq -\frac{m(\epsilon )\epsilon }{2}\norm{\bm{q}  }+\frac{K_2  (m(\epsilon )  )}{2} ,    
    \end{aligned}
    \end{equation*}
    where
    \begin{equation*}
        \begin{aligned}
            K_2  (m(\epsilon )  ) =&  m(\epsilon )N^2  (A_{max}+S_{max}  )^2\\
            &+2m(\epsilon )^2 N^2  (A_{max}+S_{max})\left(\epsilon +A_{max}+\lambda_{max} \right). 
        \end{aligned}
    \end{equation*}
\end{claim}
Then, the theorem follows from the Foster-Lyapunov theorem   (Proposition 2.2.4 in \cite{fayolle1995topics}).
The detailed proof of the claim is in Appendix \ref{Proof of claim: 4.2.1}. 
\end{proof}

\subsection{Heavy-traffic Analysis}
We will now study the switch in heavy-traffic, as the arrival rate approaches the face $\mathcal{F}$ on the boundary of the capacity region $\mathcal{C}$. To this end, consider $\left\{\bm{X}^t\right\}^{  (\epsilon  )}_{t\geq 0}$, a set of irreducible, positive recurrent and aperiodic underlying vector valued Markov chains indexed by the heavy-traffic parameter $\epsilon \in   (0,1)$ taking values in $\Omega^{N^2}$. 
The arrival process for the system indexed by $\epsilon$ is given {as}  $\left(A_{ij}^t\right)^{(\epsilon)} = {{}f_{ij}}\left(\left(X_{ij}^t\right)^{(\epsilon)}\right)$ for a non-negative valued function, ${{}f_{ij}}(\cdot)$.
For any fixed $\epsilon$, let $\overline{\bm{X}}^{  (\epsilon  )}$ be the steady state variable to which $\left(\bm{X}^t\right)^{  (\epsilon  )}$ converges in distribution with $E\left[\overline{\bm{A}}^{(\epsilon)} \right] = E\left[\bm{f}\left(\overline{\bm{X}}^{  (\epsilon  )}\right)\right] = \bm{\lambda}^{(\epsilon)} = (1 -\epsilon)\bm{v}$, where $\bm{v}$ is an arrival rate on the boundary of the capacity region $\mathcal{C}$ such that $v_{ij} > 0$ for all $i, j$, and all the input and output ports are saturated. In other words, we assume $\bm{v} \in \mathcal{F}$.
Let $\gamma_{ij,kl}^{(\epsilon)}(t)$ to be the {{}spatio-temporal} correlation function of the arrival process between the $ij-th$ port at time $t$ and $kl-th$ port at time $0$, starting from steady state $\bm{X^0} \stackrel{d}{=}\overline{\bm{X}}^{  (\epsilon  )}$,
i.e., 
    \[
    {}
        \gamma_{ij,kl}^{(\epsilon)}(t) = E \left[  \left(\left(A_{ij}^t\right)^{(\epsilon)}-\lambda_{ij}  \right)  \left(\left(A_{kl}^0\right)^{(\epsilon)}-\lambda_{kl}  \right) \right].
    \]
Assume that $\forall t \in \mathbb{N}$, 
\begin{equation}\label{eq: gamma_ij to gamma_ij}
   {} \lim_{\epsilon \to 0}\gamma_{ij,kl}^{(\epsilon)}(t) = \gamma_{ij,kl}(t).
\end{equation}
    Let
        \[{}
        \sigma_{ij,kl}^2 = \gamma_{ij,kl}(0)+2\lim_{m \to \infty}\sum_{t=1}^{m}\gamma_{ij,kl}(t),
    \]
{ and $\bm{\sigma}=(\sigma_{ij,kl}) \in \mathbb{R}^{N^2\times N^2}$ is the limiting effective covariance matrix}. For any scheduling algorithm under which the switch system is stable, let $\left(  \left(\overline{\bm{Q}}\right)^{(\epsilon )},\left(\overline{\bm{X}}\right)^{(\epsilon )} \right)$  be a steady-state random variable to which
     the process $\left\{ \left( \left(\bm{Q}^t \right)^{(\epsilon)},\left(\bm{X}^t \right)^{(\epsilon)}  \right)\right\}_{t\geq 0}$ converges in distribution. 
    Assume for any $\epsilon \in (0,1)$ and for any ${i,j}$, $\sup_\epsilon \alpha^{(\epsilon)}_{ij}<\alpha_{max}<1$ and $\sup_{\epsilon}C^{(\epsilon)}_{ij} < C_{max} < \infty$.   

\subsubsection{Universal lower bound}
In this section, we will give a lower bound on the average queue length under heavy traffic, which is valid under any stable scheduling algorithm.
\begin{proposition}\label{prop: unicersal lower bound}
    Consider a set of switch systems parameterized by heavy-traffic parameter $\epsilon\in (0,1)$ described before. Consider a scheduling algorithm under which the switch system is stable. Then, under heavy traffic, we have 
    \begin{equation*}
        \liminf_{\epsilon \to 0} \epsilon E\left[\sum_{ij}\overline{\bm{Q}^{(\epsilon)}} \right] \geq \frac{\norm{\bm{\sigma}}^2}{2}.
        \end{equation*}
\end{proposition}
Using a coupling argument, it was shown in \cite{maguluri2016heavy} that the row sum of queue lengths under any policy is lower bounded by a single server queue with an arrival process that is sum of the arrivals to all the queues in the row. The result then follows from using Theorem \ref{thm: pos rec single server queue} for a single server queue, and so we omit the proof.

\subsubsection{State space collapse under MaxWeight policy}
A key step in characterizing the heavy-traffic behavior is to establish state space collapse. It was shown in \cite{maguluri2016heavy} that under MaxWeight algorithms, the queue lengths collapse into the cone $\mathcal{K}$ in heavy-traffic. In this subsection, we establish the same result under Markovian arrivals. 
More formally, we will show that under the MaxWeight algorithm, $\overline{\bm{q}}_{{\perp\mathcal{K}}}^{(\epsilon)}$ can be bounded by some constant independent of $\epsilon$. Thus, when the heavy-traffic parameter $\epsilon$ goes to zero (the mean arrival rate vector $\bm{\lambda}$ approaches the boundary $\mathcal{F}$ of the capacity region $\mathcal{C}$),
$\overline{\bm{q}}_{{\perp\mathcal{K}}}^{(\epsilon)}$ is negligible compared to $\overline{\bm{q}}_{{\parallel\mathcal{K}}}^{(\epsilon)}$. 

We define the following quadratic Lyapunov functions and their corresponding $m$-step drifts:
\begin{align*}
       V  (\bm{q},\bm{x}  )  & \triangleq \norm{\bm{q}}^2 = \sum_{ij}q_{ij}^2,  &W_{{\perp\mathcal{K}}}  (\bm{q},\bm{x}  )  &\triangleq \norm{\bm{q}_{\perp\mathcal{K}}}, \\
        V_{\perp\mathcal{K}}  (\bm{q},\bm{x}  )  &\triangleq \norm{\bm{q}_{\perp\mathcal{K}}}^2 = \sum_{ij}q_{{\perp\mathcal{K}} ij}^2,  &V_{\parallel\mathcal{K}}  (\bm{q} ,\bm{x} )  &\triangleq \norm{\bm{q}_{\parallel\mathcal{K}}}^2 = \sum_{ij}q_{{\parallel\mathcal{K}} ij}^2,
\end{align*}
\begin{equation*}
    \begin{aligned}
        & \Delta^m V  (\bm{q},\bm{x}  ) \triangleq   \left[V  \left(\bm{Q}^{t+m },\bm{X}^{t+m}   \right)-V  \left(\bm{Q}^t,\bm{X}^t    \right)  \right] \mathcal{I}  \left(  (\bm{Q}^t  ,\bm{X}^t    \right) =   (\bm{q},\bm{x}  )  ),\\
        & \Delta^m W_{\perp\mathcal{K}}  (\bm{q},\bm{x}  ) \triangleq  \left[W_{\perp\mathcal{K}}  \left( \bm{Q}^{t+m},\bm{X}^{t+m}  \right) -W_{\perp\mathcal{K}}  \left(\bm{Q}^t ,\bm{X}^t   \right)  \right] \mathcal{I}  \left(  \left(\bm{Q}^t  ,\bm{X}^t    \right) =   (\bm{q},\bm{x}  )  \right),\\ 
        & \Delta^m V_{\perp\mathcal{K}}  (\bm{q},\bm{x}  ) \triangleq   \left[V_{\perp\mathcal{K}}  \left(\bm{Q}^{t+m} ,\bm{X}^{t+m} \right)-V_{\perp\mathcal{K}}  \left(\bm{Q}^t ,\bm{X}^t   \right)  \right] \mathcal{I}  \left(  \left(\bm{Q}^t  ,\bm{X}^t   \right ) =   (\bm{q},\bm{x}  ) \right),\\
        & \Delta^m V_{\parallel\mathcal{K}}  (\bm{q} ,\bm{x} ) \triangleq   \left[V_{\parallel\mathcal{K}} \left(\bm{Q}^{t+m},\bm{X}^{t+m}  \right)-V_{\parallel\mathcal{K}}  \left(\bm{Q}^t ,\bm{X}^t \right)  \right] \mathcal{I}  \left(  \left(\bm{Q}^t  ,\bm{X}^t    \right) =   (\bm{q},\bm{x}  )  \right).\\
    \end{aligned}
\end{equation*}

The next proposition states the state space collapse.
\begin{proposition}\label{prop: upper bound of q_perp}
    Consider a set of switch systems under MaxWeight scheduling algorithm parameterized by heavy-traffic parameter $0<\epsilon<1$ described before, further assume that $v_{\min} \triangleq \min_{ij}v_{ij}>0$.
    Then, for each system with $\epsilon < v_{\min}/2\norm{\bm{v}}$, the steady state queue length vector satisfies
    \begin{equation*}
       E \left[\norm{\left(\overline{\bm{Q}}_{\perp\mathcal{L}} \right)^{(\epsilon )}}^2  \right]\leq  E \left[\norm{\left(\overline{\bm{Q}}_{\perp\mathcal{K}} \right)^{(\epsilon )}}^2  \right] \leq K^{\star\star},
    \end{equation*}
    where $K^{\star\star}$ is a constant that does not depend on $\epsilon$.
\end{proposition}
To prove the proposition, we need the following two lemmas.
\begin{lemma}\label{lemma: Maxweight in C}
        Under MaxWeight algorithm, for any $\bm{q} \in \mathbb{R}^{N^2}$,
        \begin{equation*}
            \bm{v} +\frac{v_{\text{min}}}{\norm{\bm{q}_{\perp\mathcal{K}}}}\bm{q}_{\perp\mathcal{K}} \in \mathcal{C}.
        \end{equation*}
\end{lemma}
This Lemma was proved in \cite[Claim 2]{maguluri2016heavy}.
\begin{lemma}\label{lemma: bound W_perp}
        Drift of $W_{{\perp\mathcal{K}}}  (\cdot  )$ can be bounded in terms of drift of $V  (\cdot  )$ and $V_{\parallel\mathcal{K}}  (\cdot  )$ as follows:
        \begin{equation*}
         \Delta^m W_{\perp\mathcal{K}}  (\bm{q},\bm{x}  ) \leq \frac{1}{2\norm{\bm{q}_{\perp\mathcal{K}}}}  \left(\Delta^m V  (\bm{q},\bm{x}  )-\Delta^m V_{\perp\mathcal{K}}   (\bm{q},\bm{x} )  \right),   \quad \forall \left(\bm{q},\bm{x}\right) \in \bm{R}^{N^2}.
        \end{equation*}
\end{lemma}
The proofs of this lemma is almost identical to Lemma 7 in \cite{eryilmaz2012asymptotically}, and so we skip it here.
\begin{proof}  {of Proposition \ref{prop: upper bound of q_perp}.}
 For ease of exposition, the superscript $\epsilon$ in this proof is skipped, i.e., we will use $\bm{Q}^t$ and $\lambda$ to denote $\left(\bm{Q}^t\right)^{(\epsilon)}$ and $\bm{\lambda}^{(\epsilon)}$ respectively. The proof is based on using Lemma \ref{Lemma: $m$-step bound for q_perp_norm_2}, by verifying that the {three} conditions are satisfied. 
First we start with the following claim which is proved in Appendix \ref{proof of claim 4.3.1}.
    \begin{claim}\label{claim: 4.3.1}
    For any $m\in \mathbb{N}_+$,
        \begin{equation*}
            \begin{aligned}
               {\abs{ \Delta^m W_{\perp\mathcal{K}}  (\bm{q},\bm{x}  )} \leq m \abs{ \Delta W_{\perp\mathcal{K}}  (\bm{q},\bm{x}  )} \leq N m  (A_{max}+S_{max}  ).}
            \end{aligned}
        \end{equation*}
    \end{claim}
{Therefore, Conditions 2 and 3 are satisfied with $D(m) = Nm(A_{max}+S_{max})$ and $\hat{D}(t_0) =N  (A_{max}+S_{max}  ) $}. To verify Condition 1, we need the following claim, which is proved in Appendix  \ref{proof of claim 4.3.2}. 
    \begin{claim}\label{claim: 4.3.2}
    For any $m\in \mathbb{N}_+$,
        \begin{equation*}
        \begin{aligned}
            &E  \left[\Delta^m W_{\perp\mathcal{K}}  (\bm{q},\bm{x}  ) \mid   \left(\bm{Q}^t  ,\bm{X}^t    \right) =   (\bm{q},\bm{x}  )\right]\\
            & \leq E  \left[\sum_{l=1}^m\norm{\bm{A}^{t+m-l}  -\bm{\lambda}}\mid   \left(\bm{Q}^t  ,\bm{X}^t    \right) =   (\bm{q},\bm{x}  )  \right]+ \frac{K_3(m)}{2\norm{\bm{q}_{\perp\mathcal{K}}  }}+m  (\epsilon\norm{\bm{v}}-v_{\min}  ),\\
        \end{aligned}
        \end{equation*}
        where
        \begin{equation*}
            \begin{aligned}
              K_3(m) &= 2Nm^2\left(A_{max}+S_{max}  \right)  (N(A_{max}+\lambda)+\epsilon\norm{\bm{v}} +v_{\min} )\\&+N^2m  (A_{max}+S_{max}  )^2.   
            \end{aligned}
        \end{equation*}
    \end{claim}

We can use Lemma \ref{lem: geo} to bound the first term on the RHS above, as follows. 
\begin{equation*}
        \begin{aligned}
            &E  \left[\sum_{l=1}^m\norm{\bm{A}^{t+m-l}  -\bm{\lambda}}\mid   \left(\bm{Q}^t  ,\bm{X}^t    \right) =   (\bm{q},\bm{x}  )  \right] \\
            &\leq \sum_{l=1}^{m}\sqrt{\sum_{ij} 4A_{max}^2C_{ij}^2\alpha_{ij}^{2  (m-l  )}}\\
            & \leq \sum_{l=0}^{m-1}2N A_{max} C_{max}\alpha_{max}^l\\
            & = 2N A_{max} C_{max} \frac{1-\alpha_{max}^m}{1-\alpha_{max}}.\\
        \end{aligned}
    \end{equation*}
Let
    \begin{equation*}
        \begin{aligned}
            &m_0 =  \min\Bigg \{m \in \mathbb{N}_+ \mid  \frac{2 N A_{max} C_{max}}{1-\alpha_{max}} \leq \frac{m v_{\min}}{4} \Bigg \}.
        \end{aligned}
    \end{equation*}
Clearly, such a $m_0 \in \mathbb{N}_+$ exists.  

Therefore, we have for $m\geq m_0$, 
\begin{align*}
    E  \left[\sum_{l=1}^m\norm{\bm{A}^{t+m-l}  -\bm{\lambda}}\mid   \left(\bm{Q}^t  ,\bm{X}^t    \right) =   (\bm{q},\bm{x}  )  \right] \leq \frac{m v_{\min}}{4}. 
\end{align*}
So, using this with Claim \ref{claim: 4.3.2}, when $\epsilon \leq \frac{v_{\min}}{4\norm{\bm{v}}}$, we have
    \begin{equation*}
        \begin{aligned}
            E  \left[\Delta^{m_0} W_{\perp\mathcal{K}}  (\bm{q},\bm{x}  )  \mid   \left(\bm{Q}^t  ,\bm{X}^t    \right) =   (\bm{q},\bm{x}  )  \right]
            &\leq \frac{K_3(m_0)}{2\norm{\bm{q}_{\perp\mathcal{K}}  }} -\frac{m_0 v_{\min}}{2} \leq -\frac{v_{\min}}{4},
        \end{aligned}
    \end{equation*}
    for all $\bm{q}\ \text{such that}\  W_{\perp\mathcal{K}}  (\bm{q},\bm{x}  )\geq \frac{2K_3  (m_0  )}{m_0v_{\min}}$.
    Let
    \begin{align*}
        Z  (\bm{q},\bm{x}  ) = W_{\perp\mathcal{K}}  (\bm{q},\bm{x}  ),
    \end{align*}
    {Define $\kappa(m_0) = \frac{2K_3  (m_0  )}{m_0v_{\min}}$, $\eta = \frac{v_{\min}}{4}$, $D(m_0) = N m_0  (A_{max}+S_{max}  )$ and $\hat{D}(t_0) = N(A_{max}+S_{max}  )$, which verifies three conditions in Lemma \ref{Lemma: $m$-step bound for q_perp_norm_2}. We have that there exists a $K^{\star\star}(m_0)$ such that
    \begin{equation*}
        E \left[\norm{\left(\overline{\bm{Q}}_{\perp\mathcal{K}} \right)^{(\epsilon )}}^2  \right] \leq K^{\star\star}(m_0).
    \end{equation*}
Since $\mathcal{K} \in \mathcal{L}$, we conclude that
    \begin{align*}
        E \left[\norm{\left(\overline{\bm{Q}}_{\perp\mathcal{L}} \right)^{(\epsilon )}}^2  \right] \leq E \left[\norm{\left(\overline{\bm{Q}}_{\perp\mathcal{K}} \right)^{(\epsilon )}}^2  \right] \leq K^{\star\star}(m_0).
    \end{align*}
Since the parameters $ \eta, \kappa(m_0), D(m_0)$ and $\hat{D}(t_0)$ in the three conditions of the Lemma \ref{Lemma: $m$-step bound for q_perp_norm_2} do not involve $\epsilon$, we have that $K^{\star\star}(m_0)$  does not depend on $\epsilon$. To simplify the notation, we will use $K^{\star\star}$ to represent $K^{\star\star}(m_0)$. }
\end{proof}

\subsubsection{Asymptotically tight upper and lower bounds under the MaxWeight policy} 
We now use the state space collapse result from the previous section to obtain an exact expression for the heavy-traffic scaled mean sum of the queue lengths in an input queued switch under Markovian arrivals. In particular, we will obtain lower and upper bounds on the steady-state mean queue lengths that differ by only $o(1/\epsilon)$, and so are negligible in the heavy-traffic limit. We will again use the $m$-step drift argument from Section \ref{sec: singleserver}. 
We will use $V'(\bm{q},\bm{x}) = \norm{\textbf{q}_{{\parallel}\mathcal{L}}}^2$ as the Lyapunov function to this end. This Lyapunov function was introduced in the i.i.d. case in \cite{maguluri2018optimal}.

The $1$-step drift of $V'(q)$ is defined as 
\begin{equation}
        \Delta V'  (\bm{q} ,\bm{x} )  \triangleq   \left[V'  \left(\bm{Q}^{t+1},\bm{X}^{t+1} \right)-V'  \left(\bm{Q}^t,\bm{X}^t    \right)    \right]\mathcal{I}  \left(\left(\bm{Q}^t,\bm{X}^t\right)  =(\bm{q},\bm{x})  \right).
\end{equation}

\begin{lemma}\label{Lemma: finite V_1, V_2, V_3}
For any arrival rate vector $\bm{\lambda}$ in the interior of the capacity region $\bm{\lambda} \in int  (\mathcal{C}  )$, the steady state mean $E\left[\norm{\overline{\bm{Q}}}^2\right]$ is finite and consequently $E  [V'  (\bm{Q} ,\bm{X} )  ]$ is finite for the input queued switch under MaxWeight algorithm.
\end{lemma}
This lemma is needed to set the drift of $V'  (\bm{q},\bm{x})$ to zero in steady-state. The lemma is proved in Appendix \ref{Proof of Lemma finite V_1, V_2, V_3}.
We will now state and prove the main result of this paper.
\begin{theorem}\label{thm: Switch heavy-traffic limit}
        Consider a set of switch systems under MaxWeight scheduling algorithm, 
        parameterized by heavy-traffic parameter $0<\epsilon<1$ described before. For each system with $\epsilon < v_{\min}/2\norm{\bm{v}}$, in the heavy-traffic limit, as $\epsilon \to 0$, we have 
        \begin{equation}\label{Eq Switch heavy-traffic limit}
        {}
            \lim_{\epsilon \to 0}\epsilon E\left[\sum_{ij}\overline{Q}_{ij}^\epsilon \right]
            =  \sum_{ij}\left(\frac{1}{N} \sum_{j'} \sigma^2_{ij,ij'}+\frac{1}{N} \sum_{i'}\sigma^2_{ij,i'j}-\frac{1}{N^2} \sum_{i'j'} \sigma^2_{ij,i'j'}\right). 
        \end{equation}
\end{theorem}
\begin{corollary}
{}\label{corollary 1}
If the underlying Markov chains are independent across input-output pairs, i.e., $\sigma_{ij,kl}=0, \forall ij\neq kl$, then Eq \eqref{Eq Switch heavy-traffic limit} can be  simplified as
\begin{equation*}
\lim_{\epsilon \to 0}\epsilon E\left[\sum_{ij}\overline{Q}_{ij}^\epsilon \right]=  (1-\frac{1}{2N}  )\norm{\bm{\sigma}}^2. 
\end{equation*}
\end{corollary}

\begin{corollary}
{} \label{corollary 2}
If the arrival process are i.i.d. across time, i.e., $\gamma_{ij,kl} (t)=0, \forall t\in\{1,2,...\}$, then Eq \eqref{Eq Switch heavy-traffic limit} can be simplified as
\begin{equation*}
\lim_{\epsilon \to 0}\epsilon E\left[\sum_{ij}\overline{Q}_{ij}^\epsilon \right]=  \sum_{ij}\left(\frac{1}{N} \sum_{j'} \gamma^2_{ij,ij'}(0)+\frac{1}{N} \sum_{i'}\gamma^2_{ij,i'j}(0)-\frac{1}{N^2} \sum_{i'j'} \gamma^2_{ij,i'j'}(0)\right), 
\end{equation*}
which matches the result in \cite{hurtado2022heavy} (Corollary 1, Page 13).
\end{corollary}
{\textbf{Remark}: Although Theorem \ref{thm: Switch heavy-traffic limit} and Corollaries \ref{corollary 1} and \ref{corollary 2} present the results in the heavy traffic limit as $\epsilon \to 0$, our approach provides upper and lower bounds on $E\left[\sum_{ij}\overline{Q}_{ij}^\epsilon \right]$ for all  $\epsilon \in(0, v_{\min}/2\|\bm{v}\|)$. See equations \eqref{upper bound} and \eqref{lower bound} for 
 the explicit bounds.}
\begin{proof}  {of Theorem \ref{thm: Switch heavy-traffic limit}.  }
    Fix an $\epsilon \in (0, v_{\min}/2\norm{\bm{v}})$ and we consider the system with index $\epsilon$. In the rest of the proof, we consider the system in its steady-state, and so for every time $t$, 
 $\left(\bm{Q}^t\right)^{(\epsilon)} \stackrel{d}= \overline{\bm{Q}}^{(\epsilon)}$. We again skip the superscript ${(\cdot)}^{(\epsilon)}$ in this proof and use $\bm{Q}^t$ to denote the steady state queue length vector. Then $\bm{A}^t$ denotes the arrival vector in steady state and $\bm{Q}^{t+m}$ denote the queue length at time $t+m$, $\bm{Q}^{t+m} = \bm{Q}^t+\sum_{l=0}^{m-1} \bm{A}^{t+l}-\sum_{l=0}^{m-1} \bm{S}^{t+l}+\sum_{l=0}^{m-1} \bm{U}^{t+l}$ which has the same distribution as $\bm{Q}^t$ for all $m \in \mathbb{N}_+$.
    
    The steady state mean $V' (\bm{Q}^t  )$ is finite from Lemma \ref{Lemma: finite V_1, V_2, V_3}. Therefore, we can set the mean drift of $V'  (\cdot  )$ in steady state to zero,
    \begin{align}
        E  \left[\Delta V'  (\bm{q} ,\bm{x} )  \right]=&-\underbrace{2E \left[\left\langle \bm{Q}^t_{{\parallel}\mathcal{L}},\bm{S}^t_{{\parallel}\mathcal{L}}-\bm{A}^t_{{\parallel}\mathcal{L}}  \right\rangle \right]}_{\mathcal{T}_{1}} +\underbrace{E\left[ \norm{\bm{A}^t_{{\parallel}\mathcal{L}}-\bm{S}^t_{{\parallel}\mathcal{L}}}^2 \right]}_{\mathcal{T}_{2}}  \nonumber \\
        &-\underbrace{E \left[ \norm{\bm{U}^t_{{\parallel}\mathcal{L}}}^2 \right]}_{\mathcal{T}_{3}}+\underbrace{2E\left[ \left \langle  \bm{Q}^{t+1}_{{\parallel}\mathcal{L}}, \bm{U}^t_{{\parallel}\mathcal{L}} \right \rangle \right]}_{\mathcal{T}_{4}} =0.\label{sum t}
    \end{align}
We will now bound each of those four terms. In the steady state, we have 
\begin{equation*}
\begin{aligned}
       &E   \left[ \langle\bm{Q}^{t},\bm{A}^{t}  \rangle   \right]
      =E   \left[\langle\bm{Q}^{t+m},\bm{A}^{t+m}  \rangle  \right] \\
\end{aligned}
\end{equation*}
since the steady state mean $E   \left[ \langle\bm{Q}^{t},\bm{A}^{t}  \rangle   \right] \leq \sqrt{E\left[ \norm{\bm{Q}^{t}}^2\right]E\left[ \norm{\bm{A}^{t}}^2 \right]}< \infty $ according to Lemma \ref{Lemma: finite V_1, V_2, V_3}. 
        \begin{align}
            \mathcal{T}_{1} =& 2E \left[\left\langle \bm{Q}^t_{{\parallel}\mathcal{L}},\bm{S}^t_{{\parallel}\mathcal{L}}-\bm{\lambda}_{{\parallel}\mathcal{L}}  \right\rangle \right]+2E \left[\left\langle \bm{Q}^t_{{\parallel}\mathcal{L}},\bm{\lambda}_{{\parallel}\mathcal{L}}-\bm{A}^t_{{\parallel}\mathcal{L}}  \right\rangle \right]  \nonumber\\
             =& 2E \left[\left\langle \bm{Q}^t_{{\parallel}\mathcal{L}},\bm{S}^t_{{\parallel}\mathcal{L}}-\bm{\lambda}_{{\parallel}\mathcal{L}}  \right\rangle \right]+2E \left[\left\langle \bm{Q}^{t+m}_{{\parallel}\mathcal{L}},\bm{\lambda}_{{\parallel}\mathcal{L}}-\bm{A}^{t+m}_{{\parallel}\mathcal{L}}  \right\rangle \right] \nonumber \\
             = &2E \left[\left\langle \bm{Q}^t_{{\parallel}\mathcal{L}},\bm{S}^t_{{\parallel}\mathcal{L}}-\bm{\lambda}_{{\parallel}\mathcal{L}}  \right\rangle \right]  \nonumber\\
            & +2E \left[\left\langle \bm{Q}^{t}_{{\parallel}\mathcal{L}}+\sum_{l=1}^m \bm{A}^{t+m-l}_{{\parallel}\mathcal{L}}-\bm{S}^{t+m-l}_{{\parallel}\mathcal{L}}+\bm{U}^{t-m+l}_{{\parallel}\mathcal{L}},\bm{\lambda}_{{\parallel}\mathcal{L}}-\bm{A}^{t+m}_{{\parallel}\mathcal{L}}  \right\rangle \right]  \nonumber\\
             = &2E \left[\left\langle \bm{Q}^t_{{\parallel}\mathcal{L}},\bm{S}^t_{{\parallel}\mathcal{L}}-\bm{\lambda}_{{\parallel}\mathcal{L}}  \right\rangle \right]+2E \left[\left\langle \bm{Q}^{t}_{{\parallel}\mathcal{L}},\bm{\lambda}_{{\parallel}\mathcal{L}}-\bm{A}^{t+m}_{{\parallel}\mathcal{L}}  \right\rangle \right]  \nonumber\\
             & +2E \left[\left\langle\sum_{l=1}^m \bm{A}^{t+m-l}_{{\parallel}\mathcal{L}}-\bm{\lambda}_{{\parallel}\mathcal{L}},\bm{\lambda}_{{\parallel}\mathcal{L}}-\bm{A}^{t+m}_{{\parallel}\mathcal{L}}  \right\rangle \right]  \nonumber\\
            & +2E \left[\left\langle\sum_{l=1}^m \bm{\lambda}_{{\parallel}\mathcal{L}}-\bm{S}^{t+m-l}_{{\parallel}\mathcal{L}},\bm{\lambda}_{{\parallel}\mathcal{L}}-\bm{A}^{t+m}_{{\parallel}\mathcal{L}}  \right\rangle \right]  \nonumber\\
            & +2E \left[\left\langle\sum_{l=1}^m \bm{U}^{t-m+l}_{{\parallel}\mathcal{L}}, \bm{\lambda}_{{\parallel}\mathcal{L}}-\bm{A}^{t+m}_{{\parallel}\mathcal{L}}  \right\rangle \right]  \label{eq: t1}
        \end{align}
    
We will now bound each of the terms on the RHS of \eqref{eq: t1}. The first term  can be simplified as following. Since we assumed that the schedule is always perfect matching, we have, $\sum_i s_{ij}=\sum_j s_{ij}=1$. Since $\bm{\lambda}^\epsilon =   (1-\epsilon  )\bm{v}$ and $\bm{v} \in \mathcal{F}$,  $\sum_i \lambda_{ij}=\sum_j \lambda_{ij}=1-\epsilon$, according to \eqref{eq: <x||,y||>} we have
    \begin{equation*}
        \begin{aligned}
            E \left[\left\langle \bm{Q}^t_{{\parallel}\mathcal{L}},\bm{S}^t_{{\parallel}\mathcal{L}}-\bm{\lambda}_{{\parallel}\mathcal{L}}  \right\rangle \right]= \frac{\epsilon}{N} E \left[\sum_{ij}Q_{ij}^t\right].
        \end{aligned}
    \end{equation*}
    
    For second term in the RHS of \eqref{eq: t1}, according to \eqref{eq: <x||,y||>} we have
        \begin{align*}
            & E \left[\left\langle \bm{Q}^{t}_{{\parallel}\mathcal{L}},\bm{\lambda}_{{\parallel}\mathcal{L}}-\bm{A}^{t+m}_{{\parallel}\mathcal{L}}  \right\rangle \right] \\
            =& \frac{1}{N}E   \left[\sum_{ij} Q_{ij}^t    \left(\sum_{j'}  \left(\lambda_{ij'}-A^{t+m}_{ij'}  \right)    \right)    \right]\\
            &+  \frac{1}{N}E   \left[\sum_{ij} Q_{ij}^t    \left(\sum_{i'}  \left(\lambda_{i'j}-A^{t+m}_{i'j}  \right)    \right)    \right]\\
            &-\frac{1}{N^2}E   \left[    \sum_{ij} Q_{ij}^t     \left(\sum_{i'j'}  \left(\lambda_{i'j'}-A^{t+m}_{i'j'}  \right)    \right)    \right].
        \end{align*}
    According to Lemma \ref{lem: geo}, due to geometric mixing of the Markov chain underlying the arrivals, we have
    \begin{equation*}
        \begin{aligned}
           \left \vert E   \left[\sum_{ij}Q_{ij}^t   \left(\sum_{j'}  \left(\lambda_{ij'}-A^{t+m}_{ij'}  \right)    \right)    \right] \right\vert
           &\leq E   \left[\sum_{ij} Q_{ij}^t      \left(\sum_{j'}2A_{max}C_{ij}\alpha_{ij}^m    \right)    \right]\\
           &\leq 2NA_{max}C_{max}\alpha_{max}^m E   \left[ \sum_{ij} Q_{ij}^t     \right].
        \end{aligned}
    \end{equation*}
    where $C_{max} = \max_{ij}\{C_{ij}\}$ and $\alpha_{max} = \max_{ij}\{\alpha_{ij}\}$.
    Thus the second term in the RHS \eqref{eq: t1} can be bounded as follows,
    \begin{equation*}
        \abs{E \left[\left\langle \bm{Q}^{t}_{{\parallel}\mathcal{L}},\bm{\lambda}_{{\parallel}\mathcal{L}}-\bm{A}^{t+m}_{{\parallel}\mathcal{L}}  \right\rangle \right]} \leq 6A_{max}C_{max}\alpha_{max}^m E   \left[ \sum_{ij} Q_{ij}^t     \right].
    \end{equation*}
    
    For the third term in the RHS of \eqref{eq: t1}, according to \eqref{eq: <x||,y||>} we have 
    \begin{equation*}
    {}
        \begin{aligned}
           &E \left[\left\langle\sum_{l=1}^m \bm{A}^{t+m-l}_{{\parallel}\mathcal{L}}-\bm{\lambda}_{{\parallel}\mathcal{L}},\bm{\lambda}_{{\parallel}\mathcal{L}}-\bm{A}^{t+m}_{{\parallel}\mathcal{L}}  \right\rangle \right]\\
           & = -\sum_{ij}\sum_{l=1}^{m}E \bigg[  \sum_{j'}\frac{1}{N}(A_{ij}^{t+m}-\lambda_{ij}  ) (A_{ij'}^{t+m-l}-\lambda_{ij'})\\
           &\quad\quad\quad\quad\quad\quad +\sum_{i'}\frac{1}{N}(A_{ij}^{t+m}-\lambda_{ij}  ) (A_{i'j}^{t+m-l}-\lambda_{i'j})\\
           &\quad\quad\quad\quad\quad\quad -\sum_{i'j'} \frac{1}{N^2} (A_{ij}^{t+m}-\lambda_{ij}  )(A_{i'j'}^{t+m-l}-\lambda_{i'j'}) \bigg]\\
           & = -\sum_{ij}\sum_{l=1}^{m}\left(\frac{1}{N} \sum_{j'} \gamma_{ij,ij'}(l)+\frac{1}{N} \sum_{i'} \gamma_{ij,i'j}(l)-\frac{1}{N^2} \sum_{i'j'} \gamma_{ij,i'j'}(l)  \right).
        \end{aligned}
    \end{equation*}
    
    For the fourth  term in the RHS of \eqref{eq: t1}, since the future arrival is independent of past service, and the Markov chain is in steady state, according to \eqref{eq: <x||,y||>} we have
    \begin{equation*}
        \begin{aligned}
               E \left[\left\langle\sum_{l=1}^m \bm{\lambda}^{t+m-l}_{{\parallel}\mathcal{L}}-\bm{S}^{t+m-l}_{{\parallel}\mathcal{L}},\bm{\lambda}_{{\parallel}\mathcal{L}}-\bm{A}^{t+m}_{{\parallel}\mathcal{L}}  \right\rangle \right]
               =0.
        \end{aligned}
    \end{equation*}
    
    Now we bound the fifth term in the RHS \eqref{eq: t1}. From Lemma \ref{Lemma: finite V_1, V_2, V_3},
    we have $E\left[\norm{\bm{\overline{Q}}}_2\right] < \infty$. Since $\abs{E\left[ \sum_{ij} \overline{Q}_{ij} \right]} = \norm{\bm{\overline{Q}}}_1 \leq N^2E\left[\norm{\bm{\overline{Q}}}_2\right]$, 
    we conclude that $E\left[ \sum_{ij} \overline{Q}_{ij} \right]$ is finite. Thus the drift is zero in steady state,
    \begin{align*}
            E    \left[ \sum_{ij} \left( Q^{t+1}_{ij} - Q^{t}_{ij}\right) \right] &= E   \left[  \sum_{ij} \left( A^t_{ij} - S^t_{ij} +U^t_{ij} \right) \right]\\
            &= E   \left[ N(1-\epsilon) - N+  \sum_{ij}  U^t_{ij}   \right] 
            =0,
    \end{align*}
    which gives
        $E   \left[   \sum_{ij} U_{ij}^t \right]  = N\epsilon. $
        
    Since $\abs{\sum_{j'}  \left(\lambda_{ij'}-A^{t+m}_{ij'}  \right) }\leq 1+N A_{max}$, according to \eqref{eq: <x||,y||>}, we have
    \begin{equation*}
        \begin{aligned}
           \abs{E \left[\left\langle\sum_{l=1}^m \bm{U}^{t-m+l}_{{\parallel}\mathcal{L}}, \bm{\lambda}_{{\parallel}\mathcal{L}}-\bm{A}^{t+m}_{{\parallel}\mathcal{L}}  \right\rangle \right]} \leq 3m\epsilon   (1+NA_{max}).
        \end{aligned}
    \end{equation*}
    Putting all terms back in \eqref{eq: t1}, we conclude that
    \begin{equation}\label{eq: T01}
        {}
        \begin{aligned}
           \mathcal{T}_{1} \geq & 2\left (\frac{1}{N}-6A_{max}C_{max}\frac{\alpha_{max}^m}{\epsilon}\right) E \left[\epsilon\sum_{ij} Q_{ij}^t\right]\\
           &-2\sum_{ij}\sum_{l=1}^{m}\left(\frac{1}{N} \sum_{j'} \gamma_{ij,ij'}(l)+\frac{1}{N} \sum_{i'} \gamma_{ij,i'j}(l)-\frac{1}{N^2} \sum_{i'j'} \gamma_{ij,i'j'}(l)  \right)\\
           &-6m\epsilon   (1+NA_{max}  ),
        \end{aligned}
    \end{equation}
    and
    \begin{equation}\label{eq: T02}
    {}
        \begin{aligned}
           \mathcal{T}_{1}  \leq &2\left (\frac{1}{N}+6A_{max}C_{max}\frac{\alpha_{max}^m}{\epsilon}\right) E \left[\epsilon\sum_{ij} Q_{ij}^t\right]\\
           &-2\sum_{ij}\sum_{l=1}^{m}\left(\frac{1}{N} \sum_{j'} \gamma_{ij,ij'}(l)+\frac{1}{N} \sum_{i'} \gamma_{ij,i'j}(l)-\frac{1}{N^2} \sum_{i'j'} \gamma_{ij,i'j'}(l)  \right)\\
           &+6m\epsilon   (1+NA_{max}  ).
        \end{aligned}
    \end{equation}    

In bounding $\mathcal{T}_{2}$,  $\mathcal{T}_{3}$ and  $\mathcal{T}_{4}$, we do not have to deal with the correlation between $\bm{Q}$ and $\bm{A}$, and so they can be bounded in the same manner as in \cite{maguluri2016heavy}. We present a brief overview here. It can be shown as in \cite{maguluri2016heavy,maguluri2018optimal} that
    \begin{equation}\label{eq: T2}
    {}
        \begin{aligned}
            \mathcal{T}_{2}&=E\left[ \norm{\bm{A}^t_{{\parallel}\mathcal{L}}-\bm{S}^t_{{\parallel}\mathcal{L}}}^2 \right]\\
            & = E\left[ \norm{\bm{\lambda}_{{\parallel}\mathcal{L}}-\bm{S}^t_{{\parallel}\mathcal{L}}}^2 \right]+E\left[ \norm{\bm{A}^t_{{\parallel}\mathcal{L}}-\bm{\lambda}_{{\parallel}\mathcal{L}}}^2 \right]\\
            &= \epsilon^2+  \sum_{ij}\left(\frac{1}{N} \sum_{j'} \gamma_{ij,ij'}(0)+\frac{1}{N} \sum_{i'} \gamma_{ij,i'j}(0)-\frac{1}{N^2} \sum_{i'j'} \gamma_{ij,i'j'}(0)  \right).
        \end{aligned}
    \end{equation}    

 The term   $\mathcal{T}_{3}$ can be bounded as 
    \begin{equation}
        0\leq \mathcal{T}_{3} = E \left[ \norm{\bm{U}^t_{{\parallel}\mathcal{L}}}^2 \right]\\
        \leq E \left[ \norm{\bm{U}^t}^2 \right]\\
        =E \left[ \sum_{ij}\left(U_{ij}^t\right)^2 \right] \stackrel{(a)}= E \left[ \sum_{ij}U_{ij}^t \right]\\
         =N\epsilon
    \end{equation}
    where $(a)$ follows from $u_{ij} \in\{0,1\}$, and the last equality follows from \eqref{ E(u)=epsilon}. Now for the term $\mathcal{T}_{4}$, we have 
    \begin{align}\label{eq: T3,4}
        \mathcal{T}_{4} = 2E\left[ \left \langle  \bm{Q}^{t+1}_{{\parallel}\mathcal{L}}, \bm{U}^t_{{\parallel}\mathcal{L}} \right \rangle \right]
        =2E\left[ \left \langle  \bm{Q}^{t+1}_{{\parallel}\mathcal{L}}, \bm{U}^t \right \rangle \right]\notag
        =&2E\left[ \left \langle  \bm{Q}^{t+1}, \bm{U}^t \right \rangle \right]-2E\left[ \left \langle  \bm{Q}^{t+1}_{{\perp}\mathcal{L}}, \bm{U}^t\right \rangle \right]\notag\\
        \stackrel{(a)}=&-2E\left[ \left \langle  \bm{Q}^{t+1}_{{\perp}\mathcal{L}}, \bm{U}^t\right \rangle \right]
    \end{align}
    where $(a)$ follows from the definition of unused service.  Using Cauchy–Schwartz inequality, we have
    \begin{align}\label{eq: T5}
        \abs{2E\left[ \left \langle  \bm{Q}^{t+1}_{{\perp}\mathcal{L}}, \bm{U}^t\right \rangle \right]}\leq 2\sqrt{E\left[\norm{\bm{Q}^{t+1}_{{\perp}\mathcal{L}}}^2 \right]E\left[\norm{\bm{U}^t}^2 \right]}
        \leq 2\sqrt{K^{\star\star}N\epsilon}
    \end{align}
    where the last inequality follows from Proposition \ref{prop: upper bound of q_perp}.
    Substitute (\ref{eq: T01}-\ref{eq: T5}) into (\ref{sum t}) and put back the superscript ${(\cdot)}^{(\epsilon)}$ and pick $m(\epsilon)=\lfloor1/\sqrt{\epsilon}\rfloor$, we get
    \begin{align}\label{upper bound}
                &2\left (\frac{1}{N}-6A_{max}C_{max}\frac{\alpha_{max}^{m(\epsilon)}}{\epsilon}\right) E \left[\epsilon\sum_{ij} \overline{Q}^{(\epsilon)}_{ij}\right] \nonumber \\ 
                \leq &2\sum_{ij}\sum_{l=1}^{m(\epsilon)}\left(\frac{1}{N} \sum_{j'} \gamma^{(\epsilon)}_{ij,ij'}(l)+\frac{1}{N} \sum_{i'} \gamma^{(\epsilon)}_{ij,i'j}(l)-\frac{1}{N^2} \sum_{i'j'} \gamma^{(\epsilon)}_{ij,i'j'}(l)  \right)
                +6m(\epsilon)\epsilon   (1+NA_{max}  )\nonumber \\
               &+\epsilon^2+  \sum_{ij}\left(\frac{1}{N} \sum_{j'} \gamma^{(\epsilon)}_{ij,ij'}(0)+\frac{1}{N} \sum_{i'} \gamma^{(\epsilon)}_{ij,i'j}(0)-\frac{1}{N^2} \sum_{i'j'} \gamma^{(\epsilon)}_{ij,i'j'}(0)  \right)+N\epsilon+2\sqrt{K^{\star\star}N\epsilon}, 
    \end{align}
    and
    \begin{align}\label{lower bound}
           &2\left (\frac{1}{N}+6A_{max}C_{max}\frac{\alpha_{max}^{m(\epsilon)}}{\epsilon}\right) E \left[\epsilon\sum_{ij} \overline{Q}^{(\epsilon)}_{ij}\right]\nonumber \\
           \geq&2\sum_{ij}\sum_{l=1}^{m(\epsilon)}\left(\frac{1}{N} \sum_{j'} \gamma^{(\epsilon)}_{ij,ij'}(l)+\frac{1}{N} \sum_{i'} \gamma^{(\epsilon)}_{ij,i'j}(l)-\frac{1}{N^2} \sum_{i'j'} \gamma^{(\epsilon)}_{ij,i'j'}(l)  \right)
           -6m(\epsilon)\epsilon   (1+NA_{max}  ) \nonumber \\
           &+\epsilon^2+  \sum_{ij}\left(\frac{1}{N} \sum_{j'} \gamma^{(\epsilon)}_{ij,ij'}(0)+\frac{1}{N} \sum_{i'} \gamma^{(\epsilon)}_{ij,i'j}(0)-\frac{1}{N^2} \sum_{i'j'} \gamma^{(\epsilon)}_{ij,i'j'}(0)  \right) -2\sqrt{K^{\star\star}N\epsilon}.
    \end{align}
    Let $\epsilon \to 0$,  we have
    \begin{equation*}
    {}
    \begin{aligned}
        &\lim_{\epsilon \to 0}\epsilon E\left[ \sum_{ij}\overline{Q}^{(\epsilon)}_{ij} \right] \\
        &= \lim_{\epsilon \to 0} \sum_{ij}\left(\frac{1}{N} \sum_{j'} \gamma^{(\epsilon)}_{ij,ij'}(0)+\frac{1}{N} \sum_{i'} \gamma^{(\epsilon)}_{ij,i'j}(0)-\frac{1}{N^2} \sum_{i'j'} \gamma^{(\epsilon)}_{ij,i'j'}(0)\right)\\
        &+ 2\sum_{ij}\sum_{t=1}^{\frac{1}{\lfloor1/\sqrt{\epsilon}\rfloor}}\left(\frac{1}{N} \sum_{j'} \gamma^{(\epsilon)}_{ij,ij'}(l)+\frac{1}{N} \sum_{i'} \gamma^{(\epsilon)}_{ij,i'j}(l)-\frac{1}{N^2} \sum_{i'j'} \gamma^{(\epsilon)}_{ij,i'j'}(l)  \right) .     
    \end{aligned}
    \end{equation*}
    The proof is now complete after using Claim \ref{claim: 3.2.3}.
\end{proof}

\section{Conclusion}\label{sec:conclusion}
In this paper, we {{}analyzed} the heavy-traffic  behavior in a switch operating under the MaxWeight scheduling algorithm when the arrivals are Markovian.
The steady state sum queue length was obtained, which is consistent with the result in i.i.d. case. This paper generalized the drift method that was developed in \cite{eryilmaz2012asymptotically,maguluri2016heavy} and the transform method in \cite{hurtado2020transform} to the case of Markovian arrivals. The key ideas are to consider drift over a time window whose size depends on the heavy-traffic parameter, and to exploit geometric mixing of Markov chains to get a handle on the Markovian correlations.

There are several possible future directions. An immediate future direction is to generalize the result to the so-called `Generalized switch model' \cite{stolyar2004maxweight,langegenswitch} when the complete resource pooling condition is not satisfied. This is a general queueing model that includes a switch that is incompletely saturated, a switch where the arrivals across the ports are saturated, wireless networks under interference and fading, cloud computing scheduling etc. A second future direction is to establish the generality of Markovian arrivals in discrete-time systems. It was shown in \cite{asmussen1993marked} that Markovian arrival processes approximate any marked point process to an arbitrary degree of accuracy. Exploring the validity of such a result for discrete-time arrival processes is an open question.  
\section*{Acknowledgments}
The authors thank Dr. Daniela Hurtado Lange for her discussions. This work was partially supported by NSF grants EPCN-2144316, CMMI-2140534 and CCF-1850439.
\bibliographystyle{APT}
\bibliography{APT_final_version/APT_bib} 

\renewcommand{\theHsection}{A\arabic{section}}
\begin{APPENDICES}

\section{Heavy-traffic Limit of Queue Length Distribution}\label{Appd, queue lenth distribution}
In this subsection, we will obtain the heavy-traffic limiting steady state distribution of $\epsilon \overline{Q}^{(\epsilon)}$. One way of doing this is to use $q^{t+1}$ as the test function to obtain the $k^{\text{th}}$ moment of the queue length. Once all the moments are obtained, the distribution can be inferred. Such an approach was used in \cite{eryilmaz2012asymptotically}. Here, we instead using the transform method that was presented in \cite{hurtado2020transform}. The key contribution is to extend the result in \cite{hurtado2020transform} to the case of Markovian arrivals in a single server queue.
\begin{theorem}\label{thm: exp}
Consider the same setting in Theorem \ref{thm: pos rec single server queue}.
Let $\overline{Q}^{  (\epsilon  )}$ be a steady-state random variable to which the queue length processes $\left\{Q^t\right\}^{  (\epsilon  )}_{t \geq 1}$ converges in distribution. 
Then for any $\theta \leq 0$,
\begin{equation*}
    \lim_{\epsilon \to 0}E  \left[e^{\epsilon\theta \overline{Q}^{(\epsilon)}}  \right] = \frac{1}{1-\theta\frac{\sigma_a^2+\sigma_s^2}{2}}.
\end{equation*}
    Therefore, we have that $\epsilon \overline{Q}^{(\epsilon)}$ converges in distribution to an exponential variable with mean $\frac{\sigma_a^2+\sigma_s^2}{2}$.
\end{theorem}
The proof is presented in Appendix \ref{Appd: Heavy-traffic Limit Distribution for Queue Length}. The key idea is to consider the  $m$-step  drift of the exponential test function,   $e^{\epsilon\theta q}$. As in the proof of Theorem \ref{thm: pos rec single server queue}, we again pick $m$ as a function of $\epsilon$, and exploit the  mixing rate of the underlying Markov chain. In addition, we use the following lemma to compare the arrival process with an independent Markovian process with the same transition probabilities. This lemma, which is proved in Appendix \ref{Proof of Lemma: comparing} enables us to asymptotically decouple of the queue length and arrival process in the heavy-traffic regime. 
\begin{lemma}\label{lem: comparing}
        For every $\epsilon \in (0,\mu)$, let $\left\{Y^t\right\}^{(\epsilon)}_{t\geq0}$ be a Markov chain that is independent of the chain $\left\{X^t\right\}^{(\epsilon)}_{t \geq 0}$, but is defined on the same state space $\Omega$, and with the same transition probability, and consequently has the same stationary distribution, i.e. $\overline{Y}^{  (\epsilon  )}\stackrel{d}=\overline{X}^{  (\epsilon  )} \sim \pi^{(\epsilon)}$. 
        Let $\left(b^t\right)^{(\epsilon)}=f  \left(\left(Y^t\right)^{(\epsilon)} \right )$ and ${\left(Y^0\right)}^{(\epsilon)} \sim \pi^{(\epsilon)}$. Then,
        {\scriptsize
        \begin{align*}
            & \Bigg\vert \sum_{l=0}^{m-1}  \bigg\{ E  \left [  \left (\left(A^{t+l}\right)^{(\epsilon)}-\lambda^{(\epsilon)}  \right)  \left (\left(A^{t+m}\right)^{(\epsilon)}-\lambda^{(\epsilon)}  \right) -  \left(\left(b^{m-l}\right)^{(\epsilon)}-\lambda^{(\epsilon)}  \right)  \left (\left(b^{0}\right)^{(\epsilon)}-\lambda^{(\epsilon)}  \right) \mid \left(X^{t}\right)^{(\epsilon)} \right]  \bigg\}\Bigg\vert\\
            & \leq 4  \left (A_{max}+\lambda^{(\epsilon)} \right)A_{max}\left (C^{(\epsilon)}\right)^2m\left(\alpha^{(\epsilon)}\right)^m,
        \end{align*}
        }
        and
        {\footnotesize\begin{align*}
            \abs{E  \left[  \left(\left(A^{t+m}\right)^{(\epsilon)}-\lambda^{(\epsilon)}  \right)^2-  \left(\left(b^0\right)^{(\epsilon)}-\lambda^{(\epsilon)}  \right)^2 \mid \left(X^{t}\right)^{(\epsilon)} \right ]}  \leq 2 \left (A_{max}+\lambda^{(\epsilon)}  \right)^2\left(C^{(\epsilon)}\right)\left(\alpha^{(\epsilon)}\right)^m.
        \end{align*}}
\end{lemma}

\section{Proof of Lemmas in Section \ref{sec:preliminaries}}
\subsection{Proof of Lemma \ref{Lemma: $m$-step bound for q_perp_norm_2}} \label{proof of Lemma: $m$-step bound for q_perp_norm_2}
    \begin{proof}  {of Lemma \ref{Lemma: $m$-step bound for q_perp_norm_2}.  } Let $\theta \in (0,1)$, applying Taylor expansion and using Condition 2 in the Lemma, we have
       \begin{align}\label{B4.1}
              & E \left[e^{\theta \Delta^m Z  (\bm{q},\bm{x}  )}\mid   \left(\bm{Q}^t, \bm{X}^t  \right)=  \left(\bm{q},\bm{x}  \right)    \right]\nonumber\\
              & \leq E\left[1+\theta \Delta^m Z  (\bm{q},\bm{x}  )+\sum_{l=2}^\infty \frac{\theta^l D^l}{l!}\mid   \left(\bm{Q}^t, \bm{X}^t  \right)=  \left(\bm{q},\bm{x}  \right)    \right]\\
              &=  1+\theta E\left[ \Delta^m Z  (\bm{q},\bm{x}  )\mid    \left(\bm{Q}^t  ,\bm{X}^t    \right)=  (\bm{q},\bm{x}  )   \right]+\sum_{l=2}^\infty \frac{\theta^l D^l}{l!}.\nonumber
        \end{align}
According to Conditions 1 and 2, we have,
    \begin{equation}\label{B4.2}
    \begin{aligned}
       & E\left[ \Delta^m Z  (\bm{q},\bm{x}  )\mid    \left(\bm{Q}^t,\bm{X}^t    \right)=  (\bm{q},\bm{x}  )   \right] \leq -\eta \mathcal{I}\left(\left(\bm{q},\bm{x}\right) \in \mathcal{B} \right) + D\mathcal{I}\left(\left(\bm{q},\bm{x}\right) \in \mathcal{B}^c \right),
    \end{aligned}
    \end{equation}
    where
    \begin{align*}
        \mathcal{B} = \left\{  \left(\bm{q}  ,\bm{x}    \right) \in   (\mathcal{Q},\mathcal{X}  ):  Z  (\bm{q},\bm{x}  ) \geq \kappa\right\}.
    \end{align*}
Since $\theta \in (0,1)$, we have,
\begin{align}\label{B4.3}
    \sum_{l=2}^\infty \frac{\theta^l D^l}{l!} = \theta^2 \sum_{l=2}^\infty \frac{\theta^{l-2} D^l}{l!}\leq \theta^2 \sum_{l=2}^\infty \frac{D^l}{l!} = \theta^2(e^D-1-D).
\end{align}
Combining (\ref{B4.1}-\ref{B4.3}), we have,
        \begin{align}\label{B4.4}
            & E \left[e^{\theta \Delta^m Z  (\bm{q},\bm{x}  )}\mid   \left(\bm{Q}^t, \bm{X}^t  \right)=  \left(\bm{q},\bm{x}  \right)    \right]\nonumber\\
              & \leq   \left[1-\theta \eta +\theta^2  \left(e^D -1-D  \right)  \right]\mathcal{I}  \left(  \left(\bm{q}, \bm{x}  \right) \in \mathcal{B}  \right)\nonumber\\
              &\quad +\left[1+\theta D +\theta^2  \left(e^D -1-D  \right) \right]\mathcal{I}  \left(  \left(\bm{q}  ,\bm{x}    \right) \in \mathcal{B}^c  \right)\nonumber\\
              & \leq \delta_1 \mathcal{I}\left(  \left(\bm{q}  ,\bm{x}   \right) \in \mathcal{B}  \right)+\delta_2 \mathcal{I}  \left(  \left(\bm{q}  ,\bm{x}    \right) \in \mathcal{B}^c  \right),
       \end{align}
       where 
       \begin{align*}
            \delta_1 = 1-\theta \eta +\theta^2  \left(e^D -1-D  \right) <1
       \end{align*}
       for sufficient small $\theta \leq \frac{1}{\eta}$ and
       \begin{align*}
            \delta_2 = 1+\theta D +\theta^2  \left(e^D -1-D  \right). 
       \end{align*}
    According to (\ref{B4.4}), we have
    \begin{align*}
        & E   \left[ e^{\theta Z  (\bm{Q}  (t+m  ),\bm{X}  (t+m  )  )}     \right] \\
        =&\sum_{  (\bm{q},\bm{x}  ) \in (\mathcal{Q},\mathcal{X}  )} E   \left[ e^{\theta \Delta^m Z  (\bm{q},\bm{x}  )} e^{\theta Z  \left(\bm{Q}^t  ,\bm{X}^t    \right)}\mid   \left(\bm{Q}^t  ,\bm{X}^t    \right)=  (\bm{q},\bm{x}  )    \right] P  \left(  \left(\bm{Q}^t  ,\bm{X}^t    \right)=  (\bm{q},\bm{x}  )  \right)\\
        = &\sum_{(\bm{q},\bm{x}  ) \in \mathcal{B}} e^{\theta Z  \left(\bm{q}  ,\bm{x}    \right)}E   \left[ e^{\theta \Delta^m Z  (\bm{Q}^t,\bm{X}^t  )} \mid   \left(\bm{Q}^t  ,\bm{X}^t    \right)=  (\bm{q},\bm{x}  )    \right] P  \left(  \left(\bm{Q}^t  ,\bm{X}^t    \right)=  (\bm{q},\bm{x}  )  \right) \\
        & +\sum_{(\bm{q},\bm{x}  ) \in \mathcal{B}^c}  e^{\theta Z  \left(\bm{q}  ,\bm{x}    \right)}E   \left[ e^{\theta \Delta^m Z  (\bm{Q}^t,\bm{X}^t  )}\mid   \left(\bm{Q}^t  ,\bm{X}^t    \right)=  (\bm{q},\bm{x}  )    \right] P  \left(  \left(\bm{Q}^t  ,\bm{X}^t    \right)=  (\bm{q},\bm{x}  )  \right) \\
        \leq & \delta_1 \sum_{  (\bm{q},\bm{x}  ) \in (\mathcal{Q},\mathcal{X}  )} e^{\theta Z  \left(\bm{q}  ,\bm{x}    \right)} P  \left(  \left(\bm{Q}^t  ,\bm{X}^t    \right)=  (\bm{q},\bm{x}  )  \right)\\
        &+ \sum_{\mathcal{B}^c}  \left(\delta_2-\delta_1  \right)e^{\theta Z  \left(\bm{Q}^t  ,\bm{X}^t    \right)} P  \left(  \left(\bm{Q}^t  ,\bm{X}^t    \right)=  (\bm{q},\bm{x}  )  \right)\\
         \leq & \delta_1  E   \left[e^{\theta Z  \left(\bm{Q}^t  ,\bm{X}^t    \right) }    \right] +   \theta(D+\eta)e^{\theta \kappa}.
    \end{align*}
    By induction, we have  $ \forall t_0 \in [0, m]$ and $l \in \mathbb{N}_+$,
    \begin{align*}
           & E   \left[ e^{\theta Z  \left(\bm{Q}^{t_0+lm} ,\bm{X}^{t_0+lm}    \right)}     \right] \leq \delta_1^l E   \left[ e^{\theta Z  \left(\bm{Q}^{t_0}  ,\bm{X}^{t_0}    \right)}     \right]+\frac{1-\delta_1^l}{1-\delta_1}  \theta\left(D+\eta \right)e^{\theta \kappa}.
    \end{align*}
    {Note that $E\left[ e^{\theta Z  \left(\bm{Q}^{t_0}  ,\bm{X}^{t_0}    \right)} \right] \leq e^{\theta \hat{D}(t_0)} <\infty$ for all $t_0 \in[0,m]$, according to Condition 3.} Since $\{t \in \mathbb{R}: t \geq 0 \} = \{t = t_0+l m: t_0 \in[0,m]\ \text{and}\ l \in \mathbb{N}_+ \}$, we have,
    \begin{align*}
        \limsup_{t \to \infty} E   \left[ e^{\theta Z  \left(\bm{Q}^t  ,\bm{X}^t    \right)}    \right] \leq \frac{  \theta\left(D+\eta \right)e^{\theta \kappa}}{1-\delta_1} \triangleq C^\star.
    \end{align*}
    
    \end{proof}

\subsection{Proof of Lemma \ref{lem: geo}} \label{Proof: lem: geo}
\begin{proof}  {of Lemma \ref{lem: geo}.  } Define the total variance difference between two distribution $\pi_1$ and $\pi_2$ on $\Omega$ as
    \begin{align}\label{eq: geo1}
        \norm{\pi_1-\pi_2} _{TV}=\frac{1}{2}\sum_{x \in \Omega}\abs{ \pi_1  (x  )-\pi_2  (x  )}.
    \end{align}
    It has been shown in \cite{levin2017markov} (Theorem 4.9 of Chapter 4, Page 52) that, for an irreducible, positive recurrent and aperiodic finite-state Markov chain  with transition probability matrix $P$ and stationary distribution $\pi$, there exist constants  $\alpha \in   (0,1  )$ and $C > 0$ such that $\forall m \in \mathbb{N}_+$,
    \begin{align}\label{eq: geo2}
        \max_{x \in \Omega}\norm{P  (x,\cdot  )-\pi}_{TV} \leq C\alpha^m.
    \end{align}
    Thus, for any initial distribution $X^0 \sim \pi^0$
    \begin{align*}
      &\abs{E \left[f(X^m)-\lambda \right]}\\
      \leq &\sum_{y \in \Omega}f  (y  )\abs{  \left(\pi^0P^m  \right)  (y  )-\pi  (y  )} \\
      \leq &\sum_{y \in \Omega}f  (y  )\sum_{x \in \Omega}\pi^0  (x  )\abs{P^m  (x,y  )-\pi  (y  )}\\
      \leq& L\sum_{y \in \Omega}\sum_{x \in \Omega}\pi^0  (x  )\abs{P^m  (x,y  )-\pi  (y  )}\\
      =&L\sum_{x \in \Omega}\pi^0  (x  )\sum_{y \in \Omega}\abs{P^m  (x,y  )-\pi  (y  )}\\
      \leq & 2L\sum_{x \in \Omega}\pi^0  (x  )\max_{x \in \Omega}\norm{P  (x,\cdot  )-\pi}_{TV}\\
      \leq & 2LC\alpha^m.
    \end{align*}
    Thus,
    \begin{align*}
        \abs{E  \left[f(X^m)-\lambda \right]}\leq 2LC\alpha^m.
    \end{align*}
\end{proof}

\subsection{Proof of Lemma \ref{lem: finite}} \label{Proof of Lemma: finite}
    \begin{proof}  {of Lemma \ref{lem: finite}.  }
    \begin{align*}
      \abs{\gamma(t)}= &\abs{  E  \left[  \left(f\left(X^t\right)-\lambda  \right)  \left(f\left(X^0\right)-\lambda  \right)  \right]  }\\
      =&\abs{  E_{X^0 \sim \pi }  \left[  \left(f\left(X^0\right)-\lambda  \right)E_{X^0 \sim \pi }  \left[  \left(f\left(X^t\right)-\lambda  \right) \right] \right]  }\\
      \stackrel{(a)}{\leq}&E_{X^0 \sim \pi }  \left[\abs{  \left(f\left(X^0\right)-\lambda  \right)E_{X^0 \sim \pi }  \left[  \left(f\left(X^t\right)-\lambda  \right)  \right]  }  \right] \\  
      =&E_{X^0 \sim \pi }  \left[\abs{f\left(X^0\right)-\lambda  } \abs{E_{X^0 \sim \pi }  \left[  \left(f\left(X^t\right)-\lambda  \right)  \right]  }  \right]\\
      \leq &2  (A_{max}+\lambda  )A_{max}C\alpha^k 
    \end{align*}
    
    where $(a)$ follows from Jensen's  and the last  follows from Lemma \ref{lem: geo}. Let
    $V_1 = \gamma(0)+2\lim_{m \to
        \infty}\sum_{t=1}^{m}\frac{m-t}{m}\gamma(t) $ and $V_2= \gamma(0)+2\lim_{m \to \infty}\sum_{t=1}^{m}\gamma(t)$,
    consider 
    \begin{align*}
            \abs{V_1-V_2}  =&\lim_{m\to \infty}\abs{\sum_{t=1}^m \frac{t}{m} \gamma(t)}\\
            =&\lim_{m\to \infty}\abs{\sum_{t=1}^m \frac{m-t}{m} \gamma(t)-\sum_{t=1}^m\gamma(t)}\\
            & \leq \lim_{m\to \infty}\sum_{t=1}^m \frac{t}{m} \abs{\gamma(t)}\\
            & \leq \lim_{m\to \infty}2(A_{max}+\lambda  )A_{max}C \sum_{t=1}^m \frac{t}{m}\alpha^t\\
            & \leq \lim_{m\to \infty}2(A_{max}+\lambda  )A_{max}C \left(\frac{\alpha-\alpha^{m+1}}{m(1-\alpha^2)} -\frac{\alpha^{m+1}}{1-\alpha} \right).
    \end{align*}
    Since $\alpha \leq 1$, we have $\abs{V_1-V_2} \leq 0$. Thus, $V_1=V_2$.
    \end{proof}
   
\section{Proof of Claim \ref{claim: interchange of limit}}\label{proof_Claim_interchange_of_limit}
{{}
We will prove the claim using dominated convergence theorem to justify a certain interchange of limit and summation.
    \begin{align*}
           \abs{ \gamma^{(\epsilon)}(t)} &=\abs{ E \left[  \left(\left(A^t\right)^{(\epsilon)}-\lambda  \right)  \left(\left(A^0\right)^{(\epsilon)}-\lambda  \right) \right]}\\
           & \leq E\left[\abs{\left( \left(A^0\right)^{(\epsilon)}-\lambda^{(\epsilon)} \right)} \abs{E \left[ \left( \left(A^t\right)^{(\epsilon)}-\lambda^{(\epsilon)} \right)\vert X^0 \right] } \right]\\
           & \leq (A_{max}+\lambda) 2A_{max}C\alpha^t.
    \end{align*}
    where the last inequality follows Lemma \ref{lem: geo}.
    Since 
     \begin{align*}
           \lim_{m\to\infty}\sum_{t=1}^m (A_{max}+\lambda) 2A_{max}C\alpha^t &= \lim_{m\to\infty}2(A_{max}+\lambda)A_{max}C\alpha\frac{1-\alpha^{m-1}}{1-\alpha}\\
           &\leq 2(A_{max}+\lambda)A_{max}C\frac{\alpha}{1-\alpha}< \infty,
    \end{align*}
    we conclude that the interchange of limit is solid, i.e.,
    \begin{align*}
           \lim_{\epsilon \to 0}\lim_{m \to \infty}\sum_{t=1}^{m}\gamma^{(\epsilon)}(t) = \lim_{m \to \infty}\sum_{t=1}^{m}\lim_{\epsilon \to 0}\gamma^{(\epsilon)}(t) = \lim_{m \to \infty}\sum_{t=1}^{m}\gamma(t).
    \end{align*}
    Therefore, $\lim_{\epsilon \to 0} \left(\sigma_a^{(\epsilon)}\right)^2 = \sigma_a^2.$}

\section{Details in proof of Theorem \ref{thm: pos rec single server queue}}
\subsection{Proof of Claim \ref{claim 3.2.1}.}\label{proof_Claim_3.2.1}
\begin{proof}{}
We first use the queue evolution equation \eqref{eq:1}  to recursively expand $Q^{t+m}$.                  
\begin{align*}
          & E  \left[\left(Q^{t+m}\right)^2-\left(Q^t\right)^2\mid   \left(Q^t,X^t  \right)=  (q,x  )   \right]\\
           \leq & E    \left[  (Q^{t+m-1}+A^{t+m-1}+S^{t+m-1}  )^2-\left(Q^t\right)^2\mid   \left(Q^t,X^t  \right)=  (q,x  )   \right]\\
           = &E   \bigg[\left(Q^{t+m-1}\right)^2+2Q^{t+m-1}  (A^{t+m-1}-S^{t+m-1}  ) \\
           &+  (A^{t+m-1}-S^{t+m-1}  )^2-\left(Q^t\right)^2  ]\mid   \left(Q^t,X^t  \right)=  (q,x  )   \bigg]\\
           = &E\Bigg  [2\left(Q^t+\sum_{i=0}^{m-2}A^{t+m-i}-\sum_{i=0}^{m-2}S^{t+m-i}
          +\sum_{i=0}^{m-2}U^{t+m-i}\right)\left(A^{t+m-1}-S^{t+m-1}\right) \\
          & +\left(Q^{t+m-1}\right)^2+  (A^{t+m-1}-S^{t+m-1}  )^2-\left(Q^t\right)^2\mid   \left(Q^t,X^t  \right)=  (q,x )\Bigg  ]\\
           \leq &E  \big[2Q^t  (A^{t+m-1}-S^{t+m-1}  )+2m  (A_{max}+S_{max}  )  (A_{max}+S_{max}  )\\
          & +  \left(Q^{t+m-1}\right)^2+(A_{max}+S_{max}  )^2-\left(Q^t\right)^2\mid   \left(Q^t,X^t  \right)=  (q,x  )  \big]\\
            = &E  \big[\left(Q^{t+m-1}\right)^2+2Q^t  (A^{t+m-1}-S^{t+m-1}  )\\
            &-\left(Q^t\right)^2+K_0  (m  )\mid   \left(Q^t,X^t  \right)=  (q,x  )  \big],
\end{align*}
where 
\begin{align*}
    K_0  (m ) = 2m  (A_{max}+S_{max}  )  (A_{max}+S_{max}  )+  (A_{max}+S_{max}  )^2.
\end{align*}
By induction we have
\begin{align*}
       & E  \left[\left(Q^{t+m}\right)^2-\left(Q^t\right)^2\mid   \left(Q^t,X^t  \right)=  (q,x  )  \right]\\
       & \leq E   \left[2Q^t   \left(\sum_{i=1}^{m}A^{t+m-i}-\sum_{i=1}^{m}S^{t+m-i}   \right)+m K_0  (m  )\mid   \left(Q^t,X^t  \right)=  (q,x  )   \right]\\
       & = E  \left [2Q^t\sum_{i=1}^{m}  (A^{t+m-i}-\lambda  )-2m\epsilon Q^t+mK_0  (m  )\mid   \left(Q^t,X^t  \right)=  (q,x  )   \right].
\end{align*}
\end{proof}

\subsection{Proof of Claim \ref{claim: q^2 finite}}\label{proof q^2 finite}
Fix $\epsilon$, let's take $Z(Q^t,X^t) = Q^t$ as the test function, 
\begin{align*}
        \Delta^{m(\epsilon)} Z(q,x)  &= \left(Q^{t+m(\epsilon)}-Q^t\right)\mathcal{I}(Q^t = q) \\
         & = \left(\sqrt{\left(Q^{ t+m(\epsilon)}\right)^2}-\sqrt{\left(Q^{ t}\right)^2}\right)\mathcal{I}(Q^t = q)\\
         &  \leq \frac{1}{2Q^{t}} \left(\left(Q^{ t+m(\epsilon)}\right)^2- \left(Q^{ t}\right)^2\right)\mathcal{I}(Q^t = q)
\end{align*}
where the  follows from the fact that $f(x)= \sqrt{x}$ is a concave function for $x\geq 0$ so that $f(y)-f(x) \leq (y-x)f'(x) = \frac{y-x}{2\sqrt{x}}$ with $y = \left(Q^{t+m(\epsilon)}\right)^2$ and $x = \left(Q^{t}\right)^2$.  
Therefore, 
\begin{align*}
    &E  \left[\Delta^{m(\epsilon)} Z(q,x)\mid   \left(Q^t,X^t  \right)=  (q,x  )  \right]\\
       &\leq E  \left[\frac{1}{2Q^{t}} \left(\left(Q^{ t+m(\epsilon)}\right)^2- \left(Q^{ t}\right)^2\right)\mid   \left(Q^t,X^t  \right)=  (q,x  )  \right]\\
       &\leq -K_1(\epsilon) +\frac{1}{{2Q^{t}}}m(\epsilon) K_0  (m(\epsilon)  ),
\end{align*}
Let ${\kappa(m(\epsilon))} =  \frac{m(\epsilon) K_0  (m(\epsilon)  )}{K_1({\epsilon})}$ and $\eta = \frac{ K_1(\epsilon)}{2}$. Then, we have
 for all $  (q,x) \in   (\mathcal{Q},\mathcal{X}  )$ with $Z  (q,x  ) \geq {\kappa(m(\epsilon))}$,
        \begin{align*}
            E  \left[\Delta^m Z  (q, x  )\mid   \left(Q^t,X^t  \right) =   (q,x  )  \right] \leq -{\eta(m(\epsilon))}.
        \end{align*}
Moreover, we have $P\left(\Delta^{m(\epsilon)} Z(Q,X) \leq {m(\epsilon)}(A_{max}+S_{max})\right)=1$. Using Lemma \ref{Lemma: $m$-step bound for q_perp_norm_2}, we conclude that 
\begin{align*}
    E  \left[  \left(\overline{Q}^{  (\epsilon  )}  \right)^2  \right] < \infty.
\end{align*} 

\subsection{Proof of Claim \ref{claim: 3.2.2}}\label{proof Claim 3.2.2}
\begin{proof}{}
We will prove the claim by bounding the terms on the RHS of (\ref{eq: qe}).     Firstly, setting drift of $Q^t$ to zero in steady-state, we have $\forall i \in \mathbb{N}_+$
    \begin{align}\label{ E(u)=epsilon}
     E  \left[U^{t+i}  \right]=E  \left(Q^{t+i+1}-Q^{t+i}+S^{t+i}-A^{t+i}  \right)=\mu-\lambda=\epsilon.
    \end{align}
    Thus,
    \begin{align}
       -\epsilon {{}S_{max}}= -  (\mu-\lambda  )S_{max}\leq E  \left[-U^t S^t  \right]\leq E  \left[-\left(U^t\right)^2  \right] \leq 0. \label{eq:u2}
    \end{align}
    This bounds one of the terms in (\ref{eq: qe}). Also note that this implies the following bound, which we will use shortly. 
    \begin{align}
      \abs{\sum_{i=1}^{m}E  [2U^{t+m-i}  \left(A^{t+m}-\lambda \right )  ]}\leq 2m  (A_{max}+\lambda  ) \epsilon. \label{eq:qa-s}        
    \end{align}

Now we will bound the first term on the RHS of (\ref{eq: qe}) which is the most challenging term. 
Notice that, 
\begin{align*}
    E  \left[2Q^t  (A^t-\lambda  )  \right] \leq 2(A_{max}+\lambda) E  \left[Q^t\right] < \infty.
\end{align*}
Since we are in the steady state, we have $\forall m \in \mathbb{N}_+$, 
   \begin{align}
        E  \left[2Q^t  (A^t-\lambda  )  \right] = E  \left[2Q^{t+m}  \left(A^{t+m}-\lambda \right )  \right].
   \end{align}
    According to (\ref{eq:1}), we have 
    \begin{align}
            & E  \left[2Q^{t+m}  \left(A^{t+m}-\lambda \right )  \right] \nonumber\\
            =& E  \left[2  (Q^{t+m-1}+A^{t+m-1}-S^{t+m-1}+U^{t+m-1}  )  \left(A^{t+m}-\lambda \right )  \right]\\
            =& E  \left[2Q^{t+m-1}  \left(A^{t+m}-\lambda \right )  \right]+E  \left[2  (A^{t+m-1}-\lambda  )  \left(A^{t+m}-\lambda \right )  \right]\nonumber\\
            &+E  \left[2U^{t+m-1}  \left(A^{t+m}-\lambda \right )  \right]\nonumber,
    \end{align}
    where the last equality uses the independence of $S^{t+m-1}$ and $A^{t+m}$. By repeating these steps recursively, we have, $\forall m \in \mathbb{N}_+$,
    \begin{align}\label{eq: proof claim 3.2.2}
        & E  \left[2Q^t  (A^t-\lambda  )  \right] \nonumber\\
        =& E  \left[2Q^t  \left(A^{t+m}-\lambda \right )  \right] + 2\sum_{i=1}^{m}\gamma(i)+ \sum_{i=1}^{m}E  \left[2U^{t+m-i}  \left(A^{t+m}-\lambda \right )  \right],
    \end{align}
    where
    \begin{align*}
    \gamma(i) = E  \left[  \left(A^{t+m-i}-\lambda  \right)  \left(A^{t+m}-\lambda  \right)  \right]
        = E  \left[  \left(A^{t+i}-\lambda  \right)  \left(A^{t}-\lambda  \right)  \right].       
    \end{align*}

   According to Lemma \ref{lem: geo},
   \begin{align*}
         &\abs{ E  \left[2Q^t  \left(A^{t+m}-\lambda \right )  \right]} \leq E  \left[2Q^t \abs{ E  \left[  \left(A^{t+m}-\lambda  \right)\mid   (X^t,Q^t  )   \right]}  \right] \leq 4A_{max}C\alpha^m E  [Q^t  ].     
   \end{align*}
   Thus, we have from (\ref{eq: proof claim 3.2.2}) and (\ref{eq:qa-s}),
    \begin{align*}
        &-4A_{max}C\alpha^m E  [2Q^t  ] - 2m  (A_{max}+\lambda  ) \epsilon+2\sum_{i=1}^{m}\gamma(i)\\
        \leq& E  [2Q^t  (A^t-\lambda  )  ] \\
        \leq&  4A_{max}C\alpha^m E  [2Q^t  ] + 2m  (A_{max}+\lambda  ) \epsilon+2\sum_{i=1}^{m}\gamma(i).
    \end{align*}
    Now, substitute this and (\ref{eq:u2}) in (\ref{eq: qe}), we have, $\forall m \in \mathbb{N}_+$,
    \begin{align*}
             2\left(1-2A_{max}C\frac{\alpha^m}{\epsilon}\right) E\left[\epsilon Q^t\right] &
            \leq \gamma(0) + 2\sum_{i=1}^{m}\gamma(i)  + \sigma_s^2 + 2m  (A_{max}+\lambda  ) \epsilon +   \epsilon^2 \\
             2\left(1+2A_{max}C\frac{\alpha^m}{\epsilon}\right) E\left[\epsilon Q^t\right] &
             \geq \gamma(0) + 2\sum_{i=1}^{m}\gamma(i)   + \sigma_s^2- 2m  (A_{max}+\lambda  )\epsilon -  S_{max}\epsilon +   \epsilon^2.
    \end{align*}    
\end{proof}
\subsection{Proof of Claim \ref{claim: 3.2.3}}\label{proof: claim 3.2.3}
\begin{proof}{}
We will prove the claim using dominated convergence theorem to justify a certain interchange of limit and summation. Towards this end, note that
        for any $ \epsilon>0$ and an integer $M \geq \lceil \frac{1}{\sqrt{\epsilon}} \rceil$, since $X^0 \stackrel{d}{=}\overline{X}^{  (\epsilon  )}$, we have,
    \begin{align*}
           &  \Bigg\vert\sum_{i=1}^{M} E\left[ \left( \left(A^i\right)^{(\epsilon)}-\lambda^{(\epsilon)} \right) \left( \left(A^0\right)^{(\epsilon)}-\lambda^{(\epsilon)} \right) \right] - \sum_{i=1}^{\left\lfloor \frac{1}{\sqrt{\epsilon}} \right\rfloor } E\left[ \left( \left(A^i\right)^{(\epsilon)}-\lambda^{(\epsilon)} \right) \left( \left(A^0\right)^{(\epsilon)}-\lambda^{(\epsilon)} \right)\right]\Bigg\vert\\
           & = \Bigg\vert\sum_{i=\lceil \frac{1}{\sqrt{\epsilon}} \rceil}^{ M} E\left[ \left( \left(A^i\right)^{(\epsilon)}-\lambda^{(\epsilon)} \right) \left( \left(A^0\right)^{(\epsilon)}-\lambda^{(\epsilon)} \right) \right]\Bigg\vert\\
           & = \Bigg\vert\sum_{i=\left\lceil \frac{1}{\sqrt{\epsilon}} \right\rceil}^{ M} E\left[\left( \left(A^0\right)^{(\epsilon)}-\lambda^{(\epsilon)} \right) E \left[ \left( \left(A^i\right)^{(\epsilon)}-\lambda^{(\epsilon)} \right) \right] \right]\Bigg\vert\\
           & \leq \sum_{i=\lceil \frac{1}{\sqrt{\epsilon}} \rceil}^{ M} E\left[\abs{\left( \left(A^0\right)^{(\epsilon)}-\lambda^{(\epsilon)} \right)} \abs{E \left[ \left( \left(A^i\right)^{(\epsilon)}-\lambda^{(\epsilon)} \right) \right]} \right]\\
            & \stackrel{(a)}\leq \sum_{i=\lceil \frac{1}{\sqrt{\epsilon}} \rceil}^{ M} (A_{max}+\lambda) 2A_{max}C\alpha^t\\
           & \leq 2A_{max}C(A_{max}+\lambda)\frac{\alpha^{\frac{1}{\sqrt{\epsilon}}}}{1-\alpha}.
    \end{align*}
    where we used Lemma \ref{lem: geo} to get ($a$).
    Let $M \to \infty$, for any  $ \epsilon>0$, since $X^0 \stackrel{d}{=}\overline{X}^{  (\epsilon  )}$, we have,
    \begin{align}\label{eq: proof claim 3.2.3,1}
            & \sum_{i=1}^{\lfloor \frac{1}{\sqrt{\epsilon}} \rfloor } E\left[ \left( \left(A^i\right)^{(\epsilon)}-\lambda^{(\epsilon)} \right) \left( \left(A^0\right)^{(\epsilon)}-\lambda^{(\epsilon)} \right)\right] \\
            &\geq 
            \lim_{M \to \infty}\sum_{i=1}^{M} E\left[ \left( \left(A^i\right)^{(\epsilon)}-\lambda^{(\epsilon)} \right) \left( \left(A^0\right)^{(\epsilon)}-\lambda^{(\epsilon)} \right) \right] - 2A_{max}C(A_{max}+\lambda)\frac{\alpha^{\frac{1}{\sqrt{\epsilon}}}}{1-\alpha}\nonumber
    \end{align}
    and 
    \begin{align}\label{eq: proof claim 3.2.3,2}
            &  \sum_{i=1}^{\left\lfloor \frac{1}{\sqrt{\epsilon}} \right\rfloor } E\left[ \left( \left(A^i\right)^{(\epsilon)}-\lambda^{(\epsilon)} \right) \left( \left(A^0\right)^{(\epsilon)}-\lambda^{(\epsilon)} \right) \right]\\ 
            &\leq 
            \lim_{M \to \infty} \sum_{i=1}^{M} E\left[ \left( \left(A^i\right)^{(\epsilon)}-\lambda^{(\epsilon)} \right) \left( \left(A^0\right)^{(\epsilon)}-\lambda^{(\epsilon)} \right) \right]+ 2A_{max}C(A_{max}+\lambda)\frac{\alpha^{\frac{1}{\sqrt{\epsilon}}}}{1-\alpha}.\nonumber
    \end{align}
    Let $\epsilon \to 0$ and combine (\ref{eq: proof claim 3.2.3,1}) and (\ref{eq: proof claim 3.2.3,2}) we have
    \begin{align*}
       &\lim_{\epsilon \to 0}\sum_{i=1}^{\left\lfloor \frac{1}{\sqrt{\epsilon}} \right\rfloor } E\left[ \left( \left(A^i\right)^{(\epsilon)}-\lambda^{(\epsilon)} \right) \left( \left(A^0\right)^{(\epsilon)}-\lambda^{(\epsilon)} \right) \right] \\
       =& 
        \lim_{\epsilon \to 0} \lim_{M \to \infty} \sum_{i=1}^{M} E\left[ \left( \left(A^i\right)^{(\epsilon)}-\lambda^{(\epsilon)} \right) \left( \left(A^0\right)^{(\epsilon)}-\lambda^{(\epsilon)} \right) \right].
    \end{align*}
    Notice that 
    \begin{align*}
        &\lim_{M \to \infty}\sum_{i=1}^{M} \abs{E\left[ \left( \left(A^i\right)^{(\epsilon)}-\lambda^{(\epsilon)} \right) \left( \left(A^0\right)^{(\epsilon)}-\lambda^{(\epsilon)} \right) \right]} \leq 
        2A_{max}C(A_{max}+\lambda)\alpha\frac{1}{1-\alpha}.        
    \end{align*}
    It follows from Lebesgue's dominated convergence theorem that
    \begin{align}\label{eq: proof claim3.2.3,3}
            & \lim_{\epsilon \to 0} \lim_{M \to \infty} \sum_{i=1}^{M} E\left[ \left( \left(A^i\right)^{(\epsilon)}-\lambda^{(\epsilon)} \right) \left( \left(A^0\right)^{(\epsilon)}-\lambda^{(\epsilon)} \right) \right]\nonumber\\
            & =  
            \lim_{M \to \infty} \sum_{i=1}^{M} \lim_{\epsilon \to 0} E\left[ \left( \left(A^i\right)^{(\epsilon)}-\lambda^{(\epsilon)} \right) \left( \left(A^0\right)^{(\epsilon)}-\lambda^{(\epsilon)} \right) \right].  
    \end{align}
    According to the weak convergence of the underlying Markov chain (Eq \ref{eq: gamma to gamma}), we have
    \begin{align}\label{eq: proof claim3.2.3,4}
            &\lim_{\epsilon \to 0} E\left[ \left( \left(A^i\right)^{(\epsilon)}-\lambda^{(\epsilon)} \right) \left( \left(A^0\right)^{(\epsilon)}-\lambda^{(\epsilon)} \right)\right] 
            = \gamma(i).
    \end{align}
    Thus, combining (\ref{eq: proof claim3.2.3,3}) and (\ref{eq: proof claim3.2.3,4}), we have
    \begin{align*}
            \lim_{m(\epsilon) \to \infty}\sum_{i=1}^{m(\epsilon)}\gamma^{(\epsilon)}(i) 
            = \lim_{M \to \infty} \sum_{i=1}^{M} \gamma(i).    
    \end{align*}
\end{proof}

\section{Proof of Claim \ref{claim: 4.2.1}}\label{Proof of claim: 4.2.1}
\begin{proof}{}
Similar to the proof for Theorem \ref{thm: pos rec single server queue}, we consider the $m$-stip drift of the Lyapunov function, $V(\cdot)$. For any $m \in \mathbb{N}_+$,
        \begin{align}\label{eq Detailed proof of Claim5}
                &E  \left[\Delta^m V  (\bm{q},\bm{x}  )\mid   \left(\bm{Q}^t  ,\bm{X}^t    \right) =   (\bm{q},\bm{x}  )  \right]\nonumber\\
               \leq &  E  \Bigg[\sum_{l=1}^m\big( 2
               \left\langle \bm{Q} ^{t+m-l },\bm{A}^{t+m-1}  -\bm{\lambda} \right\rangle + \norm{\bm{A}^{t+m-l}  -\bm{S}^{t+m-l}  }^2\nonumber\\
               &+ 2\langle \bm{Q} ^{t+m-l },\bm{\lambda}-\bm{S} ^{t+m-l}   \rangle\big) \mid   \left(\bm{Q}^t  ,\bm{X}^t    \right) =   (\bm{q},\bm{x}  )  \Bigg].
        \end{align}
        The first term in the RHS of \eqref{eq Detailed proof of Claim5} can be simplified as follows,
        \begin{align*}
                &E  \left[ 
               \left\langle \bm{Q} ^{t+m-l },\bm{A}^{t+m-1}  -\bm{\lambda} \right\rangle\mid   \left(\bm{Q}^t  ,\bm{X}^t    \right) =   (\bm{q},\bm{x}  )  \right]\\
                =& E  \left[ \left\langle \bm{Q}^t  +\sum_{i=0}^{m-l}  \left(\bm{A}^{t+i}  -\bm{S}^{ t+i} +\bm{U}^{t+i}    \right),\bm{A}^{t+m-1}  -\bm{\lambda} \right\rangle\mid   \left(\bm{Q}^t  ,\bm{X}^t    \right) =   (\bm{q},\bm{x}  )  \right]\\
                \leq&  \left \langle \bm{q}  ,E  \left[\bm{A}^{t+m-l}  -\bm{\lambda}\mid   \left(\bm{Q}^t  ,\bm{X}^t    \right) =   (\bm{q},\bm{x}  )  \right] \right\rangle+   mN^2  (A_{max}+S_{max}  )  (A_{max}+\lambda_{max}  ) \\
                \stackrel{(a)}\leq& \norm{ \bm{q}}  \norm{E  \left[\bm{A}^{t+m-l}  -\bm{\lambda}\mid   \left(\bm{Q}^t  ,\bm{X}^t    \right] =   (\bm{q},\bm{x}  )  \right]} +   mN^2  (A_{max}+S_{max}  )  (A_{max}+\lambda_{max}  ) \\
                \stackrel{(b)}\leq &2NA_{max}C_{max}\alpha_{max}^{m-l}\norm{\bm{q}}+   mN^2  (A_{max}+S_{max}  )  (A_{max}+\lambda_{max}  ),
        \end{align*}
        where (a) follows from Cauchy–Schwarz Inequality and (b) follows from Lemma \ref{lem: geo}. 
        The second term in the RHS of \eqref{eq Detailed proof of Claim5} can be bounded from above as follows,
        \begin{align*}
                &E  \left[\norm{\bm{A}^{t+m-l}  -\bm{S}^{t+m-l}  }^2\mid   \left(\bm{Q}^t  ,\bm{X}^t    \right) =   (\bm{q},\bm{x}  )  \right] \leq N^2(A_{max}+S_{max})^2.
        \end{align*}
        The last term in the RHS of \eqref{eq Detailed proof of Claim5}  can be written as,
        \begin{align*}
                   & E  \left[\langle \bm{Q} ^{t+m-l },\bm{\lambda}-\bm{S} ^{t+m-l}   \rangle \mid   \left(\bm{Q}^t  ,\bm{X}^t    \right) =   (\bm{q},\bm{x}  )  \right]\\ 
                    =& E  \bigg[ E  \big[\langle\bm{Q} ^{t+m-l },\bm{\lambda}-\bm{S} ^{t+m-l}  \rangle \mid   \left(\bm{Q} ^{t+m-l} ,\bm{X}^{t+m-l}  \right)  \big]\mid   \left(\bm{Q}^t  ,\bm{X}^t    \right) =   (\bm{q},\bm{x}  ) \bigg]\\
                   \stackrel{(a)}\leq & E  \left[\min_{\bm{r}\in  \mathcal{C}}\langle\bm{Q} ^{t+m-l },\bm{\lambda}-\bm{r}\rangle\mid   \left(\bm{Q}^t  ,\bm{X}^t    \right) =   (\bm{q},\bm{x}  )  \right]\\
                    \stackrel{(b)}\leq&  E  \left[\langle\bm{Q} ^{t+m-l },\bm{\lambda}-\left(\bm{\lambda}+\epsilon  \bm{1} \right)\rangle\mid   \left(\bm{Q}^t  ,\bm{X}^t    \right) =   (\bm{q},\bm{x}  )  \right] \\
                    =&-  \epsilon E  \left[\left\langle\bm{Q}^t  +\sum_{i=0}^{m-l}  \left(\bm{A}^{t+i}  -\bm{S}^{ t+i} +\bm{U}^{t+i}    \right), \bm{1}  )\right\rangle \mid   \left(\bm{Q}^t  ,\bm{X}^t    \right) =   (\bm{q},\bm{x}  )  \right],\\
                  \leq &- \epsilon {\norm{\bm{q}}_1}
                  +\epsilon mN^2(A_{max}+S_{max})\\
                  \leq & - \epsilon \norm{\bm{q}  }
                  +\epsilon mN^2(A_{max}+S_{max}),     
        \end{align*}
        where (a) follows from MaxWeight scheduling.
        Since $\lambda \in int(\mathcal{C})$, there exist a positive number $\epsilon $  such that $\bm{\lambda}+\epsilon  \bm{1} \in int(\mathcal{C})$. This gives (b). The last inequality comes from that for any vector $\bm{x}$, its $l_1$ norm is no less than its  $l_2$ norm.
        According to the discussion above, we have
        \begin{align*}
               E  \left[\Delta^m V  (\bm{q},\bm{x}  )\mid   \left(\bm{Q}^t  ,\bm{X}^t    \right) =   (\bm{q},\bm{x}  )  \right]& \leq \frac{K_2(m) }{2} +\left(2NA_{max}C_{max}\frac{1-\alpha_{max}^m}{1-\alpha_{max}}-m\epsilon \right)\norm{\bm{q}}\\
               & \leq \frac{K_2  (m(\epsilon )  )}{2} -\frac{m(\epsilon )\epsilon }{2}\norm{\bm{q}}
        \end{align*}
        where, 
        \begin{align*}
        m(\epsilon ) =  \min \left\{m \in \mathbb{N}_+\mid 2NA_{max}C_{max}\frac{1-\alpha_{max}^m}{1-\alpha_{max}} < \frac{m\epsilon }{2} \right\}
        \end{align*}
        and
        \begin{align*}
            K_2  (m(\epsilon )  ) &=  m(\epsilon )N^2  (A_{max}+S_{max}  )^2+2m(\epsilon )^2 N^2  (A_{max}+S_{max})\left(\epsilon +A_{max}+\lambda_{max} \right)
        \end{align*}
\end{proof}

\section{Details in Proof of Proposition \ref{prop: upper bound of q_perp}}\label{proof: prop1}
\subsection{Proof of Claim \ref{claim: 4.3.1}}\label{proof of claim 4.3.1}
\begin{proof}{}
The detailed proof of Claim \ref{claim: 4.3.1} is shown as follows.
    \begin{align*}
           \abs{ \Delta^m W_{\perp\mathcal{K}}  (\bm{q},\bm{x}  )}   &= \abs{  \left[W_{\perp\mathcal{K}}  \left(\bm{Q}^{t+m},\bm{x}^{t+m}   \right)-W_{\perp\mathcal{K}}  \left(\bm{Q}^t ,\bm{X}^t   \right)  \right]} \mathcal{I}  \left(\left(\bm{Q}^t  ,\bm{X}^t   \right ) =   (\bm{q},\bm{x}  )  \right)\\
           &\stackrel{(a)}{\leq}\norm{\bm{Q}_{\perp\mathcal{K}}^{t+m}-\bm{Q}_{\perp\mathcal{K}}^{t}}\\
           & \stackrel{(b)}{\leq} \norm{\bm{Q}^{t+m}-\bm{Q}^{t}}\\
           & = \sqrt{\sum_{ij}  \left(Q_{ij}^{t+m}-Q_{ij}^{t}  \right)^2}\\
           &  \stackrel{(c)}{\leq}\sqrt{\sum_{ij}  (m\left(A_{max}+S_{max}  \right) )^2}\\
           & = Nm\left(A_{max}+S_{max}  \right).
    \end{align*}
    where ($a$) follows from the triangle  and  ($b$) follows from the contraction property of projection.  ($c$) is valid because
        \begin{align*}
            \abs{Q_{ij}^{t+m}-Q_{ij}^{t}}=\abs{\sum_{t=1}^m A^t-S^t+U^t}\leq  m\left(A_{max}+S_{max}  \right).
        \end{align*}
\end{proof}
\subsection{Proof of Claim \ref{claim: 4.3.2}}\label{proof of claim 4.3.2}
\begin{proof}{}
We will use Lemma \ref{lemma: bound W_perp} to bound the $m$-step drift $\Delta^m W_{\perp\mathcal{K}}(\bm{q}  ) $. To this end, we will first bound the drift    $\Delta^m V  (\bm{q},\bm{x}  )$ and then bound $\Delta^m V _{\parallel\mathcal{K}}  (\bm{q},\bm{x} )$. 
Forst, consider the drift $\Delta^m V _{\parallel\mathcal{K}}  (\bm{q},\bm{x} )$:
    \begin{align}\label{eq: proof claim 4.3.2 1}
           &E  \left[\Delta^m V  (\bm{q},\bm{x}  )\mid   \left(\bm{Q}^t  ,\bm{X}^t    \right) =   (\bm{q},\bm{x}  )  \right]\nonumber\\
           = &E  \left[\norm{\bm{Q}^{t+m}}^2-\norm{\bm{Q}^t }^2\mid   \left(\bm{Q}^t  ,\bm{X}^t    \right) =   (\bm{q},\bm{x}  )  \right]\nonumber\\
           \leq &E  \left[\norm{\bm{Q}^{t+m-1}+\bm{A}^{t+m-1} -\bm{S}^{t+m-1}}^2-\norm{\bm{Q}^t }^2\mid   \left(\bm{Q}^t  ,\bm{X}^t    \right) =   (\bm{q},\bm{x}  )  \right]\nonumber\\
            = & E  \bigg[\norm{\bm{Q}^{t+m-1}}^2-\norm{\bm{Q}^t  }^2+2\left\langle \bm{Q}^{t+m-1}  , \bm{A}^{t+m-1} -\bm{S}^{t+m-1}  \right\rangle \nonumber\\
            &+ \norm{\bm{A}^{t+m-1}  -\bm{S}^{t+m-1}}^2\mid   \left(\bm{Q}^t  ,\bm{X}^t    \right) =   (\bm{q},\bm{x}  )  \bigg]\nonumber\\
           = & E  \left[\sum_{l=1}^m 2\left\langle \bm{Q}^{t+m-l}  ,\bm{A}^{t+m-1}-\bm{\lambda} \right\rangle\mid   \left(\bm{Q}^t  ,\bm{X}^t    \right) =   (\bm{q},\bm{x}  )  \right]\nonumber\\
           &+E  \left[\sum_{l=1}^m \norm{\bm{A}^{t+m-l}-\bm{S}^{t+m-l}}^2\mid   \left(\bm{Q}^t  ,\bm{X}^t    \right) =   (\bm{q},\bm{x}  )  \right]\nonumber\\
           & +\sum_{l=1}^m E  \left[2\left\langle \bm{Q} ^{t+m-l },\bm{\lambda}-\bm{S}^{t+m-l}   \right\rangle \mid   \left(\bm{Q}^t  ,\bm{X}^t    \right) =   (\bm{q},\bm{x}  )  \right],
    \end{align}
    where the last equation follows from repeating the previous step inductively. 
    By tower property, the last term in (\ref{eq: proof claim 4.3.2 1}) can be written as 
    {\small \begin{align}\label{eq: proof claim 4.3.2 2}
           &\sum_{l=1}^m E  \left[2\left\langle\bm{Q} ^{t+m-l },\bm{\lambda}-\bm{S} ^{t+m-l}  \right\rangle \mid   \left(\bm{Q}^t  ,\bm{X}^t    \right) =   (\bm{q},\bm{x}  )  \right]\nonumber\\
           & = \sum_{l=1}^m E  \bigg[E  \bigg[2\left\langle\bm{Q} ^{t+m-l },  (1-\epsilon  )\bm{v}-\bm{S} ^{t+m-l}  \right\rangle\mid \left(\bm{Q} ^{t+m-l} ,\bm{X}^{t+m-l}  \right)  \bigg]\mid   \left(\bm{Q}^t  ,\bm{X}^t    \right) =   (\bm{q},\bm{x}  )  \bigg]\nonumber\\
           & = \sum_{l=1}^m -2\epsilon E  \left[\left\langle\bm{Q} ^{t+m-l },\bm{v}\right\rangle\mid   \left(\bm{Q}^t  ,\bm{X}^t    \right) =   (\bm{q},\bm{x}  )  \right]\nonumber\\
           & +\sum_{l=1}^m E  \bigg[E  \bigg[2\left\langle\bm{Q} ^{t+m-l },\bm{v}-\bm{S} ^{t+m-l}  \right\rangle\mid \left(\bm{Q} ^{t+m-l} ,\bm{X}^{t+m-l}  \right)=  (\bm{q}',\bm{x}'  )  \bigg]\mid   \left(\bm{Q}^t  ,\bm{X}^t    \right) =   (\bm{q},\bm{x}  )  \bigg].
    \end{align}}
    Since we use MaxWeight scheduling, using  Lemma \ref{lemma: Maxweight in C}, we have
    \begin{align}\label{eq: proof claim 4.3.2 3}
          & E  \bigg[2\left\langle\bm{Q} ^{t+m-l },\bm{v}-\bm{S} ^{t+m-l}  \right\rangle\mid \left(\bm{Q} ^{t+m-l} ,\bm{X}^{t+m-l}  \right)=  (\bm{q}',\bm{x}'  )  \bigg]\nonumber\\
          & \leq  E  \bigg[2\left\langle\bm{Q} ^{t+m-l },\bm{v}-\left(\bm{v} +\frac{v_{\min}}{\norm{\bm{q}_{\perp\mathcal{K}}}}\bm{q}_{\perp\mathcal{K}} )\right)  \right\rangle\mid \left(\bm{Q} ^{t+m-l} ,\bm{X}^{t+m-l}  \right)=  (\bm{q}',\bm{x}'  )  \bigg]\nonumber\\
          & \leq -2v_{\min}\norm{\bm{Q}_{\perp\mathcal{K}}^{t+m-l}  }\nonumber\\
          & \leq -2v_{\min}\norm{\bm{Q}_{\perp\mathcal{K}}^t  }+2v_{\min}\norm{\bm{Q}_{\perp\mathcal{K}}^{t+m-l}  -\bm{Q}_{\perp\mathcal{K}}^t  }\nonumber\\
          & \leq -2v_{\min}\norm{\bm{Q}_{\perp\mathcal{K}}^t  }+2v_{\min}\norm{\bm{Q} ^{t+m-l }-\bm{Q}^t  }\nonumber\\
          & \leq -2v_{\min}\norm{\bm{Q}_{\perp\mathcal{K}}^t  }+2v_{min}N m \left(A_{max}+S_{max}  \right).
    \end{align}
    Combining  (\ref{eq: proof claim 4.3.2 1}-\ref{eq: proof claim 4.3.2 3}), we have
    \begin{align}\label{eq: V}
            & E  \left[\Delta^m V  (\bm{q},\bm{x}  )\mid   \left(\bm{Q}^t  ,\bm{X}^t    \right) =   (\bm{q},\bm{x}  )  \right]\nonumber\\
             \leq& E  \left[\sum_{l=1}^m 2\left\langle\bm{Q} ^{t+m-l },\bm{A}^{t+m-l}  -\bm{\lambda}\right\rangle\mid   \left(\bm{Q}^t  ,\bm{X}^t    \right) =   (\bm{q},\bm{x}  )  \right] \nonumber\\
             &+E  \left[\sum_{l=1}^m \norm{\bm{A}^{t+m-l}  -\bm{S}^{t+m-l}  }^2\mid   \left(\bm{Q}^t  ,\bm{X}^t    \right) =   (\bm{q},\bm{x}  )  \right]+2Nm^2v_{min}\left(A_{max}+S_{max}  \right)\nonumber\\            
            & + \sum_{l=1}^m -2\epsilon E  [\left\langle\bm{Q} ^{t+m-l },\bm{v}\right\rangle\mid   \left(\bm{Q}^t  ,\bm{X}^t    \right) =   (\bm{q},\bm{x}  )  ]-2m v_{\min}\norm{\bm{Q}_{\perp\mathcal{K}}  }\nonumber\\
             \leq& E  [\sum_{l=1}^m 2\left\langle\bm{Q} ^{t+m-l },\bm{A}^{t+m-l}  -\bm{\lambda}\right\rangle\mid   \left(\bm{Q}^t  ,\bm{X}^t    \right) =   (\bm{q},\bm{x}  )  ]\nonumber\\
             &+ \sum_{l=1}^m -2\epsilon E  [\left\langle\bm{Q} ^{t+m-l },\bm{v}\right\rangle\mid   \left(\bm{Q}^t  ,\bm{X}^t    \right) =   (\bm{q},\bm{x}  )  ]\nonumber\\
            & +2Nm^2v_{min}\left(A_{max}+S_{max}  \right)+N^2m  (A_{max}+S_{max}  )^2-2m v_{\min}\norm{\bm{Q}_{\perp\mathcal{K}}  }.
    \end{align}
    Next, we will bound the drift of $\Delta^m V_{\parallel\mathcal{K}}  (\bm{q} ,\bm{x} )$.
    \begin{align}\label{eq: proof claim 4.3.2 6}
           &E  \left[\Delta^m V _{\parallel\mathcal{K}}  (\bm{q},\bm{x} )\mid   \left(\bm{Q}^t  ,\bm{X}^t    \right) =   (\bm{q},\bm{x}  )  \right]\nonumber\\
           & = E  \left[\norm{\bm{Q}_{\parallel\mathcal{K}}^{t+m}  }^2-\norm{\bm{Q}_{\parallel\mathcal{K}}^{t}  }^2\mid   \left(\bm{Q}^t  ,\bm{X}^t    \right) =   (\bm{q},\bm{x}  )  \right]\nonumber\\
           & = E  \left[\norm{\bm{Q}_{\parallel\mathcal{K}}^{t+m}  -\bm{Q}_{\parallel\mathcal{K}}^{t+m-1}}^2\mid   \left(\bm{Q}^t  ,\bm{X}^t    \right) =   (\bm{q},\bm{x}  )  \right]\nonumber\\
           &+E  \left[2\left\langle \bm{Q}_{\parallel\mathcal{K}}^{t+m-1}, \bm{Q}_{\parallel\mathcal{K}}^{t+m}  -\bm{Q}_{\parallel\mathcal{K}}^{t+m-1}\right\rangle\mid   \left(\bm{Q}^t  ,\bm{X}^t    \right) =   (\bm{q},\bm{x}  )  \right]\nonumber\\
           & +E  \left[\norm{\bm{Q}_{\parallel\mathcal{K}}^{t+m-1}}^2-\norm{\bm{Q}_{\parallel\mathcal{K}}^{t}  }^2\mid   \left(\bm{Q}^t  ,\bm{X}^t    \right) =   (\bm{q},\bm{x}  )  \right]\nonumber\\
           &\geq E  \left[2\left\langle \bm{Q}_{\parallel\mathcal{K}}^{t+m-1}, \bm{Q}_{\parallel\mathcal{K}}^{t+m}  -\bm{Q}_{\parallel\mathcal{K}}^{t+m-1}\right\rangle\mid   \left(\bm{Q}^t  ,\bm{X}^t    \right) =   (\bm{q},\bm{x}  )  \right]\nonumber\\
           &+E  \left[\norm{\bm{Q}_{\parallel\mathcal{K}}^{t+m-1}}^2-\norm{\bm{Q}_{\parallel\mathcal{K}}^{t}  }^2\mid   \left(\bm{Q}^t  ,\bm{X}^t    \right) =   (\bm{q},\bm{x}  )  \right]
    \end{align}
    The first term can be lower bounded as follows:
    \begin{align}\label{eq: proof claim 4.3.2 5}
            &E  \left[2\left\langle \bm{Q}_{\parallel\mathcal{K}}^{t+m-1}, \bm{Q}_{\parallel\mathcal{K}}^{t+m}  -\bm{Q}_{\parallel\mathcal{K}}^{t+m-1}\right\rangle\mid   \left(\bm{Q}^t  ,\bm{X}^t    \right) =   (\bm{q},\bm{x}  )  \right]\nonumber\\ 
            =&\bigg[2\left\langle \bm{Q}_{\parallel\mathcal{K}}^{t+m-1}, \bm{Q}^{t+m}-\bm{Q}_{\perp\mathcal{K}}^{t+m}  -\left(\bm{Q}^{t+m-1}-\bm{Q}_{\parallel\mathcal{K}}^{t+m-1}\right)\right\rangle\mid   \left(\bm{Q}^t  ,\bm{X}^t    \right) =   (\bm{q},\bm{x}  )  \bigg]\nonumber\\
            \geq & E  \left[2\left\langle \bm{Q}_{\parallel\mathcal{K}}^{t+m-1}, \bm{Q}^{t+m}  -\bm{Q}^{t+m-1}\right\rangle\mid   \left(\bm{Q}^t  ,\bm{X}^t    \right) =   (\bm{q},\bm{x}  )  \right]
    \end{align}
    The last  follows the fact that since $\bm{Q}_{\parallel\mathcal{K}}^{t+m-1} \in \mathcal{K}$, $\bm{Q}_{\perp\mathcal{K}}^{t+m} \in \mathcal{K}^o$, we have 
    \begin{align*}
    \left\langle\bm{Q}_{\parallel\mathcal{K}}^{t+m-1},\bm{Q}_{\perp\mathcal{K}}^{t+m} \right\rangle \leq 0 \quad \text{and} \quad \left\langle \bm{Q}_{\parallel\mathcal{K}}^{t+m-1}, \bm{Q}_{\perp\mathcal{K}}^{t+m-1} \right\rangle = 0.
    \end{align*}

Substituting (\ref{eq: proof claim 4.3.2 5}) into (\ref{eq: proof claim 4.3.2 6}), we have
    \begin{align} \label{eq: v_para}
           &E  \left[\Delta^m V_{\parallel\mathcal{K}}  (\bm{q} ,\bm{x} )\mid   \left(\bm{Q}^t  ,\bm{X}^t    \right) =   (\bm{q},\bm{x}  )  \right]\nonumber\\
            \geq& E  \bigg[2\left\langle \bm{Q}_{\parallel\mathcal{K}}^{t+m-1},\bm{A}^{t+m-1}  -\bm{S}^{t+m-1}  +\bm{U}^{t+m-1}  \right\rangle \nonumber\\
            &+\norm{\bm{Q}_{\parallel\mathcal{K}}^{t+m-1}}^2-\norm{\bm{Q}_{\parallel\mathcal{K}}^{t}  }^2\mid   \left(\bm{Q}^t  ,\bm{X}^t    \right) =   (\bm{q},\bm{x}  )  \bigg]\nonumber\\
            \stackrel{(a)}{\geq}& \sum_{l=1}^m E  \left[2\left\langle \bm{Q}_{\parallel\mathcal{K}}^{t+m-l}  ,\bm{A}^{t+m-l}  -\bm{\lambda}\right\rangle\mid   \left(\bm{Q}^t  ,\bm{X}^t    \right) =   (\bm{q},\bm{x}  )  \right]\nonumber\\
            &+\sum_{l=1}^m E  \left[2\left\langle \bm{Q}_{\parallel\mathcal{K}}^{t+m-l}  ,\bm{\lambda}-\bm{S} ^{t+m-l}  \right\rangle\mid   \left(\bm{Q}^t  ,\bm{X}^t    \right) =   (\bm{q},\bm{x}  )  \right]\nonumber\\
            \stackrel{(b)}{=} &\sum_{l=1}^m E  \left[2\left\langle \bm{Q}_{\parallel\mathcal{K}}^{t+m-l}  ,\bm{A}^{t+m-l}  -\bm{\lambda}\right\rangle\mid   \left(\bm{Q}^t  ,\bm{X}^t    \right) =   (\bm{q},\bm{x}  )  \right]\nonumber\\
            &+\sum_{l=1}^m E  \left[-2\epsilon\left\langle \bm{Q}_{\parallel\mathcal{K}}^{t+m-l}  ,\bm{v}\right\rangle\mid   \left(\bm{Q}^t  ,\bm{X}^t    \right) =   (\bm{q},\bm{x}  )  \right]\nonumber\\
           & +\sum_{l=1}^m E  \left[-2\left\langle\bm{Q}_{\parallel\mathcal{K}}^{t+m-l}  , \bm{S} ^{t+m-l}  -\bm{v}\right\rangle\mid   \left(\bm{Q}^t  ,\bm{X}^t    \right) =   (\bm{q},\bm{x}  )  \right]\nonumber\\
            = &\sum_{l=1}^m E  \left[2\left\langle \bm{Q}_{\parallel\mathcal{K}}^{t+m-l}  ,\bm{A}^{t+m-l}  -\bm{\lambda}\right\rangle\mid   \left(\bm{Q}^t  ,\bm{X}^t    \right) =   (\bm{q},\bm{x}  )  \right]\nonumber \\
            &+ \sum_{l=1}^m E  \left[-2\epsilon\left\langle \bm{Q}_{\parallel\mathcal{K}}^{t+m-l}  ,\bm{v}\right\rangle\mid   \left(\bm{Q}^t  ,\bm{X}^t    \right) =   (\bm{q},\bm{x}  )  \right].
    \end{align}
    where  ($a$) follows by recursively opening up the previous  and by noting that 
    \begin{align*}
        \left\langle \bm{Q}^{t+m-1}, \bm{U}^{t+m-1} \right\rangle \geq 0,
    \end{align*}
    since each component of $\bm{Q}^{t+m-1}$ and $\bm{U}^{t+m-1}$ are non-negative. Equality ($b$) follows from $\bm{\lambda} = (1-\epsilon)\bm{v}$. Since $\bm{q_{\parallel\mathcal{K}}} \in \mathcal{K}$ and $\bm{s}, \bm{v} \in \mathcal{F}$, according to (\ref{eq: orthognal}), we have $\left\langle\bm{Q}_{\parallel\mathcal{K}}^{t+m-l}  , \bm{S} ^{t+m-l}  -\bm{v}\right\rangle=0$, which gives us the last equation. Combining  (\ref{eq: V}) and (\ref{eq: v_para}), we have
    \begin{align}\label{eq: proof claim 4.3.2.7}
        & E \left[\Delta^m V  (\bm{q},\bm{x}  )-\Delta^m V  (\bm{q}_{\perp\mathcal{K}}  )\mid   \left(\bm{Q}^t  ,\bm{X}^t    \right) =   (\bm{q},\bm{x}  )  \right]  \nonumber\\
        \leq & E  \left[\sum_{l=1}^m 2\left\langle\bm{Q}_{\perp\mathcal{K}}^{t+m-l}  ,\bm{A}^{t+m-l}  -\bm{\lambda}\right\rangle\mid   \left(\bm{Q}^t  ,\bm{X}^t    \right) =   (\bm{q},\bm{x}  )  \right]\nonumber\\
        &+2Nm^2 A_{max}+N^2m  (A_{max}+S_{max}  )^2\nonumber\\
        & + \sum_{l=1}^m -2\epsilon E  \left[\left\langle\bm{Q}_{\perp\mathcal{K}}^{t+m-l}  ,\bm{v}\right\rangle\mid   \left(\bm{Q}^t  ,\bm{X}^t    \right) =   (\bm{q},\bm{x}  )  \right] -2m v_{\min}\norm{\bm{Q}_{\perp\mathcal{K}}^t  }\nonumber\\
         \stackrel{(a)}\leq & 2\norm{\bm{Q}_{\perp\mathcal{K}}^t  }E  \left[\sum_{l=1}^m\norm{\bm{A}^{t+m-l}  -\bm{\lambda}}\mid   \left(\bm{Q}^t  ,\bm{X}^t    \right) =   (\bm{q},\bm{x}  )  \right]\nonumber\\
         &+ 2Nm^2 A_{max}  (N(A_{max}+\lambda)+\epsilon\norm{\bm{v}}  )-2m v_{\min}\norm{\bm{Q}_{\perp\mathcal{K}}^t  }\nonumber\\
        & +2m\epsilon\norm{\bm{Q}_{\perp\mathcal{K}}^t  }\norm{\bm{v}}+2Nm^2v_{\min}\left(A_{max}+S_{max}  \right)+N^2m  (A_{max}+S_{max}  )^2\nonumber\\
        =& 2\norm{\bm{Q}_{\perp\mathcal{K}}^t  }  E  \left[\sum_{l=1}^m\norm{\bm{A}^{t+m-l}  -\bm{\lambda}}\mid   \left(\bm{Q}^t  ,\bm{X}^t    \right) =   (\bm{q},\bm{x}  )  \right]\nonumber\\
        &+ 2\norm{\bm{Q}_{\perp\mathcal{K}}^t  }  \left(m\epsilon\norm{\bm{v}}-mv_{\min}  \right)+K_3(m),   
        \end{align}
    
    where ($a$) follows by using Cauchy–Schwarz Inequality and,
    \begin{align*}
              K_3(m) &= 2Nm^2\left(A_{max}+S_{max}  \right)  (N(A_{max}+\lambda)\\
              &+\epsilon\norm{\bm{v}} +v_{\min} )+N^2m  (A_{max}+S_{max}  )^2 ,   
    \end{align*}
    According to Lemma \ref{lemma: bound W_perp}, we conclude 
    \begin{align*}
            &E  \left[\Delta^m W_{\perp\mathcal{K}}  (\bm{q},\bm{x}  ) \mid   \left(\bm{Q}^t  ,\bm{X}^t    \right) =   (\bm{q},\bm{x}  )\right]\\
            &\leq E\left[\frac{1}{2\norm{\bm{q}_{\perp\mathcal{K}}}}  \left(\Delta^m V  (\bm{q},\bm{x}  )-\Delta^m V  (\bm{q}_{\perp\mathcal{K}}  )  \right)\mid   \left(\bm{Q}^t  ,\bm{X}^t    \right) =   (\bm{q},\bm{x}  )  \right]\\
            & \leq E  \left[\sum_{l=1}^m\norm{\bm{A}^{t+m-l}  -\bm{\lambda}}\mid   \left(\bm{Q}^t  ,\bm{X}^t    \right) =   (\bm{q},\bm{x}  )  \right] + \frac{K_3(m)}{2\norm{\bm{Q}_{\perp\mathcal{K}}^t  }}+m  (\epsilon\norm{\bm{v}}-v_{\min}  )
    \end{align*}
\end{proof}

\section{Proof of Lemma \ref{Lemma: finite V_1, V_2, V_3}} \label{Proof of Lemma finite V_1, V_2, V_3}
    \begin{proof}{of Lemma \ref{Lemma: finite V_1, V_2, V_3}.}
    Note that  $V'(\bm{q},\bm{x}) = \norm{\bm{q}_{{\parallel}\mathcal{L}}}^2\leq \norm{\bm{q}}^2=V(\bm{q},\bm{x})$. Therefore, in order to prove the lemma, we will show that $E  [V  (\bm{Q},\bm{X}  )  ]$ is finite in steady-state. We do this using Lemma \ref{Lemma: $m$-step bound for q_perp_norm_2}
on the Lyapunov function $W  (\bm{q},\bm{x}  ) = \norm{\bm{q}} = \sqrt{V  (\bm{q},\bm{x}  )}$. Note that its $m$-step drift is
        \begin{equation*}
            \Delta^m W  (\bm{q},\bm{x}  ) \triangleq   \left[W  \left(\bm{Q}^{t+m},\bm{X}^{t+m}    \right)-W  \left(\bm{Q}^t,\bm{X}^t  \right)  \right]\mathcal{I}  \left(  \left(\bm{Q}^t  ,\bm{X}^t    \right)=  (\bm{q},\bm{x}  )  \right).
        \end{equation*}
        
To use Lemma  \ref{Lemma: $m$-step bound for q_perp_norm_2}, we first check that the Conditions 2 and 3 of Lemma \ref{Lemma: $m$-step bound for q_perp_norm_2}. 
        \begin{align*}
               \abs{ \Delta W  (\bm{q},\bm{x}  )}  & = \abs{ \norm{\bm{Q}^{t+1}  }-\norm{\bm{Q}^t  }}\mid \mathcal{I}  (  \left(\bm{Q}^t  ,\bm{X}^t    \right)=  (\bm{q},\bm{x}  )  )\\
               & \leq \norm{\bm{Q}^{t}  -\bm{Q}^t}\\
               & = \sqrt{\sum_{ij}  \left(Q_{ij}^{t}  -Q_{ij}^t    \right)^2}\\
               & \leq \sqrt{\sum_{ij}  \left(m  (A_{max}+S_{max}  )  \right)^2}\\
               & = N  (A_{max}+S_{max}  ).       
        \end{align*}
        {Therefore, Conditions 2 and 3 are satisfied with $D(m) = Nm(A_{max}+S_{max})$ and $\hat{D}(t_0) =N  (A_{max}+S_{max}  ) $.}
        
        We will now verify Condition 1.
        \begin{align*}
               &E  \left[\Delta^m W  (\bm{q},\bm{x}  )\mid   \left(\bm{Q}^t  ,\bm{X}^t    \right)=  (\bm{q},\bm{x}  )  \right]\\
               =& E  \left[\norm{\bm{Q}^{t+m}  }-\norm{\bm{Q}^t  }\mid   \left(\bm{Q}^t  ,\bm{X}^t    \right)=  (\bm{q},\bm{x}  )  \right]\\
                \leq & E   \left[\frac{1}{2\norm{\bm{Q}^t  }}  \left(\norm{\bm{Q}^{t+m}  }^2-\norm{\bm{Q}^t  }^2  \right)\mid   \left(\bm{Q}^t  ,\bm{X}^t    \right)=  (\bm{q},\bm{x}  )   \right]\\
               \leq &  \frac{1}{2\norm{\bm{q} }}E  \left[\Delta^m V  (\bm{q},\bm{x}  )\mid   \left(\bm{Q}^t  ,\bm{X}^t    \right) =   (\bm{q},\bm{x}  )  \right]\\
               \stackrel{(a)}\leq & -\frac{m(\epsilon )}{4}\ \text{ for all $\bm{q}$ such that $W (\bm{q}) \geq \frac{2K_2  (m(\epsilon ) )}{\epsilon }$}\
        \end{align*}
        where, (a) follows from Claim \ref{claim: 4.2.1},  
        \begin{align*}
        m(\epsilon ) =  \min \left\{m \in \mathbb{N}_+\mid 2NA_{max}C_{max}\frac{1-\alpha_{max}^m}{1-\alpha_{max}} < \frac{m\epsilon }{2} \right\}
        \end{align*}
        and
        \begin{align*}
            K_2  (m(\epsilon )  ) &=  m(\epsilon )N^2  (A_{max}+S_{max}  )^2+2m(\epsilon )^2 N^2  (A_{max}+S_{max})\left(\epsilon +A_{max}+\lambda_{max} \right).
        \end{align*}
        {Therefore, Condition 1 is also verified with $\kappa(m(\epsilon)) = \frac{2K_2  (m(\epsilon ) )}{\epsilon }$ and $\eta(m(\epsilon)) = \frac{m(\epsilon )}{4}$}. According to Lemma \ref{Lemma: $m$-step bound for q_perp_norm_2}, we conclude that all moments of $W  (\overline{\bm{Q}},\overline{\bm{X}}  )$ are finite in steady state. The lemma follows by noting that $V'  (\bm{q} ,\bm{x} ) = \norm{\mathbb{q}_{{\parallel}\mathcal{L}}}^2$ is the norm of projection of $\mathbb{q}$ onto $\mathcal{S}$.
    \end{proof}

\section{Details in Proof of Theorem \ref{thm: exp}}
\subsection{Proof of Lemma \ref{lem: comparing}} \label{Proof of Lemma: comparing}
   \begin{proof}  {of Lemma \ref{lem: comparing}.  }
   Since $\epsilon$ is fixed, we drop $(\cdot)^{(\epsilon)}$ for simplicity. First, let $X^t =\bm{x}_0$, let $\alpha^{t+l}$ denote the conditional distribution of $X^{t+l}$ and let $P$ denote the transition matrix of $\{X^t\}_{t\geq0}$. Then, for any $l \in \{0,1,...,m-1\}$,
    \begin{align*}
    &E  \left[  \left(  \left(A^{t+l}-\lambda  \right)  \left(A^{t+m}-\lambda  \right)\mid X^t =\bm{x}_0 \right)  \right]\nonumber\\
    &=\sum_{x \in \Omega}\alpha^{t+l}  \left(x  \right)  \left[f  (x  )-\lambda  \right]  \left[\sum_{y \in \Omega}P_{xy}^{m-l}f  (y  )-\lambda  \right]\\
    &=\sum_{x \in \Omega}\alpha^{t+l}  (x  )  \left[f  (x  )-\lambda  \right]  \left[\sum_{y \in \Omega}  \left(P_{xy}^{m-l}-\pi  (y  )  \right)f  (y  )  \right].\nonumber
    \end{align*}
    Similarly, we have
    \begin{align*}
    &E  \left[  \left(b^0-\lambda  \right)  \left(b^{m-l}-\lambda  \right)  \right]\\
    &=\sum_{x \in \Omega}\pi  (x  )  \left[f  (x  )-\lambda  \right]  \left[\sum_{y \in \Omega}P_{xy}^{m-l}f  (y  )-\lambda  \right]\\
    &=\sum_{x \in \Omega}\pi  (x  )  [f  (x  )-\lambda  ]  \left[\sum_{y \in \Omega}  \left(P_{xy}^{m-l}-\pi  (y  )  \right)f  (y  )  \right].
    \end{align*}
    Thus,
    \begin{align*}
        &\abs{\sum_{l=0}^{m-1}  \left(E  \left[  \left(A^{t+l}-\lambda  \right)  \left(A^{t+m}-\lambda  \right)\mid X^t=\bm{x}_0  \right]-E  \left[  \left(b^{m-l}-\lambda  \right)  \left(b^0-\lambda  \right)  \right] \right)}\\
        =&\abs{\sum_{l=0}^{m-1}\sum_{x \in \Omega}  \left(\alpha^{t+l}  (x  )-\pi  (x  )  \right)  \left(f  (x  )-\lambda  \right)\sum_{y \in \Omega}  \left[P_{xy}^{m-l}-\pi  (y  )  \right]f  (y  )}\\
        =&\abs{\sum_{l=0}^{m-1}\sum_{x \in \Omega}\sum_{z \in \Omega}\alpha^t  (z  )  \left[P_{zx}^l  (x  )-\pi  (x  )  \right]  \left[f  (x  )-\lambda  \right]\sum_{y \in \Omega}  \left[P_{xy}^{m-l}-\pi  (y  )  \right]f  (y  )}\\
        \leq&   (A_{max}+\lambda  )A_{max}\sum_{l=0}^{m-1}\sum_{x \in \Omega}\sum_{z \in \Omega}\alpha^t  (z  )\abs{P_{zx}^l  (x  )-\pi  (x  )}\sum_{y \in \Omega}\abs{P_{xy}^{m-l}-\pi  (y  )}
    \end{align*}
    and
    \begin{align*}
        &\abs{E  \left[  \left(A^{t+m}-\lambda \right )^2-  (b^0-\lambda  )^2\mid X^t =\bm{x}_0   \right]}\\
        =&\abs{\sum_{x \in \Omega}  \left(\alpha^{t+m}  (x  )-\pi  (x  )  \right)  \left(f  (x  )-\lambda  \right)^2}\\
        =&\abs{\sum_{x \in \Omega}\sum_{y \in \Omega}\pi  (y  )  \left(P_{yx}^m-\pi  (x  )  \right)  \left(f  (x  )-\lambda  \right)^2}\\
        \leq & 2  (A_{max}+\lambda  )^2C\alpha^m.
    \end{align*}
    According to (\ref{eq: geo1}) and (\ref{eq: geo2}), we have  $\forall \bm{x}_0 \in \Omega$,
    \begin{align*}
           &\abs{\sum_{l=0}^{m-1}  \left(E  \left[  \left(A^{t+l}-\lambda  \right)  \left(A^{t+m}-\lambda  \right)\mid X^t=\bm{x}_0  \right]-E  \left[  \left(b^{m-l}-\lambda  \right)  \left(b^0-\lambda  \right)  \right] \right)}\\
           &\leq 4  \left(A_{max}+\lambda  \right)A_{max}C^2m\alpha^m,
    \end{align*}
    and
    \begin{align*}
             \abs{E  \left[  \left(A^{t+m}-\lambda \right )^2-  (b^0-\lambda  )^2 \mid  X^t=\bm{x}_0\right]} \leq  2  (A_{max}+\lambda  )^2C\alpha^m.
    \end{align*}
Since these bounds are true for all $ \bm{x}_0 \in \Omega$, and since the bounds do not depend on $\bm{x}_0$, we have the lemma. 
\end{proof}

\subsection{Proof of Theorem \ref{thm: exp} }\label{Appd: Heavy-traffic Limit Distribution for Queue Length}
\begin{proof}  {of Theorem \ref{thm: exp}.  }
In the rest of the proof, we consider the system in its steady-state, and so for every time $t$, 
 $\left(Q^t\right)^{(\epsilon)} \stackrel{d}= \overline{Q}^{(\epsilon)}$. For ease of exposition, we again drop the superscript $(\cdot)^{(\epsilon)}$
 and just use $Q^t$. 
 Then, $A^t$  denotes the arrival in steady state, and the queue length at time $t+m$, is  $Q^{t+m}=Q^t+\sum_{l=0}^{m-1} A^{t+l}-\sum_{l=0}^{m-1} S^{t+l}+\sum_{l=0}^{m-1} U^{t+l}$, 
 which has the same distribution as $Q^t$ 
 for all $m \in \mathbb{N}_+$.  Note that, from the definition of the unused service and (\ref{eq: 2}), it can be easily shown (by considering the cases of $U=0$ and $U\neq 0$) that
\begin{equation*}
\begin{aligned}
      \left(e^{\epsilon\theta Q^{t+1}}-1  \right)  \left(e^{-\epsilon\theta U^{t}}-1  \right)=0\\
e^{\epsilon\theta Q^{t+1}} = 1- e^{-\epsilon\theta U^{t}}+e^{\epsilon\theta \left(Q^t+A^t-S^t\right)}.
\end{aligned}
\end{equation*}
Taking expectation with respect to the stationary distribution on both sides, we have
\begin{align}\label{ref: exp 1}
E  \left[e^{\epsilon\theta Q^{t+1}}  \right]= 1- E  \left[e^{-\epsilon\theta U^{t}}  \right]+E  \left[ e^{\epsilon \theta  (Q^t+A^t-S^t  )}  \right].
\end{align}
Since $\theta \leq 0$, we have that $e^{\epsilon\theta Q^t}\leq 1$. Therefore, in steady-state $E  \left[e^{\epsilon\theta Q^{t}}\right]$ is bounded, and so we can set its drift to zero:
\begin{align}\label{ref: exp 2}
    E  \left[e^{\epsilon\theta Q^{t+1}}  \right] -E  \left[e^{\epsilon\theta Q^{t}}  \right]=0.
\end{align}
Combining (\ref{ref: exp 1}) and (\ref{ref: exp 2}), we have
\begin{align*}
    E  \left[e^{\epsilon\theta Q^{t}}  \right]= 1- E  \left[e^{-\epsilon\theta U^{t}}  \right]+E  \left[e^{\epsilon \theta  (Q^t+A^t-S^t  )}  \right]\\
        E  \left[e^{\epsilon\theta Q^{t}}  \left(e^{\epsilon \theta  (A^t-S^t  )}-1  \right)  \right]=  E  \left[e^{-\epsilon\theta U^{t}}  \right]-1.
\end{align*}
Therefore, since we are in steady-state and so we can replace index $t$ by $t+m$ and get,
\begin{align*}
        E  \left[e^{\epsilon\theta Q^{t+m}}  \left(e^{\epsilon \theta  (A^{t+m}-S^{t+m}  )}-1  \right)  \right]=  E  \left[e^{-\epsilon\theta U^{t+m}}  \right]-1 
   =  E  \left[e^{-\epsilon\theta U^{t}}  \right]-1,
\end{align*}
where the last equality is because of steady-state, and the fact that $E  \left[e^{-\epsilon\theta U^{t}}\right]$ is bounded in steady-state. Boundedness follows from the observation that $U^t$ is bounded by $S_{max}$. Expanding $Q^{t+m}$ using \eqref{eq:1}, we get
\begin{align}\label{eq: main}
  & E  \left[e^{\epsilon\theta Q^{t}}e^{\epsilon\theta  \left(\sum_{l=0}^{m-1}A^{t+l}-\sum_{l=0}^{m-1}S^{t+l}+\sum_{l=0}^{m-1}U^{t+l}  \right)}  \left(e^{\epsilon \theta  \left(A^{t+m}-S^{t+m} \right )}-1  \right) \right ]\nonumber\\
  &= E  \left[e^{-\epsilon\theta U^{t}}  \right]-1.
\end{align}
Taking Taylor expansion of RHS of (\ref{eq: main}) with respect to $\theta$, we have
\begin{align}\label{eq: RHS}
        E  \left[e^{-\epsilon\theta U^{t}}  \right]-1
        =& -E  \left[-\epsilon\theta U^{t}  \right]+E  \left[\frac{\epsilon^2\theta^2 \left(U^{t}\right)^2}{2}  \right]+E\left[\frac{-\left(\epsilon U^t\right)^3\hat{\theta}^3}{6}\right]\nonumber\\
        \stackrel{(a)}{=}& -\epsilon^2\theta+\frac{\epsilon^2\theta^2 E  \left[\left(U^{t}\right)^2  \right]}{2}+ E\left[\frac{-\left(\epsilon U^t\right)^3\hat{\theta}^3}{6}\right]\nonumber\\
        = &\epsilon^2\theta  \left(-1+\frac{\theta E  \left[\left(U^{t}\right)^2  \right]}{2}  \right)+E\left[\frac{-\left(\epsilon U^t\right)^3\hat{\theta}^3}{6}\right],
\end{align}
where $(a)$ is due to (\ref{ E(u)=epsilon}).

Since $\hat{\theta} \in (\theta,0]$, and $U^t \leq S_{max}$, the absolute value of last term in (\ref{eq: RHS}) can be upper bounded by $K \epsilon^3$ where $K$ is a constant. Therefore, for ease of exposition, we can write (\ref{eq: RHS}) as:
\begin{equation*}
    \begin{aligned}
        &E  \left[e^{-\epsilon\theta U^{t}}  \right]-1 = \epsilon^2\theta  \left(-1+\frac{\theta E  \left[\left(U^{t}\right)^2  \right]}{2}  \right)+O\left(\epsilon^3\right) = \epsilon^2\theta  \left(-1+\frac{\theta O(\epsilon)}{2}  \right)+O\left(\epsilon^3\right).
\end{aligned}    
\end{equation*}
where the last equation comes from:
$E  \left[\left(U^{t}\right)^2  \right] \leq E\left[\left(U^{t}S^{t}\right)  \right] \leq S_{max}E\left[ U^{t}\right]\leq S_{max}\epsilon$ using \eqref{ E(u)=epsilon}.
Taking Taylor expansion of LHS of (\ref{eq: main}) with respect to $\theta$, for the same reason, we have
\begin{equation*}
    \begin{aligned}
        &  E  \left[e^{\epsilon\theta   Q^{t}}e^{\epsilon\theta  \left(\sum_{l=0}^{m-1}A^{t+l}-\sum_{l=0}^{m-1}S^{t+l}+\sum_{l=0}^{m-1}U^{t+l}  \right)}  \left(e^{\epsilon \theta  \left(A^{t+m}-S^{t+m} \right )}-1  \right)  \right]\\
        = &\epsilon^2\theta E  \Bigg[e^{\epsilon\theta Q^{t}}  \left(1+\theta\epsilon  \left(\sum_{l=0}^{m-1}A^{t+l}-\sum_{l=0}^{m-1}S^{t+l}+\sum_{l=0}^{m-1}U^{t+l}  \right)+O  \left(\epsilon^2  \right)  \right)\\
        &\left(\epsilon^{-1}  \left(A^{t+m}-S^{t+m}  \right)+\frac{\theta}{2}  \left(A^{t+m}-S^{t+m}  \right)^2+O  \left(\epsilon  \right)  \right)  \Bigg]\\
        \end{aligned}
\end{equation*}
It can be divided into seven terms as follows:         
\begin{align*}
        &  \epsilon^2\theta \Bigg \{\underbrace{ E  \left[e^{\theta \epsilon Q^t}\epsilon^{-1}  \left(A^{t+m}-S^{t+m}  \right)  \right]}_{\mathcal{T}_{5}}\\
        +&  \underbrace{E  \left[e^{\theta \epsilon Q^t}\theta  \left(\sum_{l=0}^{m-1}A^{t+l}-\sum_{l=0}^{m-1}S^{t+l}+\sum_{l=0}^{m-1}U^{t+l}  \right)  \left(A^{t+m}-S^{t+m} \right )  \right]}_{\mathcal{T}_{6}}\\
        +& \underbrace{E  \left[O  (\epsilon  )e^{\theta\epsilon Q^t}  \left(A^{t+m}-S^{t+m} \right )  \right]}_{\mathcal{T}_{7}}+ \underbrace{E  \left[e^{\theta \epsilon Q^t}\frac{\theta}{2}  \left(A^{t+m}-S^{t+m} \right )^2  \right]}_{\mathcal{T}_{8}}\\
        +& \underbrace{E  \left[e^{\theta \epsilon Q^t}\frac{ \epsilon\theta^2}{2}  \left(\sum_{l=0}^{m-1}A^{t+l}-\sum_{l=0}^{m-1}S^{t+l}+\sum_{l=0}^{m-1}U^{t+l}  \right)  \left(A^{t+m}-S^{t+m} \right )^2  \right]}_{\mathcal{T}_{9}}\\
        +&\underbrace{E  \left[O  \left(\epsilon^2  \right)\frac{\theta}{2}  \left(A^{t+m}-S^{t+m}  \right)^2  \right]}_{\mathcal{T}_{10}}\\
        +& \underbrace{E  \left[e^{\theta\epsilon Q^t}  \left(1+\theta\epsilon  \left(\sum_{l=0}^{m-1}A^{t+l}-\sum_{l=0}^{m-1}S^{t+l}+\sum_{l=0}^{m-1}U^{t+l}  \right)+O  (\epsilon^2  )  \right)O  (\epsilon  )  \right]}_{\mathcal{T}_{11}}\Bigg \}.
\end{align*}
According to Lemma \ref{lem: geo}, we have
\begin{align*}
    \mathcal{T}_{5}=&E  \left[e^{\theta \epsilon Q^t}\epsilon^{-1}  \left(A^{t+m}-\lambda \right )  \right] +E  \left[e^{\theta \epsilon Q^t}\epsilon^{-1}  \left(\lambda-S^{t+m}  \right)  \right] \\
    =&  E  \left[e^{\theta \epsilon Q^t}\epsilon^{-1}  \left(A^{t+m}-\lambda \right )  \right] - E  \left[e^{\theta\epsilon Q^t}  \right]. 
\end{align*}
Since 
\begin{align*}
  \abs{E  \left[e^{\theta \epsilon Q^t}\epsilon^{-1}  \left(A^{t+m}-\lambda \right )  \right]} \leq \frac{\alpha^m}{\epsilon},
\end{align*}
we have 
\begin{align*}
    - E  \left[e^{\theta\epsilon Q^t}  \right]-\frac{\alpha^m}{\epsilon} \leq \mathcal{T}_{5}\leq& - E  \left[e^{\theta\epsilon Q^t}  \right]+\frac{\alpha^m}{\epsilon}. 
\end{align*}

\begin{align*}
       \mathcal{T}_{6}= &E  \left[e^{\theta \epsilon Q^t}\theta  \left(\sum_{l=0}^{m-1}A^{t+l}-\sum_{l=0}^{m-1}S^{t+l}+\sum_{l=0}^{m-1}U^{t+l}  \right)  \left(A^{t+m}-S^{t+m} \right )  \right]\\
        =&\theta E  \left[e^{\theta \epsilon Q^t} \left (\sum_{l=0}^{m-1}A^{t+l}-\sum_{l=0}^{m-1}S^{t+l}  \right)  \left(A^{t+m}-S^{t+m} \right )  \right]\\
        &+ \theta E  \left[e^{\theta \epsilon Q^t} \sum_{l=0}^{m-1}U^{t+l}  \left(A^{t+m}-S^{t+m} \right )  \right].
\end{align*}
Since
\begin{align}
   & \abs{\theta E  \left[e^{\theta \epsilon Q^t} \sum_{l=0}^{m-1}U^{t+l}  \left(A^{t+m}-S^{t+m} \right )  \right]}\nonumber\\
   \leq &\abs{ \theta} \sqrt{E  \left[e^{2\theta \epsilon Q^t}  \left(A^{t+m}-S^{t+m} \right )^2  \right]}\sqrt{  \left(\sum_{l=0}^{m-1}U^{t+l}  \right)^2}\nonumber\\
   \leq &\abs{ \theta}  m\sqrt{\epsilon}  (A_{max}+S_{max}  )\sqrt{S_{max}},
\end{align}
We have 
\begin{align*}
       \mathcal{T}_{6} &=\theta E  \left[e^{\theta \epsilon Q^t} \left (\sum_{l=0}^{m-1}A^{t+l}-\sum_{l=0}^{m-1}S^{t+l}  \right)  \left(A^{t+m}-S^{t+m} \right )  \right] +mO(\sqrt{\epsilon})
\end{align*}
Also, $\abs{ \mathcal{T}_{7}}, \abs{ \mathcal{T}_{9}}, \abs{ \mathcal{T}_{10}}, \abs{ \mathcal{T}_{11}}$ can be bounded as follows:
\begin{align*}
   &\abs{ \mathcal{T}_{7}} =O  (\epsilon  ),\\
       &\abs{ \mathcal{T}_{9}} \leq \frac{\theta^2\epsilon}{2}m  (A_{max}+S_{max}  )  (A_{max}+S_{max}  )^2=mO(\epsilon),\\ 
       &\abs{ \mathcal{T}_{10}} \leq \frac{\mid \theta\mid }{2}  (A_{max}+S_{max}  )^2O  \left(\epsilon^2  \right)=O(\epsilon^2),\\
       &\abs {\mathcal{T}_{11}} \leq O  (\epsilon  ).
\end{align*}
Now consider:
\begin{align*}
       &\mathcal{T}_{6}+\mathcal{T}_{8}\\
       =& \theta E  \left[e^{\theta\epsilon Q^t}\sum_{l=0}^{m-1}  (A^{t+l}-\lambda  )  \left(A^{t+m}-\lambda \right )  \right] -\theta E  \left[e^{\theta\epsilon Q^t}\sum_{l=0}^{m-1}  (A^{t+l}-\lambda  )  (S^{t+m}-\lambda  )  \right]\\
       &-\theta E  \left[e^{\theta\epsilon Q^t}\sum_{l=0}^{m-1}  (S^{t+l}-\lambda  )  \left(A^{t+m}-\lambda \right )  \right] +\theta E  \left[e^{\theta\epsilon Q^t}\sum_{l=0}^{m-1}  (S^{t+l}-\lambda  )  (S^{t+m}-\lambda  )  \right]\\
       &+mO(\sqrt{\epsilon})+ \frac{\theta}{2}E  \left[e^{\theta\epsilon Q^t}  \left(  \left(A^{t+m}-\lambda \right )^2+\sigma_s^2  \right)  \right] \\
       &+\frac{\theta}{2}E  \left[e^{\theta\epsilon Q^t}  \left(2  \left(A^{t+m}-\lambda \right )  \left(\lambda-S^{t+m}  \right)+\epsilon^2  \right)  \right]\\
       = &\theta E  \left[e^{\theta\epsilon Q^t}\sum_{l=0}^{m-1}  (A^{t+l}-\lambda  )  \left(A^{t+m}-\lambda \right )  \right] +\frac{\theta}{2}E  \left[e^{\theta\epsilon Q^t}  \left(  \left(A^{t+m}-\lambda \right )^2+\sigma_s^2  \right)  \right]\\
       &+mO(\epsilon   )+mO(\sqrt{\epsilon}),
\end{align*}
where the last equation comes from the independence between the current service process and the previous arrival, queue length and service process.
According to Lemma \ref{lem: comparing},  we have
\begin{align}\label{eq: comp-exp}
        \mathcal{T}_{6}+\mathcal{T}_{8}\leq &\frac{\theta}{2} E  \left[e^{\theta\epsilon Q^t}  \right]  \bigg(2\sum_{l=0}^{m-1}E  \left[  \left(b^{m-l}-\lambda \right )  \left(b^0-\lambda  \right)  \right]+E  \left[  (b^0-\lambda  )^2  \right]+\sigma_s^2  \bigg)\nonumber\\
        &+mO(\sqrt{\epsilon})+mO\left(\epsilon \right)+L_1m\alpha^m
     +  L_2\alpha^m,\nonumber\\
             \mathcal{T}_{6}+\mathcal{T}_{8}\geq &\frac{\theta}{2} E  \left[e^{\theta\epsilon Q^t}  \right]  \bigg(2\sum_{l=0}^{m-1}E  \left[  \left(b^{m-l}-\lambda \right )  \left(b^0-\lambda  \right)  \right]+E  \left[  (b^0-\lambda  )^2  \right]+\sigma_s^2  \bigg)\nonumber\\
        &+mO(\sqrt{\epsilon})+mO\left(\epsilon \right)-L_1m\alpha^m
     -  L_2\alpha^m.
\end{align}
where $L_1 = 4  \left(A_{max}+\lambda  \right)A_{max}C^2$ and $L_2 = 2  (A_{max}+\lambda  )^2C$.
Combine (\ref{eq: main}) to (\ref{eq: comp-exp}), we have
\begin{align*}
   &E  \left[e^{\theta\epsilon Q^t}  \right]  \left(-1+\frac{\theta}{2}     \bigg(2\sum_{l=0}^{m-1}E  \left[  \left(b^{m-l}-\lambda \right )  \left(b^0-\lambda  \right)  \right] +E  \left[  (b^0-\lambda  )^2  \right]+\sigma_s^2  \bigg)\right)\\
   &+mO(\sqrt{\epsilon})+mO\left(\epsilon \right)+\frac{\alpha^m}{\epsilon}-L_1m\alpha^m
     -  L_2\alpha^m \\
   \leq& O  (\epsilon  )-1.
\end{align*}
\begin{align*}
      &E  \left[e^{\theta\epsilon Q^t}  \right]  \left(-1+\frac{\theta}{2}     \bigg(2\sum_{l=0}^{m-1}E  \left[  \left(b^{m-l}-\lambda \right )  \left(b^0-\lambda  \right)  \right] +E  \left[  (b^0-\lambda  )^2  \right]+\sigma_s^2  \bigg)\right)\\
      & +mO(\sqrt{\epsilon})+mO\left(\epsilon \right)+\frac{\alpha^m}{\epsilon}+L_1m\alpha^m
     +  L_2\alpha^m \\
   \geq& O  (\epsilon  )-1.
\end{align*}
Let $m=\epsilon^{-\frac{1}{4}}$ and let $\epsilon \to 0$, we have
\begin{align*}
   \lim_{\epsilon \to 0}E  [e^{\theta\epsilon Q^t}  ] =& \frac{1}{1-\frac{2\lim_{m \to \infty}\sum_{l=0}^{m-1} E  [ \left(b^{m-l}-\lambda \right )  \left(b^0-\lambda  \right)  ]+E  [  (b^0-\lambda  )^2  ]+\sigma_s^2}{2}}.
\end{align*}
Put the ${(\cdot )}^{(\epsilon)}$ notation back. We have
\begin{align*}
        &\lim_{\epsilon \to 0}E  \left[e^{\theta\epsilon \overline{Q}^{(\epsilon)}}  \right] \\
        =& \frac{1}{1-\frac{2\lim_{m(\epsilon) \to \infty}\sum_{l=1}^{m(\epsilon)}E  \left[  \left (b^{(\epsilon)}_{l}-\lambda^{(\epsilon)}  \right)  \left (\left(b^{0}\right)^{(\epsilon)}-\lambda^{(\epsilon)}  \right)  \right]+ \lim_{\epsilon \to 0}E\left[ \left(b^{(\epsilon)}_{l}-\lambda^{(\epsilon)}  \right)^2 \right]+\sigma_s^2}{2}}\\
        =& \frac{1}{1-\frac{\lim_{\epsilon \to 0} \gamma^{(\epsilon)}(0) + \lim_{m(\epsilon) \to \infty} 2\sum_{i=1}^{m(\epsilon)}\gamma^{(\epsilon)}(t)+\sigma_s^2}{2}}.\\
\end{align*}
Similarly as the proof in part 2 of Theorem \ref{thm: pos rec single server queue}, using Claim \ref{claim: 3.2.3}, we have
\begin{align*}
    \lim_{\epsilon \to 0}E  \left[e^{\epsilon\theta \overline{Q}^{(\epsilon)}}  \right] = \frac{1}{1-\theta\frac{\sigma_a^2+\sigma_s^2}{2}}.
\end{align*}
Note that $\frac{1}{1-\theta\frac{\sigma_a^2+\sigma_s^2}{2}}$ is the one-sided Laplace transform of an exponential random variable with mean $\frac{\sigma_a^2+\sigma_s^2}{2}$. 
Then, according to Lemma 6.1 in \cite{braverman2017heavy}, this implies that $\epsilon \overline{Q}^{(\epsilon)} $ converge in distribution to an exponential random variable with mean $\frac{\sigma_a^2+\sigma_s^2}{2}$, this completing the proof.
\end{proof}
\end{APPENDICES}

\end{document}